\documentclass[letterpaper,12pt]{article}
\newtheorem{theo}{Theorem}[section]
\newtheorem{prop}[theo]{Proposition}
\newtheorem{lem}[theo]{Lemma}
\newtheorem{defn}[theo]{Definition}

\newtheorem{cor}[theo]{Corollary}

\usepackage{thmtools}
\usepackage{thm-restate}
\newcounter{claim}[section]

\newenvironment{proof}{\noindent {\sc Proof}.}
                {\phantom{a} \hfill \framebox[2.2mm]{ } \bigskip}

\usepackage[
  top=0.9in,
  bottom=0.9in,
  left=0.9in,
  right=0.9in
]{geometry}

\usepackage{algpseudocode}
\usepackage{float}
\usepackage{algorithm}

\usepackage[colorlinks=true,linkcolor=blue,citecolor=blue]{hyperref}
\usepackage[dvips]{graphics}
\usepackage{amsmath}
\usepackage{amsfonts,enumerate}
\usepackage{pdfpages,caption,geometry}
\usepackage{subcaption}
\usepackage{tikz}
\usepackage{mathrsfs}
\usepackage[mathscr]{euscript}
\usepackage{amssymb}
\DeclareSymbolFont{rsfs}{U}{rsfs}{m}{n}
\DeclareSymbolFontAlphabet{\mathscrsfs}{rsfs}
\usepackage{mathtools}

\usepackage{hyperref}
\hypersetup{
    colorlinks=true,
    linkcolor=blue,
    filecolor=magenta,      
    urlcolor=cyan,
    pdftitle={Overleaf Example},
    pdfpagemode=FullScreen,
    }

\newcommand{\ZZ}{\mathbb{Z}}

\newcommand{\C}{\mathscrsfs{D}}
\newcommand{\F}{\mathcal{F}}

\newcommand{\W}{\mathcal{W}}
\newcommand{\D}{\mathcal{D}}

\title{A complete solution to the generalized honeymoon Oberwolfach problem with one round table}

\author{Masoomeh Akbari\thanks{Department of Mathematics, University of Ottawa, ON, Canada.}}

\begin{document}
\maketitle \baselineskip 18pt
\begin{abstract}

The generalized honeymoon Oberwolfach problem (HOP) asks whether it is possible to seat $2n$ participants consisting of $n$ newlywed couples at a conference with $s$ tables of size $2$ and $t$ ``round'' tables of sizes $2m_1, 2m_2, \ldots, 2m_t$, where  \(n = s + \sum_{i=1}^{t} m_i \) with all $m_i \geq 2$, over several nights so that each participant sits next to their spouse every time and next to each other participant exactly once. We denote this problem by $\mathrm{HOP}(2^{\langle s \rangle}, 2m_1, \ldots, 2m_t)$.

In this paper, we provide a complete solution to the generalized HOP with one {\em round} table,
showing that the obvious necessary conditions for $\mathrm{HOP}(2^{\langle s \rangle}, 2m)$ to have a solution are also sufficient.

\end{abstract}

\noindent \textbf{Keywords:} Oberwolfach Problem, Honeymoon Oberwolfach Problem, cycle decomposition, complete multigraph.

\section{Introduction}

The classic Oberwolfach problem, posed by Gerhard Ringel in 1967 at a conference in Oberwolfach, asks whether $n$ attendees can be seated at $t$ round tables of sizes $m_1, m_2, \dots, m_t$ over several meals so that all tables are full at each meal and each participant sits beside every other participant exactly once. A recent variant of the Oberwolfach problem is the honeymoon Oberwolfach problem (HOP), introduced by \v{S}ajna \cite{LDMSaj}. This problem asks whether it is possible to seat $2m_1 + 2m_2 + \dots + 2m_t = 2n$ participants, consisting of $n$ newlywed couples, at $t$ round tables of sizes $2m_1, 2m_2, \dots, 2m_t$ (where rach $m_i \geq 2$) for $2n - 2$ nights, so that each participant sits next to their spouse every night and next to every other participant exactly once. 

In graph-theoretic terms, a solution to HOP corresponds to a decomposition of $K_{2n} + (2n - 3)I$ into $2$-factors. Here, the multigraph $K_{2n} + (2n - 3)I$ is obtained from the complete graph $K_{2n}$ by adjoining $2n-3$ additional copies of a fixed $1$-factor $I$, and the $2$-factors in the decomposition are vertex-disjoint unions of cycles of lengths $2m_1, 2m_2, \ldots, 2m_t$, where within each cycle, every other edge is a copy of an edge of $I$. This problem is denoted by HOP$(2m_1, 2m_2, \ldots, 2m_t)$. In the case where $m_1 = m_2 = \ldots = m_t$ and $n = tm$, it is denoted by HOP$(2n; 2m)$. HOP has been studied by Jerade, Lepine, and \v{S}ajna, and some significant cases of it have been solved \cite{MRMSaj, LDMSaj}. 


The generalized HOP, studied in this paper, preserves the original seating conditions of HOP, except that the $2n$ participants are seated at $s$ tables of size $2$ and $t$ tables of sizes $2m_1, 2m_2, \ldots, 2m_t$, where $n = s + m_1 + \ldots + m_t$ and all $m_i \geq 2$. We denote this problem by HOP$(2^{\langle s \rangle}, 2m_1, 2m_2, \ldots, 2m_t)$, and refer to tables of size at least $4$ as \emph{round tables}.

The generalized HOP was introduced in \cite{Akbari26}. There, two significant cases were solved: the case with exactly two round tables, and the case where the sum of the sizes of the round tables is at most $20$ (both subject to additional conditions on $n$). In this paper, we focus on the generalized HOP with one round table and establish the following result.

\begin{restatable}{theo}{mainresultOne}\label{theo:main-color}
Let $s \geq 0$, $m \geq 2$, and $n = s + m$.  
Then $\operatorname{HOP}(2^{\langle s \rangle}, 2m)$ has a solution 
if and only if $m \mid 2n(n - 1)$.
\end{restatable}

This paper is organized as follows. In Section~\ref{sec:2}, we define the necessary terminology, and in Section~\ref{sec:3}, we present the supporting tools and previous results. Section~\ref{sec:10} provides an overview of our proof strategy. The subsequent sections develop the constructions and tools used in Section~\ref{sec:8} to establish our main result, Theorem~\ref{theo:main-color}.

\section{Terminology}{\label{sec:2}}

The graphs considered in this paper are loopless but may contain parallel edges or directed edges. The multigraph $\lambda G$ is obtained by replacing each edge of a simple graph $G$ with $\lambda$ parallel copies.
 As usual, \(K_n\) and \(\lambda K_n\) denote the complete graph and the \(\lambda\)-fold complete graph of order \(n\), respectively. The symbol \( K_{m[k]} \) denotes the complete multipartite graph with \( m \) parts of size \( k \), and \( K_{m,n} \) denotes the complete bipartite graph with parts of sizes \( m \) and \( n \). 
 
A collection of subgraphs $\{H_1, H_2, \ldots, H_t\}$ of a graph $G$ is said to be a \textit{decomposition} of $G$ if $\{E(H_1), E(H_2), \ldots, E(H_t)\}$ is a partition of $E(G)$. When this occurs, we write $G = H_1 \oplus H_2 \oplus \dots \oplus H_t$. If each of the subgraphs $H_1, H_2, \ldots, H_t$ is isomorphic to a graph $H$, then the collection is called an \textit{$H$-decomposition} of $G$.

A {\em $(C_{m_1}, C_{m_2},\ldots,C_{m_t})$-subgraph} of a graph $G$ is a $2$-regular subgraph consisting of $t$ disjoint cycles of lengths $m_1,m_2,\ldots,m_t$; if this subgraph is spanning, it is called a \textit{$(C_{m_1}, \dots, C_{m_t})$-factor}. A {\em $(C_{m_1}, C_{m_2},\allowbreak \ldots,C_{m_t})$-decomposition} of $G$ is a decomposition of $G$ into $(C_{m_1}, C_{m_2},\ldots,C_{m_t})$-subgraphs; if these are factors, the decomposition is called a \textit{$(C_{m_1}, \dots, C_{m_t})$-factorization}.
When $m_1=\dots=m_t=m$, we use the terms \textit{$(C_m^{\langle t \rangle})$-subgraph}, \textit{$C_m$-factor}, \textit{$(C_m^{\langle t \rangle})$-decomposition}, and \textit{$C_m$-factorization}, respectively. 
A {\em $(K_2^{\langle s \rangle},C_{m_1}, \ldots,C_{m_t})$-factor} of $G$ is a spanning subgraph whose connected components are $s$ copies of $K_2$ and a $(C_{m_1},\ldots,C_{m_t})$-subgraph of $G$. A {\em $(K_2^{\langle s \rangle},C_{m_1}, \ldots,C_{m_t})$-factorization} of $G$ is a decomposition of $G$ into {\em $(K_2^{\langle s \rangle},C_{m_1}, \ldots,C_{m_t})$-factors}.

The \emph{join} $G_1 \bowtie G_2$ of two vertex-disjoint simple graphs $G_1$ and $G_2$ is a simple graph consisting of their union together with all edges with one vertex in $G_1$ and the other in $G_2$.

The \emph{circulant graph} $\mathrm{Circ}(n; S)$, where $S \subseteq \mathbb{Z}_n^\ast$ and $S = -S$,  is a graph with vertex set $\{x_i : i \in \mathbb{Z}_n\}$ and edge set $\{x_i x_{i+d} : i \in \mathbb{Z}_n,\ d \in S\}$. An edge of the form $x_i x_{i+d}$ is said to have \emph{difference} $d$. Since each edge of difference $d$ is also of difference $n - d$, we may assume that all differences lie in the set $\{1, 2, \ldots, \lfloor \tfrac{n}{2} \rfloor\}$. When \( n \) is even, the edge \( x_i x_{i + \tfrac{n}{2}} \) connects a pair of vertices that are \emph{diametrically opposite}, and we call $d=\tfrac{n}{2}$ a \emph{diameter difference}. 
In many of the constructions given in this paper, the complete graph $K_{n}$ is viewed as $\mathrm{Circ}(n-1; \pm S) \bowtie K_1$, where $S = \{1, 2, \ldots, \lfloor\tfrac{n-1}{2} \rfloor\}$ and the vertex of $K_1$ is denoted by $x_{\infty}$. Note that an edge of the form $x_ix_{\infty}$ is said to be of {\em difference infinity}.

Let $G = \mathrm{Circ}(n; S)$ be a circulant graph with vertex set $\{x_0, x_1, \dots, x_{n-1}\}$. Following the cyclic ordering of vertex subscripts, we define the intervals: $[x_i, x_j] = \{x_i, x_{i+1}, \dots, x_j\}$, $(x_i, x_j] = \{x_{i+1}, x_{i+2}, \dots, x_j\}$, $[x_i, x_j) = \{x_i, x_{i+1}, \dots, x_{j-1}\}$, and $(x_i, x_j) = \{x_{i+1}, x_{i+2}, \dots, x_{j-1}\}$. Note that all subscripts are evaluated mod $n$. 


Let $I$ be a 1-factor in $K_{2n}$. An edge of $K_{2n}$ which belongs to $E(I)$ is called an {\em $I$-edge}; all other edges are {\em non‐I‐edges}. A graph $K_{2n}$  with  all  $I$-edges deleted is denoted by $K_{2n}- I$, and a graph $K_{2n}$  with $\lambda$ additional copies of each   $I$-edge is denoted by $K_{2n}+\lambda I$; additional copies of $I$-edges are also considered  $I$-edges. 
A cycle $C$ of $K_{2n} + \lambda I$, necessarily of even length,  is called an {\em $I$-alternating} cycle if the $I$-edges and non-$I$-edges alternate along $C$. 
Let $F$ be a  $(C_{m_1}, C_{m_2},\ldots,C_{m_t})$-subgraph of  $K_{2n} + \lambda I$. If every cycle in  $F$ is $I$-alternating, then $F$ is said to be {\em $I$-alternating}.
%
A $(K_2^{\langle s \rangle},C_{m_1}, \ldots,\allowbreak C_{m_t})$-factor of $K_{2n} + \lambda I$  is called  {\em $I$-alternating} if  its $(C_{m_1},\ldots, C_{m_t})$-subgraph is $I$‐alternating, and all other edges are $I$-edges. Moreover, a $(K_2^{\langle s \rangle},C_{m_1}, \ldots,C_{m_t})$-factorization is {\em $I$-alternating} if all of its  $(K_2^{\langle s \rangle},C_{m_1}, \ldots,C_{m_t})$-factors are $I$-alternating.

Let $I$ denote the $1$-factor of $K_{2n}$ corresponding to the $n$ couples. A solution to  $\mathrm{HOP}(2m_1, \allowbreak \dots, 2m_t)$, is equivalent to an $I$-alternating $(C_{2m_1}, \dots, C_{2m_t})$-factorization of $K_{2n} + (2n - 3)I$, where $n = \sum_{i=1}^t m_i$ \cite{LDMSaj}. Similarly, a solution to $\mathrm{HOP}(2^{\langle s \rangle}, 2m_1, \dots, 2m_t)$ is equivalent to an $I$-alternating $(K_2^{\langle s \rangle}, C_{2m_1}, \dots, C_{2m_t})$-factorization of $K_{2n} + (\gamma - 1)I$, where $n = s + \sum_{i=1}^{t} m_i$ and $\gamma = \frac{2n(n-1)}{\sum_{i=1}^{t} m_i}$ \cite{Akbari26}. Note that the obvious  necessary condition for $\mathrm{HOP}(2^{\langle s \rangle}, 2m_1, \dots, 2m_t)$ to have a solution is $\left(\sum_{i=1}^{t} m_i\right) \mid 2n(n - 1)$.


\section{Previous tools and results}{\label{sec:3}}

As in \cite{LDMSaj}, we use the symbol $4G^{\bullet}$ to denote the $4$-fold graph $G$ whose edges are coloured pink, blue, and black, with the black edges oriented so that each set of four parallel edges contains one pink edge, one blue edge, and two oppositely black arcs.

\begin{defn}{\rm{\cite{LDMSaj}}} {\label{def}}{\rm
Let $G$ be a simple graph and  $\D$  a decomposition of $4G^{\bullet}$ into $2$-regular subgraphs. We say that $\D$ is {\em HOP} if it satisfies the following condition.
\begin{description}
 \item [(C1)] For every cycle $C$ in $\D$, any two adjacent edges  of $C$ satisfy one of the following: 
\begin{itemize}
\item one is blue and the other pink; 
\item one is blue and the other black with an orientation toward the blue edge; 
\item one is pink and the other black with an orientation away from the pink edge; 
\item both are black and oriented in the same way. 
\end{itemize}
\end{description}
}
\end{defn}

The following theorem allows us to convert the generalized HOP from a problem in the multigraph $K_{2n} + (\gamma - 1)I$ to one in the multigraph $4K_n^{\bullet}$.
\begin{theo}{\rm{\cite{Akbari26, LDMSaj}}}{\label{theo:Gtool1}}
Let $s\geq 0$ and $2\leq m_1\leq \ldots\leq m_t$ be integers. Let $n=s+m_1+m_2+\ldots +m_t$. Then HOP$(2^{\langle s \rangle},2m_1, 2m_2,\ldots, 2m_t)$ has a solution if and only if $4K_n^{\bullet}$ admits an HOP $(C_{m_1}, C_{m_2},\ldots,C_{m_t})$-decomposition.
\end{theo}
We henceforth focus on finding HOP $(C_{m})$-decompositions of $4K_n^{\bullet}$ to prove the existence of solutions to HOP$(2^{\langle s \rangle}, 2m)$. 

The following results, which we will use in our constructions, are taken from previous work.

\begin{lem}{\rm{\cite {LDMSaj}}}{\label{lems-2starterQ}}
 Let $m\geq 5$ be an odd integer. Then $4K^{\bullet}_{2m}$  admits an HOP $C_m$-factorization.
\end{lem}

\begin{lem}{\rm{\cite{LDMSaj}}}\label{lems-unified}
Let $m \geq 4$ be an even integer. Then $4K^{\bullet}_{m}$ admits an HOP $C_m$-factorization.
\end{lem}
\begin{lem}{\rm{\cite {Akbari26}}}{\label{G4C}}
Let $G$ be a simple graph. If $G$ admits a decomposition into $1$-regular subgraphs of order $2t$, then $4G^{\bullet}$ admits an HOP $(C_2^{\langle t \rangle})$-decomposition.
\end{lem}

In the next lemma, the symbol $2G^{\circ}$ represents the multigraph $2G$ with a 2‐edge‐coloring with colors pink and black such that for any two adjacent vertices in $2G$, the two parallel edges between them have colors  pink and black.

\begin{lem}{\rm{\cite{LDMSaj}}}{\label{lem:Gtool2}}
Assume that $2G^{\circ}$ admits a $(C_{m_1}, C_{m_2},\ldots,C_{m_t})$-decomposition $\F$ with the property that every $m$-cycle of $\F$, for $m\geq 3$, contains an even number of pink edges. Then $4G^{\bullet}$ admits an HOP $(C_{m_1}, C_{m_2},\ldots,C_{m_t})$-decomposition.
\end{lem}

\begin{lem}{\rm{\cite{LDMSaj}}}{\label{Gtool3}}
Let $G$ be a simple graph. Let $H_1, H_2, \ldots, H_s$ be subgraphs of $G$ such that $G = H_1 \oplus H_2 \oplus \dots \oplus H_s$. If, for all $i \in \{1, \ldots, s\}$, the multigraph $4H_i^{\bullet}$ admits an HOP decomposition $\D_i$ into $(C_{m_1}, \ldots, C_{m_t})$-subgraphs, then $\D = \bigcup_{i=1}^{s} \D_i$ is an HOP decomposition of $4G^{\bullet}$ into $(C_{m_1}, \ldots, C_{m_t})$-subgraphs.
\end{lem}

\begin{lem}{\rm{\cite{LDMSaj}}}{\label{lem:Gtool4}}
Let $G$ be a simple graph. If $G$ admits a $(C_{m_1}, C_{m_2},\ldots,C_{m_t})$-decomposition, then $4G^{\bullet}$ admits an HOP  $(C_{m_1}, C_{m_2},\ldots,C_{m_t})$-decomposition.
\end{lem}

\begin{theo}{\rm{\cite {Chaudhuri}}}\label{thm:ray_wilson}
There exists a  $C_3$-factorization of $K_n$ if and only if $n \equiv 3 \pmod{6}$.
\end{theo}

\begin{theo}{\rm{\cite {BAHG,cycleIII}}}{\label{theo:Cm-dec}}
Let \( 3 \leq m \leq n \) be integers, and let \( n \) be odd. Then \( K_n \) admits a \( (C_m) \)-decomposition if and only if \( m \mid \frac{n(n-1)}{2} \).
\end{theo}

\begin{theo}{\rm{\cite {2Kn-dec}}}{\label{theo:(Cm)-dec-2Kn}}
Let $m, n\in \ZZ^+$ such that $2\leq m\leq n$. The multigraph $2K_n$ admits a $(C_m)$-decomposition if and only if \( m \mid n(n-1) \).
\end{theo}

\begin{theo}{\rm{\cite {Sotteau}}}{\label{theo-sottu}}
Let $m, r,$  and $s$ be even positive integers. The complete bipartite graph $K_{r,s}$ has a $(C_m)$-decomposition if and only if $min\{r,s\}\geq \frac{m}{2}$ and $m\mid rs$.
\end{theo}

\section{Proof strategy and overview}{{\label{sec:10}}}

By Theorem~\ref{theo:Gtool1}, $\mathrm{HOP}(2^{\langle s \rangle}, 2m)$ has a solution if and only if the multigraph $4K_n^{\bullet}$ admits an $\mathrm{HOP}(C_m)$-decomposition. Thus, the obvious necessary condition for $\mathrm{HOP}(2^{\langle s \rangle}, 2m)$ to have a solution is $m\mid |E(4K_n)|$, where  $|E(4K_n)|=2n(n-1)$. This will occur in one of the following three cases:

\begin{itemize}
    \item \textbf{Case 1:} $m$ divides $|E(K_n)|$.
    \item \textbf{Case 2:} $m$ divides $|E(2K_n)|$, but not $|E(K_n)|$.
    \item \textbf{Case 3:} $m$ divides $|E(4K_n)|$, but not $|E(2K_n)|$.
\end{itemize}
In the first case, we use the existing results on $(C_m)$-decompositions of $K_n$ and extend them to HOP $(C_m)$-decompositions for $4K_n^{\bullet}$. For the second case, we either use the results on $2K_n$ from \cite{2Kn-dec} or directly construct a $(C_m)$-decomposition of $2K_n$, which we then extend to  an HOP $(C_m)$-decomposition for $4K_n^{\bullet}$. The third case, which is the most challenging, requires us to work directly with $4K_n^{\bullet}$ to construct an HOP $(C_m)$-decomposition.

The rest of the paper is structured as follows. In Section~\ref{sec:4}, we introduce a method for recoloring edges in a cycle decomposition of $2G^\circ$. Section~\ref{sec:5} establishes the definitions and tools required for the remainder of the paper. Building on this, Section~\ref{sec:6} explains how to extend a $(C_m)$-decomposition of $2K_n$ to an $\mathrm{HOP}(C_m)$-decomposition of $4K_n^\bullet$. Following this, Section~\ref{sec:7} provides direct constructions for an $\mathrm{HOP}(C_m)$-decomposition of $4K_n^\bullet$. Finally, these results are combined in Section~\ref{sec:8} to prove the main theorem.



\section{Recoloring the edges in a cycle  decomposition of  $2G^{\circ}$}{{\label{sec:4}}}

In this section, we explain how to recolor a given $(C_m)$-decomposition of $2K_n^{\circ}$ so that every cycle contains an even number of pink edges. This method relies only on the existence of the decomposition rather than its explicit construction, however, it is applicable only for certain values of $m$ and $n$.

We begin this section with the definition of the \emph{decomposition graph}.
\begin{defn}{\label{def:color-0}}{\rm
Let $G$ be a simple graph, and $\D$ a $(C_{m})$-decomposition of $2G^{\circ}$.  The {\em{decomposition graph}} of $\D$, denoted as $\C(\D)$,  is defined as a graph with  vertex set $\D$, and two vertices in $\C(\D)$ are adjacent via $k$ edges if and only if the corresponding cycles in $\D$ share exactly $k$ pairs of parallel edges in $2G^{\circ}$.} 
\end{defn}
Observe that $\C(\D)$ is an $m$-regular graph of order $|\D|$, and hence has $\frac{|\D|\cdot m}{2}$ edges.

\begin{lem}{\rm{\cite {MSaj}}}{\label{cor:color-decGraph}}
Let $G$ be a simple graph, $m \geq 3$, and  $\D=\{F_1, F_2, \ldots, F_t\}$  a $(C_m)$-decomposition of $2G^{\circ}$. If every connected component  in $\C(\D)$ has an even number of edges, then there exists a recoloring of the edges in $\D$ so that every cycle in $\D$ has an even number of pink edges. In particular, this is possible if $m \equiv 0 \pmod{4}$.
\end{lem}

\begin{proof}
Let $H$ be a connected component of $\C(\D)$, and
$D$  the set of cycles in $\D$ corresponding to $H$. Each edge in $H$ corresponds to exactly one pair of parallel edges in $2G^{\circ}$; that is, for any cycle $C \in D$ and any edge $e \in E(C)$, there exists a cycle $C' \in D$ and $e' \in E(C')$ such that $e$ and $e'$ are parallel edges of $2G^{\circ}$. 
Each such pair contains one pink and one black edge; thus, the total number of pink edges in the cycles in $D$ is $|E(H)|$. Since $|E(H)|$ is even, the number of cycles in $D$ with an odd number of pink edges is even. 
Let $F_1$ and $F_2$ be two cycles in $D$ with an odd number of pink edges, and let $P$ be a path in $H$ connecting them. Observe that if two vertices are adjacent in $P$, then the corresponding cycles in $\D$ share at least one pair of parallel edges.
Now, for each pair of adjacent vertices in $P$, swap the colors in one pair of parallel edges in their corresponding cycles. Note that the color of exactly one edge in each of $F_1$ and $F_2$ changes, while in each cycle of $D$ corresponding to an internal vertex of $P$, the colors of two edges change. Therefore, the parity of the number of pink edges in $F_1$ and $F_2$ changes and becomes even, while the parity of the number of pink edges in all other cycles corresponding to an internal vertex of $P$ remains unchanged. Repeating this process for all pairs of cycles of $D$ with an odd number of pink edges results in  every cycle having an even number of pink edges.

Finally, since $|E(H)| = \frac{|V(H)| \cdot m}{2}$, the assumption $m \equiv 0 \pmod{4}$ ensures that $|E(H)|$ is even.
\end{proof}


\section{New definitions and tools}{\label{sec:5}}



\begin{defn}{\rm
Let $G = Circ(n-1; S) \bowtie K_1$ be a graph with vertex set $\{x_i : i \in \ZZ_{n-1}\} \cup \{x_{\infty}\}$. A cycle in $G$ is called a \emph{central cycle} if it passes through the \emph{central vertex} $x_\infty$; otherwise, it is referred to as a \emph{peripheral cycle}.
}
\end{defn}
Let $\rho=(x_{\infty})(x_0 \ x_1\ x_2 \ \ldots \ x_{n-2})$ be a permutation on the vertex set of $G$. If $C$ is a central (peripheral) cycle, then $\rho(C)$ is also a central  (peripheral) cycle. Moreover, if  $C$ generates a decomposition $\{C, \rho(C), \ldots, \rho^{i}(C)\}$, then $C$ is called a \emph{starter cycle}.

\begin{defn}{\label{def:color-1}}{\rm
Let $P=x_0 x_1 \ldots x_{n-1} x_n$ be a path. The {\em reversal} of $P$, denoted as $\stackrel{\leftarrow}{P}$, is the path obtained by reversing the order of the vertices in $P$, that is,  $\stackrel{\leftarrow}{P}=x_{n} x_{n-1} \ldots x_{1} x_0$.
} 
\end{defn}

The following lemma shows the existence of a path in a circulant graph that covers a given sequence of differences.

\begin{lem} \label{zig-zag}
Let $n, r, t, a_1, \dots, a_t$ be positive integers with $r < t$. Let $A = \{ a_1, \dots, a_r \}$ with $a_1 < a_2 < \dots < a_r \leq \lfloor \frac{n-1}{2} \rfloor$, and let $B = \{ a_{r+1}, \dots, a_t \}$ with $a_t < a_{t-1} < \dots < a_{r+1} \leq \lfloor \frac{n-2}{2} \rfloor$. 
Then $\operatorname{Circ}(n-1; \pm(A \cup B \cup \{1\}))$ admits:
\begin{enumerate}
    \item[\textbf{(i)}] a path $P = x_0 x_{v_1} \dots x_{v_t}$ covering the sequence of differences $a_1, \dots, a_t$ in this order, and
    \item[\textbf{(ii)}] for each $i \in \{1, \dots, t\}$, a path $Q = x_0 x_{v_1} \dots x_{v_{i-1}} x_{v'_i} x_{v'_{i+1}} \dots x_{v'_t}$ covering the sequence of differences $a_1, \dots, a_{i-1}, 1, a_{i+1}, \dots, a_t$ in this order.
\end{enumerate}
\end{lem}
\begin{proof}
\textbf{(i)} First, let \( d_i = a_i \) for \( i = 1, 2, \ldots, r \), and  \( d_i = (n-1) - a_i \) for \( i = r+1, r+2, \ldots, t \). Then
\[
1 \leq d_1 < \ldots < d_r < \lceil \frac{n}{2} \rceil \leq d_{r+1} < d_{r+2}< \cdots < d_t \leq n-2.
\]
Define a walk $P = x_{v_0} x_{v_1} \dots x_{v_t}$ where $v_0 = 0$ and $v_k = \sum_{j=1}^k (-1)^{j+1} d_j$, for $1\leq k\leq t$. 
By this construction, the indices of the vertices $x_{v_0}, x_{v_2}, x_{v_4}, \ldots$ (the even-indexed vertices) are strictly decreasing, while the indices of the vertices $x_{v_1}, x_{v_3}, x_{v_5}, \ldots$ (the odd-indexed vertices) are strictly increasing.

To show $P$ is a path, we prove that even-indexed and odd-indexed vertices never ``cross over''. A ``crossing'' occurs if and only if  $|v_j - v_{j-1}| \geq n-1$, for some $1\leq j\leq t$. However, $|v_j - v_{j-1}| = d_j$, and since $d_j \leq n-2$ by hypothesis, no such crossing occurs. Thus, $P$ is a path. 

\vspace{0.3cm}

\noindent\textbf{(ii)}  Now, assume $d_i$ is the term in the sequence $d_1, d_2, \ldots, d_t$ to be replaced with the difference $1$. If $i=1$, the result follows from {\bf{(i)}}. For $i > 1$, let $d^* = d_{i-1} + 1$. Then $d^* < d_{i+1}$, and for $i>2$ we have $d_{i-2} < d^*$ and if $i=2$, then $d^*\geq 2$. By {\bf{(i)}}, there exists a path $P'$ of length $t-1$ covering the sequence $d_1, \dots, d_{i-2}, d^*, d_{i+1}, \dots, d_t$, in this order. By replacing the edge of difference $d^*$ in $P'$ with two edges of differences $d_{i-1}$ and $1$, respectively, we obtain the desired path $Q$ (see Figure~\ref{fig:diff1}).
\end{proof}
\begin{figure}[H]
 \centering
   \includegraphics[width=0.7\linewidth,height=0.6\textheight,keepaspectratio]{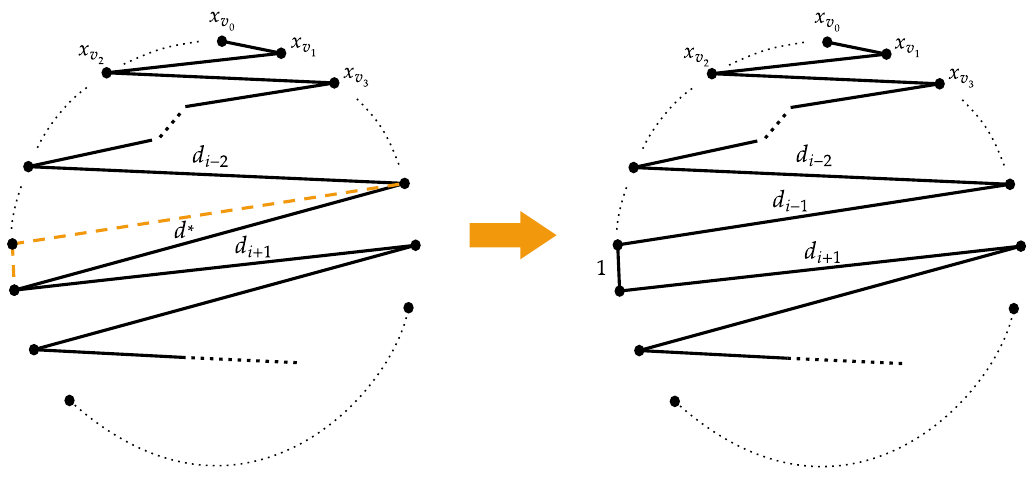}
       \caption{Replacing the edge of difference $d^*$ with two edges of differences $d_{i-1}$ and $1$.}
    \label{fig:diff1}
  \end{figure}
%
%
%
%

Lemma~\ref{lem:n-even-general} serves as a reduction step, showing that to find an  HOP $(C_m)$-decomposition of $4K_n^{\bullet}$ for even $m$ and all even $n \geq m$,  it suffices to find such a decomposition for even $n$ in the interval  $m \leq n < 2m$. The reduction step for the case of odd $n$ is stated in  Lemma~\ref{lem:n-odd-general}. These lemmas extend the approach from~\cite{2Kn-dec}, originally developed for {\em{complete symmetric digraph}} $K_n^*$, to the multigraph $4K_n^{\bullet}$.


\begin{lem} \label{lem:n-even-general}
Let \( m \geq 4\) be an even integer and let \( \alpha \in \mathbb{Z}^+ \). If \( 4K_n^{\bullet} \) admits an HOP \((C_m)\)-decomposition for all even \( n \) such that \( m \leq n < 2m \) and \( m \mid \alpha n(n-1) \), then \( 4K_n^{\bullet} \) also admits an HOP \((C_m)\)-decomposition for all even \( n \geq m \) such that \( m \mid \alpha n(n-1) \).
\end{lem}

\begin{proof}
Let $n' \geq m$ be an even integer such that $m \mid \alpha n'(n'-1)$. We write $n' = n + qm$ where $q \geq 0$ and $m \leq n < 2m$. The divisibility condition implies $m \mid \alpha n(n-1)$. By partitioning the vertex set of $4K_{n'}^{\bullet}$ into one part of size $n$ and $q$ parts of size $m$, we observe that $4K_{n'}^{\bullet}$ can be decomposed into $q$ copies of $4K_{m}^{\bullet}$, $\binom{q}{2}$ copies of $4K_{m,m}^{\bullet}$, one copy of $4K_{n}^{\bullet}$, and $q$ copies of $4K_{m,n}^{\bullet}$.

Since \(m \leq n < 2m\), by supposition, \(4K_{n}^{\bullet}\) admits an HOP \((C_m)\)-decomposition. By Lemma~\ref{lems-unified},  \(4K_{m}^{\bullet}\) admits an HOP \((C_m)\)-decomposition. Furthermore, since \(m\) and \(n\) are even, \(\min\{m, n\} \geq \frac{m}{2}\), and \(m\) divides the number of edges in both \(K_{m,m}\) and \(K_{m,n}\), by Theorem~\ref{theo-sottu}, both \(K_{m,m}\) and \(K_{m,n}\) have  \((C_m)\)-decompositions, and by Lemma~\ref{lem:Gtool4}, both \(4K_{m,m}^{\bullet}\) and \(4K_{m,n}^{\bullet}\) admit an HOP \((C_m)\)-decomposition. Finally, by applying Lemma~\ref{Gtool3}, the multigraph \(4K_{n'}^{\bullet}\) admits an HOP \((C_m)\)-decomposition.
\end{proof}


\begin{lem}{\label{lem:color-1}}
Let $m\geq 4$ be an even positive integer. Then there exists an HOP $(C_m)$-decomposition of $4K_{m+1}^{\bullet}$.
\end{lem}
\begin{proof}
Let $V(4K_{m+1}^{\bullet})=\{x_i: i\in \ZZ_{m} \}\cup \{x_{\infty}\}$, and let $\rho_{\bullet}$ be the permutation on $E(4K_{m+1}^{\bullet})$ that preserves the edge colors (and orientations), and  is induced by  the permutation  $\rho=(x_{\infty})(x_0 \ x_1\  x_2 \ \ldots \ x_{m-1})$. In \cite[Lemma 3.2]{2Kn-dec}, it is shown that $\W=\{C', C, \rho(C), \rho^{2}(C), \ldots, \allowbreak \rho^{m-1}(C)\}$ is a $(C_m)$-decomposition of $2K_{m+1}$, where
$$C= x_{\infty}\ x_0\ x_{-1}\ x_1\ x_{-2}\ x_{2}\ \ldots \ x_{-\frac{m-2}{2}}\ x_{\frac{m-2}{2}}\ x_{\infty}, \ \text{ and }\ C'=x_0\ x_1\ x_2\ \ldots\ x_{-2}\ x_{-1}\ x_0. $$ 
For \( m \equiv 2 \pmod{4} \), let $P = x_{-1}\ x_1\ x_{-2}\ x_2\ \ldots\ x_{-\frac{m-2}{4}}\ x_{\frac{m-2}{4}};$ thus, 
\[
C = x_{\infty}\ x_0\ x_{-1}\ P\ x_{\frac{m-2}{4}}\ x_{-\frac{m+2}{4}}\ \rho^{\frac{m}{2}}(\stackrel{\leftarrow}{P})\ x_{\frac{m-2}{2}}\ x_{\infty}.
\]
For \( m \equiv 0 \pmod{4} \), let $P = x_{-1}\ x_1\ x_{-2}\ x_2\ \ldots\ x_{\frac{m-4}{4}}\ x_{-\frac{m}{4}};$ thus, 
\[
C = x_{\infty}\ x_0\ x_{-1}\ P\ x_{-\frac{m}{4}}\ x_{\frac{m}{4}}\ \rho^{\frac{m}{2}}(\stackrel{\leftarrow}{P})\ x_{\frac{m-2}{2}}\ x_{\infty}.
\]
We explain the coloring for the case \( m \equiv 2 \pmod{4} \). The case \( m \equiv 0 \pmod{4} \) is analogous.

Let $F_1, F_2, F_3, F_4$ be four copies of $C$ that are colored as follows (see Figure~\ref{fig:color3}). 
\begin{figure}[h!]
 \centering
   \includegraphics[width=0.83\linewidth,height=0.83\textheight,keepaspectratio]{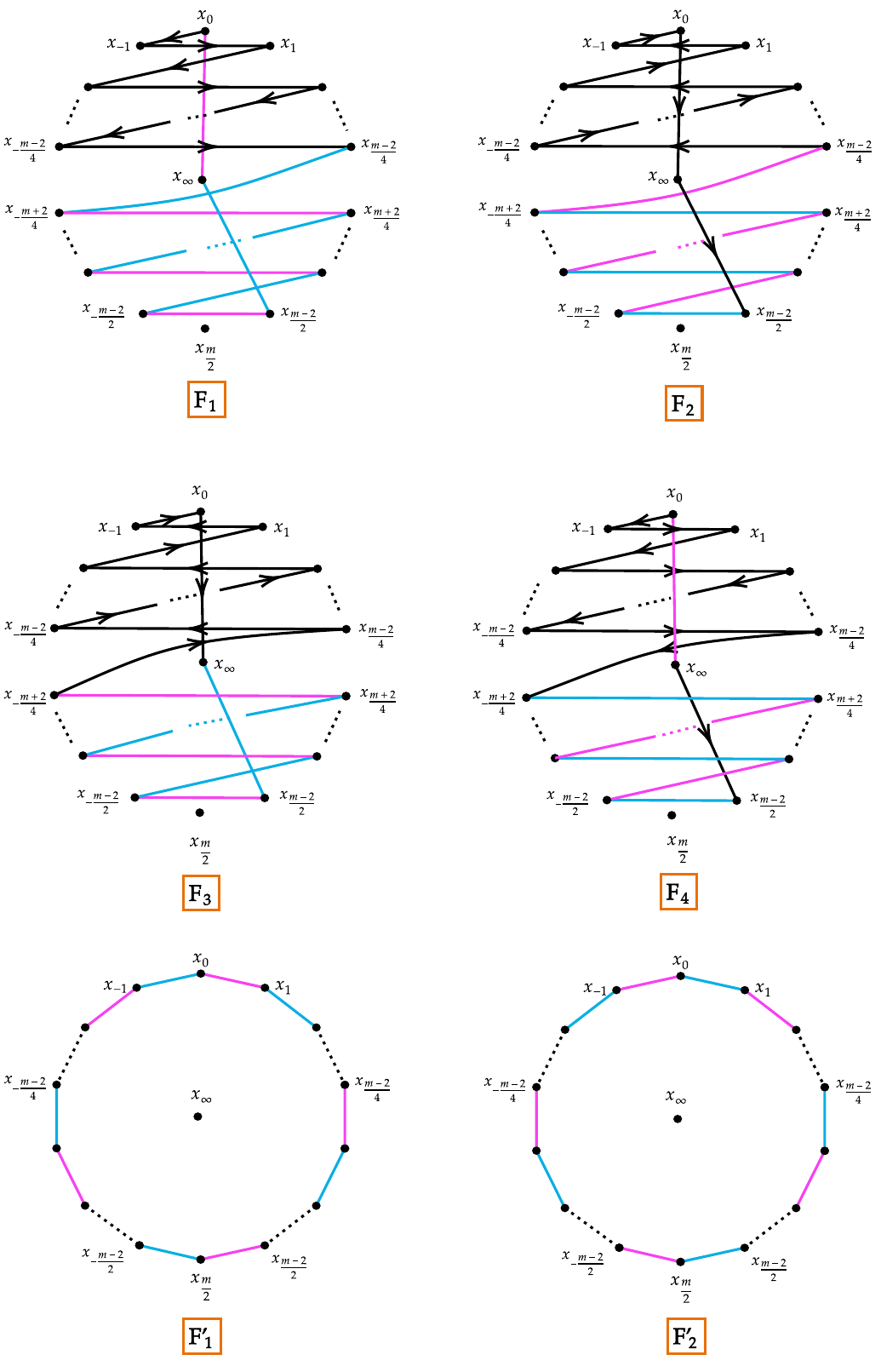}
    \caption{Starter $(C_m)$-subgraphs for an HOP decomposition of $4K_{m+1}^{\bullet}$ with $m\equiv 2\ ({\rm mod}\ 4)$.}
    \label{fig:color3}
  \end{figure}
\begin{itemize}
\item In $F_1$, color the edges of $P$ black and orient them forward, color the edge $x_{\frac{m-2}{4}} x_{-\frac{m+2}{4}}$ blue,
color the edges of $\rho^{\frac{m}{2}}(\stackrel{\leftarrow}{P})$ alternately pink and blue, beginning with pink, and since $\stackrel{\leftarrow}{P}$ has odd length, it ends with pink. Now color the edge  $x_{\frac{m-2}{2}} x_{\infty}$ blue, $x_{\infty} x_0$ pink, and color the arc  $(x_0, x_{-1})$ black. 

\item In $F_2$, color the edges of $P$ black and orient them backward, color the edge $x_{\frac{m-2}{4}} x_{-\frac{m+2}{4}}$ pink, color the edges of $\rho^{\frac{m}{2}}(\stackrel{\leftarrow}{P})$ alternately blue and pink, beginning with blue, and since $\stackrel{\leftarrow}{P}$ has  odd length, it ends with blue. Now color the arcs $(x_{-1}, x_0)$, $(x_0, x_{\infty})$, and $(x_{\infty}, x_{\frac{m-2}{2}})$  black. 

\item In $F_3$, color the edges of $P$ black and orient them backward, color the arc $(x_{-\frac{m+2}{4}}, x_{\frac{m-2}{4}})$ black, color the edges of $\rho^{\frac{m}{2}}(\stackrel{\leftarrow}{P})$ alternately pink and blue, beginning with pink, and since $\stackrel{\leftarrow}{P}$ has  odd length, it ends with pink. Now color the edge  $x_{\frac{m-2}{2}} x_{\infty}$ blue and the arcs $(x_{-1}, x_0)$ and $(x_0, x_{\infty})$ black. 

\item In $F_4$, color the edges of $P$ black and orient them forward, color the arc $(x_{\frac{m-2}{4}}, x_{-\frac{m+2}{4}})$ black, color the edges of $\rho^{\frac{m}{2}}(\stackrel{\leftarrow}{P})$ alternately blue and pink, beginning with blue, and since $\stackrel{\leftarrow}{P}$ has  odd length, it ends with blue. Then color the arcs  $(x_{\infty}, x_{\frac{m-2}{2}})$, $(x_0, x_{-1})$ black and the edge $x_{\infty} x_0$ pink. 
\end{itemize}
Let $F_1', F_2'$ be two copies of $C'$. In $F_1'$,  color the edges alternately pink and blue, begining with pink. In $F_2'$,  color the edges alternately blue and pink, begining with blue. 
We claim that  
$$\W'=\big\{\rho_{\bullet}^i(F_1), \rho_{\bullet}^i(F_2), \rho_{\bullet}^{\frac{m}{2}+i}(F_3), \rho_{\bullet}^{\frac{m}{2}+i}(F_4): i=0,1,\ldots,\textstyle{\frac{m-2}{2}} \big\}\cup\big\{F_1', F_2'\big\}$$
 is a $(C_m)$-decomposition  for $4K_{m+1}^{\bullet}$. 
 Notice that $F_1, F_2, F_3, F_4$ jointly contain exactly one edge from each orbit of $\langle \rho_{\bullet} \rangle$ corresponding to difference $\frac{m}{2}$, and two edges from each orbit of $\langle \rho_{\bullet} \rangle$ corresponding to each difference $d\in \{2,3,\ldots, \frac{m-2}{2}\}$, namely, a pair of edges of the form $(e, \rho_{\bullet}^{\frac{m}{2}}(e))$. The black orbits corresponding to difference $1$ are covered by $F_1, F_2, F_3, F_4$, while the pink and the blue orbits are completely covered by $F_1', F_2'$. 
Since $F_1, F_2, F_3, F_4, F_1', F_2'$ satisfy Condition {\bf (C1)} of Definition {\rm{\ref{def}}}, and $\rho_{\bullet}$ preserves the edge colors (and orientations), it can be verified that $\W'$ is  an HOP $(C_m)$-decomposition of  $4K_{m+1}^{\bullet}$.
\end{proof}

\begin{lem} \label{lem:n-odd-general}
Let \( m\geq 4 \) be an even integer, and let \(\alpha \in \mathbb{Z}^+\). If \(4K_n^{\bullet}\) admits an HOP $(C_m)$-decomposition for all odd \(n\) such that \(m < n < 2m\) and \(m \mid \alpha n(n-1)\), then \(4K_n^{\bullet}\) also admits an HOP $(C_m)$-decomposition for all odd \(n > m\) such that \(m \mid \alpha n(n-1)\).
\end{lem}

\begin{proof}
Take any odd \(n > m\) such that \(m \mid \alpha n(n-1)\). We can write \(n = qm + r + 1\), for integers \(q \geq 1\) and even \(0 \leq r < m - 1\). Then also  \(n=(q - 1)m + (m + r) + 1\). Label one vertex of $4K_n^{\bullet}$ as $x_{\infty}$ and partition the remaining vertices into $q - 1$ sets of size $m$ and one set of size $m + r$. Observe that $4K_n^{\bullet}$ decomposes into $q - 1$ copies of $4K_{m+1}^{\bullet}$, $\binom{q - 1}{2}$ copies of $4K_{m,m}^{\bullet}$, one copy of $4K_{m+r+1}^{\bullet}$, and $q - 1$ copies of $4K_{m,m+r}^{\bullet}$.

Since \(m\) divides \(\alpha n(n - 1)\) and \(n = qm + r + 1\), it follows that \(m\) divides \(\alpha r(r + 1)\). Consequently, \(m\) divides \(\alpha (m + r + 1)(m + r)\). Additionally, considering that \(0 \leq r < m - 1\), we have \(m < m + 1 \leq m + r + 1 < 2m\). Therefore, by the given hypothesis, \(4K_{m+r+1}^{\bullet}\) admits an HOP $(C_m)$-decomposition. 

Also, in Lemma~\ref{lem:color-1}, we proved that \(4K_{m+1}^{\bullet}\) admits an HOP $(C_m)$-decomposition.
 Since \(m\) and \(r\) are even, \(\min\{m, m + r\} \geq \frac{m}{2}\), and \(m\) divides the number of edges in both \(K_{m,m}\) and \(K_{m,m+r}\), by Theorem~\ref{theo-sottu}, both \(K_{m,m}\) and \(K_{m,m+r}\) have a $(C_m)$-decomposition, and thus by Lemma~\ref{lem:Gtool4},  both \(4K_{m,m}^{\bullet}\) and \(4K_{m,m+r}^{\bullet}\) admit an HOP $(C_m)$-decomposition. Therefore, by Lemma~\ref{Gtool3}, the multigraph \(4K_n^{\bullet}\) admits an HOP $(C_m)$-decomposition.
\end{proof}

\section{Extending a $(C_m)$-decomposition of $2K_n$  to an HOP $(C_m)$-decomposition of $4K_n^{\bullet}$}{\label{sec:6}}

In this section, we focus on constructing a \( (C_m) \)-decomposition of \( 2K_n \) and then extending it to an HOP \( (C_m) \)-decomposition of \( 4K_n^{\bullet} \). To find a decomposition for \( 2K_n \), we use different constructions depending on the parity of \( m \) and \( n \). 


\subsection{The case when \( m \) is even and \( n \) is even}

Since the case \( m \equiv 0 \pmod{4} \) can be handled using Lemma~\ref{cor:color-decGraph}, we focus on the case where \( m \equiv 2 \pmod{4} \).


\begin{lem}{\label{lem:color-m-even-n-even}}
Let $m\equiv 2\ ({\rm mod}\ 4)$, and let $n$ be an even positive integer with $6\leq m\leq n$. If $n(n-1)\equiv 0\ ({\rm mod}\ m)$, then $4K_n^{\bullet}$ admits an HOP $(C_m)$-decomposition. 
\end{lem}
\begin{proof}
By Lemma~\ref{lem:n-even-general}, it suffices to prove this result for even $n$ in the range $m\leq n<2m$. Note that if  $n=m$, then by Lemma~\ref{lems-unified}, there exists an HOP $(C_m)$-decomposition of $4K_{m}^{\bullet}$. Hence, we assume that $m< n<2m$.

The approach is to construct starter central and starter peripheral cycles and to color them appropriately. These starter cycles are designed to generate all cycles in the decomposition through suitable permutations.
We begin by defining the necessary parameters and determining how many starter central and starter peripheral cycles are needed. 

Let \( V(2K_n^{\circ}) = \{x_i : i \in \mathbb{Z}_{n-1}\} \cup \{x_{\infty}\} \), and let $\rho_{\circ}$ be the permutation on $E(2K_n^{\circ})$ that preserves the color of the edges, and is induced by  the permutation  $\rho=(x_{\infty})(x_0 \ x_1\  x_2 \ \ldots \ x_{n-2})$. Observe that the group $\langle \rho_{\circ} \rangle$ has the following orbits on the edge set of $2K_n^{\circ}$:
\begin{itemize}
\item for each $s\in \{1,2,\ldots, \textstyle{\frac{n-2}{2}}\}$, we have a pink and a black orbit $\{x_ix_{i+s}: i\in \ZZ_{n-1} \}$; and
\item a pink and a black orbit  $\{x_ix_{\infty}: i\in \ZZ_{n-1} \}$.
\end{itemize}
For convenience, let \( S = \{ 1,  2, \ldots,  \frac{n-2}{2}, \infty\} \) be the set of all differences. 
Let \( n = 2^e a \), where \( a \) is odd, and let \( m = 2a' b' \), where \( a' \) and \( b' \) are both odd, with \( a' \mid a \) and \( b' \mid (n-1) \).
To construct the peripheral cycles, we partition the \( n-1 \) vertices in \( \{x_i : i \in \mathbb{Z}_{n-1}\} \) into \( b' \) segments, each containing \( \ell = \frac{n-1}{b'} \) consecutive vertices, and contributing \( 2a' \) edges toward an \( m \)-cycle.
%

If \( 2K_n^{\circ} \) is \((C_m)\)-decomposable, then the number of \( m \)-cycles in the decomposition is
$$ \frac{n(n-1)}{m} = \frac{(n-1)(m + n - m)}{m} = (n-1) + \frac{n-1}{b'} \cdot \frac{n-m}{2a'} = (n-1) + \ell F,$$
where $F=\frac{n - m}{2a'} =  \frac{2^{e-1} a - a' b' }{a'}$ is an integer.
The number of  \( m \)-cycles suggests constructing one starter central cycle and \( F \) starter peripheral cycles. The starter central cycle is rotated through all \( n-1 \) positions, while each of the \( F \) starter peripheral cycles is rotated through  \( \ell \) positions to generate all the cycles in the decomposition. 

Next, we establish some inequalities that will be useful later in the proof. Since $n < 2m$, it follows that $2^e a < 4a' b'$ and $\frac{2^{e-1}a}{a'} < 2b'$, which implies $F = \frac{2^{e-1}a - a'b'}{a'} < b'$. Since $F$ is an integer, we have
\begin{equation}
    F \leq b' - 1. \label{eq:F<b'-1}
\end{equation}
Furthermore, the condition $m < n - 1 < 2m$ implies $2a'b' < b'\ell < 4a'b'$, which simplifies to
\begin{equation}
    2a' < \ell < 4a'. \label{eq:ell4a'}
\end{equation}
If $\ell - a' \leq 2a' - 1$, then $\ell \leq 3a' - 1$. Substituting $n = b'\ell + 1$, we obtain $n \leq 3a'b' - b' + 1$. Since $n = 2^e a$, it follows $ \frac{2^{e-1}a}{a'}  \leq b'+\frac{a'b'-b'+1}{2a'}$ and thus $F= \frac{2^{e-1}a}{a'} -b' \leq \frac{b'}{2}-\frac{b'-1}{2a'}$.  
Given $b' \geq 1$, we have $F \leq \frac{b'}{2}$. Therefore, 
\begin{equation}
\text{if } \ell - a' \leq 2a' - 1, \text{ then } F \leq \frac{b'}{2}. \label{eq:333b'}
\end{equation}
Having established the necessary inequalities, we proceed with the constructions.

\vspace{0.3cm}

\noindent \textbf{Case 1:} Assume $m \mid \frac{n(n-1)}{2}$. Since $n = 2^e a$, 
it follows that $e \geq 2$, implying $n \equiv 0 \pmod{4}$. Moreover, as $a$, $a'$, and $b'$ are odd, $F = \frac{2^{e-1} a - a' b'}{a'}$ is also odd.
%

%

We start by constructing the starter peripheral cycles. For \( i = 0, 1, \ldots, \frac{F-1}{2} \), define 
$$P_i=x_0\ x_{i\ell+1}\ x_{-1}\ x_{i\ell+2}\ x_{-2} \dots x_{i\ell+a'-1}\ x_{-(a'-1)}\ x_{i\ell+a'}\ x_{\frac{b'+1}{2}\ell}.$$
The differences covered by \( P_0 \) are:
\begin{equation}
1, \; 2, \;  3, \; 4, \; \ldots, \;  2a' - 2, \;  2a' - 1, \; \textstyle{\frac{(b' -  1)\ell}{2} + a'}, \label{eq:1-secondcase-P0}
\end{equation}
and the differences covered by \( P_i \), for \( i = 1, \ldots, \frac{F-1}{2} \), are:
\begin{equation}
i\ell + 1, \; i\ell + 2, \; i\ell + 3, \; \ldots, \; i\ell + 2a' - 2, \; i\ell + 2a' - 1, \; \textstyle{\frac{(b' - 2i + 1)\ell}{2} - a'}. \label{eq:1-secondcase}
\end{equation}

Since $2a'<\ell$ by (\ref{eq:ell4a'}), the paths \(\rho^{j}(P_i)\), for \(j = 0, \ell, 2\ell, \ldots, (b' - 1)\ell\),  are pairwise vertex-disjoint except at the endpoints. Also, since \(b'\) is odd, we have \(\gcd\left(b', \frac{b' + 1}{2}\right) = 1\). Thus,
$
C_i=P_i \cup \rho^{\ell}(P_i) \cup \ldots \cup \rho^{(b'-1)\ell}(P_i)
$
is an \(m\)-cycle.

We show that the differences covered by the $m$-cycles in (\ref{eq:1-secondcase-P0})  and (\ref{eq:1-secondcase}) are pairwise distinct.
First, we show that the differences \( i\ell + 1, i\ell + 2, \ldots, i\ell + 2a' - 1 \) are pairwise distinct for \( i = 0, 1, \ldots, \frac{F - 1}{2} \). It suffices to show that the largest difference, $\frac{F-1}{2}\ell + 2a' - 1$,  does not exceed \( \frac{n-2}{2} \). Since $F$ and $b'$ are odd, \eqref{eq:F<b'-1} implies $F \leq b' - 2$. Therefore, \(\frac{F - 1}{2} \ell \leq \frac{b' - 3}{2} \ell\), and  combining this with $2a' < \ell$ from \eqref{eq:ell4a'}, we obtain
\allowdisplaybreaks
\begin{align*}
\frac{F - 1}{2} \ell + 2a' - 1 &\leq \frac{b' - 3}{2}\ell +\ell -1 < \frac{(b' - 1)\ell}{2} + a'   <\frac{b'\ell}{2}= \frac{n - 1}{2} 
\end{align*}
In the above inequality, observe that
\(
\frac{F - 1}{2}\ell + 2a' - 1 < \frac{(b' - 1)\ell}{2} + a'< \frac{n - 1}{2}.
\)
This guarantees that for all \( j = 0, 1, \ldots, \frac{F-1}{2} \), the difference 
\(
\frac{(b' - 1)\ell}{2} + a' \notin \bigl\{j\ell + 1, j\ell + 2, \ldots, j\ell + 2a' - 1\bigr\}.
\)
Next, we show that for all \( i = 1, \ldots, \frac{F-1}{2} \) and \( j = 0, 1, \ldots, \frac{F-1}{2} \),
\[
\frac{(b' - 2i + 1)\ell}{2} - a' \notin \bigl\{j\ell + 1, j\ell + 2, \ldots, j\ell + 2a' - 1\bigr\}.
\]
Observe that 
\[
\textstyle{0<\frac{(b' - 2(\frac{F-1}{2}) + 1)\ell}{2} - a'< \frac{(b' - 2(\frac{F-3}{2}) + 1)\ell}{2} - a'<\ldots<\frac{(b' - 2\cdot 1 + 1)\ell}{2} - a'<\frac{(b' - 1)\ell}{2} + a'.}
\]
Thus, it suffices to show that for all \( j = 0, 1, \ldots, \frac{F-1}{2} \) and all \( \alpha = 1, \ldots, 2a'-1 \),
\[
\frac{(b' - 2\left(\frac{F-1}{2}\right) + 1)\ell}{2} - a' \neq j\ell + \alpha.
\]
We prove this by contradiction.
Suppose to the contrary that $\frac{b' -F}{2}\ell + \ell- a' = j\ell + \alpha$.
Since \( F \) and \( b' \) are odd, the term \( \frac{b' - F}{2} \) is an integer, and since \( 2a' < \ell \) we have  
\(
\frac{b' - F}{2} = j \ \text{and} \ \ell - a' = \alpha .
\)
Given that \( j \leq \frac{F - 1}{2} \), it follows that \( F \geq \frac{b' + 1}{2} \).
On the other hand, knowing that \(1 \leq \alpha \leq 2 a' - 1\) and \(\ell - a' = \alpha\), it follows that \(\ell - a' \leq 2 a' - 1\), and by  (\ref{eq:333b'}), we have \(F\leq \frac{b'}{2}\).
 Therefore, we have \( \frac{b' + 1}{2}\leq F \leq \frac{b'}{2}\), which is a contradiction. 
We conclude that the differences listed in (\ref{eq:1-secondcase-P0}) and (\ref{eq:1-secondcase}) are pairwise distinct.

Let \( B \) be the set of differences covered by the \( m \)-cycles \( C_1, C_2, \ldots, C_{\frac{F-1}{2}} \) as listed in (\ref{eq:1-secondcase}), and let \( G_1 = \text{Circ}(n-1; \pm B) \). 
Then, \(\{C_i, \rho(C_i), \ldots, \rho^{\ell-1}(C_i) \mid i=1, \ldots, \tfrac{F - 1}{2}\}\) is a \((C_m)\)-decomposition of \( G_1 \). Take two copies of each \( C_i \), color one copy pink and the other black. By doing so, we generate \( F - 1 \) starter peripheral cycles, which we denote \( C'_1, \ldots, C'_{F-1} \). It is now easy to see that \(\{C'_i, \rho_{\circ}(C'_i), \ldots, \rho_{\circ}^{\ell-1}(C'_i) \mid i=1, \ldots, F-1\}\) is a \((C_m)\)-decomposition of \(2G_1^{\circ}\). Moreover, since \(m\) is even, each cycle contains an even number of pink edges. Thus, by Lemma~\ref{lem:Gtool2}, there exists an HOP \((C_m)\)-decomposition of \(4G_1^{\bullet}\).


%

Recall \( A = \{1, 2, \ldots, 2a' - 2, 2a' - 1, \tfrac{(b' - 1)\ell}{2} + a'\} \) is the set of differences covered by \( C_0 \). Color the edges of $C_0$  pink. Observe that the family of peripheral cycles generated by \( C_0 \); that is, 
\(
\{C_0, \rho_{\circ}(C_0), \ldots, \rho_{\circ}^{\ell-1}(C_0)\},
\)
covers the pink orbits corresponding to the differences in \( A \).
The differences in \( A \) will also appear in the central cycle \( C \), to be constructed below. There, we need to ensure that the black orbits corresponding to the differences listed in \( A \) are covered.

Now, we focus on constructing the starter central cycle  \(C\). Let \(\mathscrsfs{L} \) be the set of unused differences in \(S\), that is, \(\mathscrsfs{L}  = S \setminus (A \cup B)=\{a_1, a_2, \ldots, a_t\} \cup \{\infty\}\), where \(a_1=2a' < a_2 < \ldots < a_t \leq \frac{n-2}{2}\) are  integers and \(t = \frac{n-2}{2} -(F+1)a'\). Note that $2a' - 1 < a_1< \tfrac{(b' - 1)\ell}{2} + a' \leq \tfrac{n - 2}{2}.$


Consider the following sets of differences: 
\begin{itemize}
\item \( W_1=\{ 1, 2, \ldots, 2 a' - 1, a_1, a_2,  \ldots, \textstyle{\frac{(b' - 1)\ell}{2} +a'}, \ldots, a_{t-1}, a_t\}\)
\item \( W_2=\{ a_1, a_2,  \ldots, a_{t-1}, a_t\} \)
\end{itemize}
Note that \( \tfrac{(b' - 1)\ell}{2} + a' \notin W_2 \), and  in the sequence \( a_1, a_2, \ldots, \tfrac{(b' - 1)\ell}{2} + a', \ldots, a_{t-1}, a_t \), we assume
\(
a_1 < a_2 < \ldots < \tfrac{(b' - 1)\ell}{2} + a' < \ldots < a_{t-1} < a_t,
\)
with the understanding that
\(
a_1 < a_2 < \ldots < a_{t-1} < a_t < \tfrac{(b' - 1)\ell}{2} + a'
\)
is also possible.
Here, \( W_1 \) contains all the differences in \( A \) and  \( \mathscrsfs{L} \) except for \( \infty \), and has size \( |W_1| = t + 2a' \). The set \( W_2 \) contains all the differences in \( \mathscrsfs{L} \) except for \( \infty \), and has size \( |W_2| = t \).  By Lemma~\ref{zig-zag}(i), there exists a path \( T \) of length \( 2t + 2a' \) that starts at \( x_0 \) and covers the differences in the following order:
\[T: 1, 2, \ldots,  2 a' - 1, a_1, a_2,  \ldots, \textstyle{\frac{(b' - 1)\ell}{2} +a'}, \ldots, a_{t-1}, a_t, a_t, a_{t-1}, \ldots, a_2, a_1.\]
Note that the length of the path \( T\) is \( 2t + 2a'=2\left(\frac{n-2}{2} - (F+1) a'\right) + 2a'  =   m - 2. \)

Let \( C = x_{\infty} T x_{\infty} \). We color the edges corresponding to the second occurrences of the differences \( \infty, a_1, a_2, \ldots, a_{t-1}, a_t \) pink, and color all remaining edges black, as shown below.
\begin{figure}[H]
 \centering
   \includegraphics[width=0.95\linewidth,height=0.8\textheight,keepaspectratio]{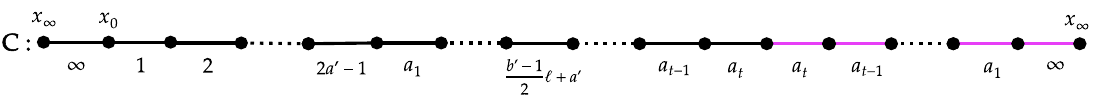}
   \vspace{-10pt}
\end{figure}
Given \( t = \frac{n-2}{2} - (F + 1)a' \), where \( n \equiv 0 \pmod{4} \) and both \( F \) and \( a' \) are odd, we see that \( t \) is odd. Hence, \( t+1 \), the number of pink edges in \( C \), is even.
Observe that \( C \) contains exactly one pink edge and one black edge from each orbit of \( \langle \rho_{\circ} \rangle \) corresponding to the differences in \( \mathscrsfs{L} \). Additionally, it contains exactly one black edge from each orbit of \( \langle \rho_{\circ} \rangle \) corresponding to the differences in \( A \).
Thus,
\(
\{C, \rho_{\circ}(C), \ldots,  \rho_{\circ}^{n-2}(C)\} \cup \{C_0, \rho_{\circ}(C_0), \ldots, \rho_{\circ}^{\ell-1}(C_0)\}
\)
is a \((C_m)\)-decomposition for \( 2G_2^{\circ} \), where
\(
G_2 = \text{Circ}(n-1; \pm(A \cup (\mathscrsfs{L} - \{\infty\}))) \bowtie K_1.
\)
Since each cycle in this decomposition has an even number of pink edges,  by Lemma~\ref{lem:Gtool2}, there exists an HOP \((C_m)\)-decomposition of \(4G_2^{\bullet}\).

We have $4K_n^{\bullet}=4G_1^{\bullet}\oplus 4G_2^{\bullet}$, and since each of $4G_1^{\bullet}$ and $4G_2^{\bullet}$ admits an HOP $(C_m)$-decomposition, by Lemma~\ref{Gtool3}, the multigraph $4K_n^{\bullet}$ admits an HOP $(C_m)$-decomposition.


\vspace{0.3cm}


\noindent \textbf{Case 2:} Assume $m \nmid \frac{n(n-1)}{2}$, that is, $2a'b' \nmid 2^{e-1} a (n-1)$. Since $b' \mid (n-1)$ and $a' \mid a$, it follows that $2 \nmid 2^{e-1}$; thus $e = 1$, $n = 2a$, and $n \equiv 2 \pmod{4}$.

Since \( e = 1 \) and \( a' \), \( b' \), and \( a \) are all odd, it follows that \( F = \frac{2^{e-1} a - a' b'}{a'} \) is even.
We begin by constructing \(\frac{F}{2}\) starter peripheral cycles for \( K_n \). 
For \( i = 0, 1, \ldots, \frac{F-2}{2} \), we define $P_i$ exactly as in Case 1. The differences covered by $P_i$ are given in \ref{eq:1-secondcase-P0} and \ref{eq:1-secondcase}.

As in Case~1, it can be shown that $C_i = P_i \cup \rho^{\ell}(P_i) \cup \dots \cup \rho^{(b'-1)\ell}(P_i)$
is an $m$-cycle, and that the differences covered by $C_0, C_1, \dots, C_{\frac{F-2}{2}}$, as defined in (\ref{eq:1-secondcase-P0}) and (\ref{eq:1-secondcase}), are pairwise distinct. Let \( B \) be the set of differences covered by these \( m \)-cycless, and let \( G_1 = \text{Circ}(n-1; \pm B) \). 
As in Case~1, it can be shown that there exists an HOP \((C_m)\)-decomposition of \(4G_1^{\bullet}\).

Now, we focus on constructing the starter central cycle  \(C\). Let \(\mathscrsfs{L} \) be the set of unused differences in \(S\), that is, \(\mathscrsfs{L}  = S \setminus B=\{a_1, a_2, \ldots, a_t\} \cup \{\infty\}\), where \(a_1 < a_2 < \ldots < a_t \leq \frac{n-2}{2}\) are  integers and \(t = \frac{n-2}{2} -Fa'\). 
Consider two sets of differences  \( W_1=W_2=\{ a_1, a_2, \ldots, a_{t-1}, a_t\}\),  where
$a_1 < a_2 < \ldots  < a_{t-1} < a_t\leq \tfrac{n-2}{2}.$
Here, \(|W_1|=|W_2|  = t\).  By Lemma~\ref{zig-zag}(i), there exists a path \( T \) of length \( 2t = m - 2\) that starts at \( x_0 \) and covers the differences in the following order:
\[T: a_1, a_2,  \ldots, a_{t-1}, a_t, a_t, a_{t-1}, \ldots, a_2, a_1.\]
Let \( C = x_{\infty} T x_{\infty} \). Observe that \( C \) covers each difference in \(\mathscrsfs{L} \) exactly twice. Let \( C' \) be another copy of \( C \).  
In \( C \), color the edges corresponding to the second occurrences of the differences \( a_1, a_2, \ldots,\allowbreak a_t \) alternately pink and blue, starting with \( a_t \) as pink. Since \( n \equiv 2 \pmod{4} \), \( F \) is even, and \( a' \) is odd, it follows that \( t = \frac{n-2}{2} - F a' \) is even, ensuring that the edge corresponding to the difference \( a_1 \) is colored blue.  
In \( C' \), color the same edges alternately pink and blue, but in the opposite way of \( C \), and complete the coloring as shown below.
\begin{figure}[H]
 \centering
   \includegraphics[width=0.85\linewidth,height=0.8\textheight,keepaspectratio]{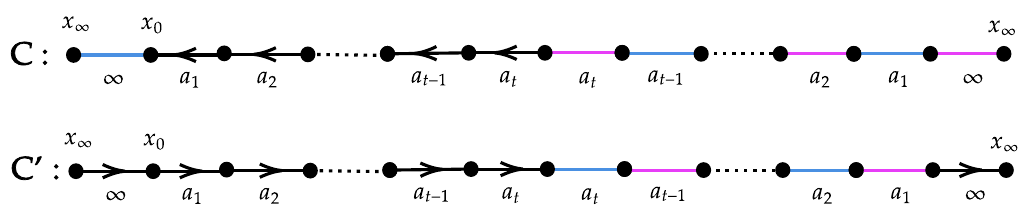}
   \vspace{-10pt}
\end{figure}
Observe that \( C \) and \( C' \) jointly contain exactly one blue edge, one pink edge, and two black arcs corresponding to the differences in \(\mathscrsfs{L} \). Additionally, the two cycles satisfy Condition {\bf (C1)} of Definition {\rm{\ref{def}}}. Therefore,  
\(
\{\rho_{\bullet}^i(C), \rho_{\bullet}^i(C') \mid i = 0, 1, \ldots, n-2\}
\)
is an HOP \((C_m)\)-decomposition of \( 4G_2^{\bullet} \), where  
\(
G_2 = \text{Circ}(n-1; \pm(\mathscrsfs{L} - \{\infty\})) \bowtie K_1.
\)

Since $4K_n^{\bullet}=4G_1^{\bullet}\oplus 4G_2^{\bullet}$ and each of $4G_1^{\bullet}$ and $4G_2^{\bullet}$ admits an HOP $(C_m)$-decomposition, so does $4K_n^{\bullet}$ by Lemma~\ref{Gtool3}.
 \end{proof}

\subsection{The case when \( m \) is even and \( n \) is odd}
In this section, we review the $(C_m)$-decompositions of $2K_n$ from  \cite{2Kn-dec} and extend them to HOP $(C_m)$-decompositions of $4K_n^{\bullet}$.  We focus on even $m$ and odd $n$; specifically, the cases $n(n-1) \equiv m \pmod{2m}$  and $m \equiv 2 \pmod{4}$. The case $n(n-1) \equiv 0 \pmod{2m}$ can be covered by  $(C_m)$-decompositions of $K_n$, and the case $m \equiv 0 \pmod{4}$ can be  addressed by Lemma~\ref{cor:color-decGraph}.


\begin{lem}{\label{lem:color-3}}
Let $m=2k$ with $k$ being an odd  positive integer. Let $n$ be an odd  integer with $6\leq m<n<2m$, and $n(n-1)\equiv m\ ({\rm mod}\ 2m)$. Let $G=Circ(n; \pm S)$, where $S=\{a_1, a_2,\ldots, a_k\}$ and $a_1, a_2,\ldots, a_k$ are positive integers with $a_1<a_2<\ldots<a_k<\frac{n}{2}.$ Then $4G^{\bullet}$ admits an HOP $(C_m)$-decomposition.
\end{lem}
\begin{proof}
Let $V(4G^{\bullet})=\{x_i: i\in \ZZ_{n} \}$. Let $\rho_{\bullet}$ be the permutation on $E(4G^{\bullet})$ that preserves the edge colors (and orientations), and  is induced by  the permutation  $\rho=(x_0 \ x_1\  x_2 \ \ldots \ x_{n-1})$. 
For $i\in\{1,2,\ldots, k-1\}$, let $\ell_i = \sum_{j=1}^{i} (-1)^{j-1} a_j$.  Then $P=x_0\ x_{\ell_1}\ x_{\ell_2}\ x_{\ell_3}\ \ldots\ x_{\ell_{k-1}}$ is a path of length $k-1$ in $G$.  Now let 
$C= P\ x_{\ell_{k-1}}x_{\ell_{k-1}-a_k} \ \rho^{n-a_k}(\stackrel{\leftarrow}{P})\ x_{-a_k} x_{0}.$
Observe that $C$ is an $m$-cycle that traverses the following differences (in order):
$$a_1, a_2, a_3, \ldots, a_{k-1}, a_{k}, a_{k-1}, a_{k-2}, \ldots, a_3, a_2, a_1, a_k .$$
In \cite[Lemma 3.4]{2Kn-dec}, it is shown that $\W=\{C, \rho(C), \rho^{2}(C), \ldots, \rho^{n-1}(C)\}$ is a $(C_m)$-decomposition for $2G$. 
Let $C'$ be a copy of $C$. Color $C$ and $C'$ as follows. (see  Figure~\ref{fig:color41})
\begin{itemize}
\item In $C$, color the edge $x_{\ell_{k-1}}x_{\ell_{k-1}-a_k}$ blue,  and color the first edge in $\rho^{n-a_k}(\stackrel{\leftarrow}{P})$, which is  of difference $a_{k-1}$,  pink. Color the rest of the edges black and orient them away from the pink edge and towards the blue edge. 

\item In $C'$, color the edges in the path $P$ alternately pink and blue, starting with pink. Since $P$ is of length $k-1$, which is even, it ends with blue. Now, color edge $x_{\ell_{k-1}}x_{\ell_{k-1}-a_k}$ pink, and color the first edge in $\rho^{n-a_k}(\stackrel{\leftarrow}{P})$, which is of difference $a_{k-1}$, black and orient it away from the pink edge. Color the rest of the edges in $\rho^{n-a_k}(\stackrel{\leftarrow}{P})$ alternately blue and pink, start with blue, and since $k-2$ is odd, it ends with blue. Finally, color the remaining  edge $x_{-a_k} x_{0}$ black and orient it away from the pink edge and towards the blue edge. 
\end{itemize}
Observe that  $C$ and $C'$ satisfy Condition {\bf (C1)} of Definition {\rm{\ref{def}}}, and they jointly contain exactly one edge from each orbit of $\langle \rho_{\bullet}\rangle$ corresponding to differences $ a_1,  a_2, \ldots,  a_{k}$. Thus, 
$\W'=\big\{\rho_{\bullet}^i(C), \rho_{\bullet}^i(C'): i=0,1,\ldots,n-1\big\}$
 is an HOP $(C_m)$-decomposition  for $4G^{\bullet}$.  
\end{proof}
\begin{figure}[h!]
 \centering
   \includegraphics[width=0.8\linewidth,height=0.8\textheight,keepaspectratio]{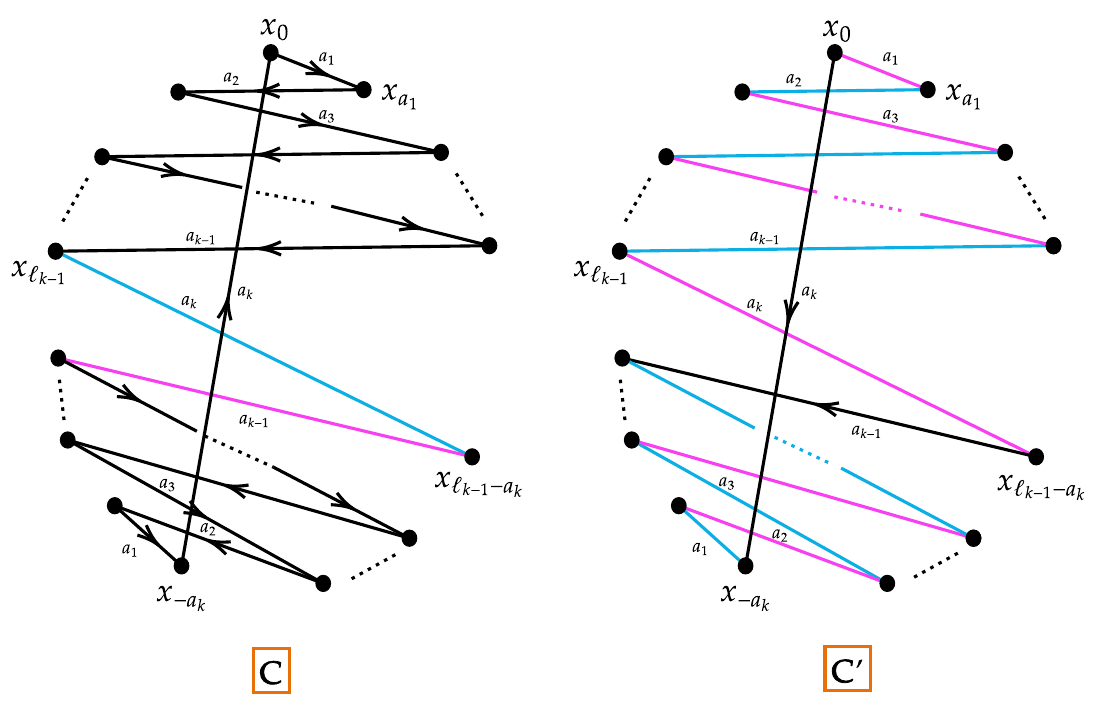}
    \caption{Lemma \ref{lem:color-3} -- Starter $m$-cycles for an HOP $(C_m)$-decomposition of $4G^{\bullet}$ for $m\equiv 2\ ({\rm mod}\ 4)$.}
    \label{fig:color41}
  \end{figure}
\begin{lem}{\label{lem:color-m-even-n-odd}}
Let $m\equiv 2\ ({\rm mod}\ 4)$, and let $n$ be an odd integer such that $6\leq m<n$. If $n(n-1)\equiv m\ ({\rm mod}\ 2m)$, then $4K_n^{\bullet}$ admits an HOP $(C_m)$-decomposition. 
\end{lem}

\begin{proof}
By Lemma~\ref{lem:n-odd-general}, it is enough to prove the result for odd \( n \) satisfying \( m < n < 2m \). Write \( n = m + r + 1 \), where \( 0 \leq r < m - 1 \) is even. 
If \( r = 0 \), then \( n = m + 1 \), and by Lemma~\ref{lem:color-1}, the multigraph \( 4K_{m+1}^{\bullet} \) admits an HOP \((C_m)\)-decomposition. Thus, we assume \( 0 < r < m - 1 \).
Let \( V(K_n) = \{x_i : i \in \ZZ_n\} \). Then  \( K_n = \text{Circ}(n;  \pm A) \), where \( A = \{1, 2, \dots, \frac{n-1}{2}\} \).
 Define \( G_1 = \text{Circ}(n;  \pm B) \), where \(|B|=\frac{r}{2}\), \( B \subseteq A \), and $B$ is carefully selected as in the proof of  \cite[Theorem 3.1]{2Kn-dec} to ensure that \( 2G_1 \) admits a \((C_m)\)-decomposition.  Let \( G_2 = \text{Circ}(n;  \pm(A \setminus B)) \). Note that $
K_n= G_1 \oplus G_2.
$
Since \( |A \setminus B| = \frac{m}{2} \), Lemma~\ref{lem:color-3} implies that \( 4G_2^{\bullet} \) admits an HOP \((C_m)\)-decomposition. Furthermore, the \((C_m)\)-decomposition \(\mathcal{D}\) of \( 2G_1 \) from \cite{2Kn-dec} has the property that for any \( m \)-cycle \( C \in \mathcal{D} \), there exists an \( m \)-cycle \( C' \) such that \( C \oplus C' = 2C \). In other words, the authors essentially  constructed a $(C_m)$-decomposition of $G_1$. Thus, by Lemma~\ref{lem:Gtool4}, there exists an HOP $(C_m)$-decomposition of $4G_1^{\bullet}$.
Hence, by Lemma~\ref{Gtool3}, the multigraph $4K_n^{\bullet}$ admits an HOP $(C_m)$-decomposition. 
\end{proof}

\subsection{The case when \( m \) is odd and \( n \) is even}

We first address the case when $m = 3$, then we assume $m\geq5$ for the remainder of this subsection.

\begin{lem}{\label{lem:m=3-00}}
Let $n$ be an even integer such that $3|n(n-1)$. There exists an HOP $(C_3)$-decomposition of $4K_n^{\bullet}$.
\end{lem}
\begin{proof}
By the assumption we have $n=3k$ for an even $k$, or $n=3k+1$ for an odd $k$.
First, assume $n=3k$. In \cite{LDMSaj}, it is shown that there exists an HOP $C_3$-factorization of $4K_{3k}^{\bullet}$. Thus, $4K_{3k}^{\bullet}$ admits an HOP $(C_3)$-decomposition.

For $n=3k+1$ with an odd $k$, we have $n=6t+4$ for an integer $t$.
We use the $(C_3)$-decomposition of $2K_n$ provided in \cite{B-m=3}.
Let $x_0$ be an arbitrary vertex of $K_n$. Obtain $G$ from $K_n$ by deleting $x_0$ and all edges having $x_0$ as an endpoint. In \cite[Theorem~\ref{thm:ray_wilson}]{Chaudhuri},  it is proved that there exists a decomposition of $G$ into $3t+1$ subgraphs (parallel classes), each consisting of $2t+1$ pairwise vertex-disjoint $3$-cycles. Let $\mathcal{P}$ be one of these subgraphs, and label its $3$-cycles as  $C_i = x_i\ y_i\ z_i \ x_i$, for $1\leq i\leq 2t+1$.
 Now, for each $i=1, \ldots, 2t+1$, construct $3$-cycles $C'_i =x_0\ x_i\ y_i\ x_0$, $C_i''=x_0\ y_i\ z_i\ x_0$, and $C_i '''=x_0\ z_i\ x_i\ x_0$ of $K_n$. Observe that $C_i, C'_i, C''_i, C'''_i$ jointly cover each of the edges  $x_0 x_i$, $x_0 y_i$, $x_0 z_i$, $x_i y_i$, $x_i z_i$, and $y_i z_i$ exactly twice. We color the four cycles as follows, ensuring they cover exactly one pink and one black copy of the mentioned edges, and each cycle has an even number of pink edges.
  \vspace{-10pt}
 \begin{figure}[H]
 \centering
   \includegraphics[width=0.8\linewidth,height=0.8\textheight,keepaspectratio]{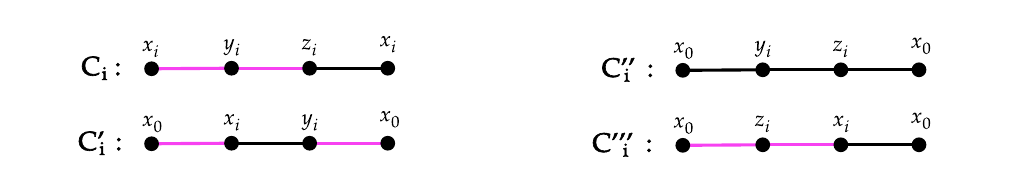}
   \vspace{-10pt}
  \end{figure}
  Let $R$ be the join of the graph $\mathcal{P}$ and $K_1$ with vertex set $\{x_0\}$. We have shown that $2R^{\circ}$ admits a $(C_3)$-decomposition in which every cycle contains an even number of pink edges. Hence by Lemma~\ref{lem:Gtool2}, $4R^{\bullet}$ admits an HOP $(C_3)$-decomposition. 
  Since $G-\mathcal{P}$ admits a $(C_3)$-decomposition, by Lemma~\ref{lem:Gtool4}, $4(G-\mathcal{P})^{\bullet}$ admits an HOP $(C_3)$-decomposition. Finally, since $K_n= R\oplus (G- \mathcal{P})$, the result follows by Lemma~\ref{Gtool3}. 
\end{proof}


We now review the necessary terminology and tools from~\cite{2Kn-dec} to establish a key reduction step that allows us to reduce the problem of finding an HOP \( (C_m) \)-decomposition of \( 4K_n^{\bullet} \), for odd \( m \) and even \( n > m \), to the case where \( n \) lies in the range \( m < n < 3m \).

\begin{defn}{\label{def:s-subg}}{\rm
Assume $m$ is odd, and let $X=\{x_0, x_1,\ldots, x_{m-1}\}$ and $Y=\{y_0, y_1,\ldots, y_{m-1}\}$ be  two sets of size $m$ that partition the vertex set of $K_{2m}$, and let $S,S'\subseteq \{1,2\ldots, \tfrac{m-1}{2}\}$ and $T\subseteq \{0,1,2\ldots, m-1\}$. 
 Then $K_{2m} \langle S, T,  S' \rangle$ denotes a subgraph $G$ of $K_{2m}$ such that $G[X]\cong {\rm{Circ}}(m; S)$,   $G[Y]\cong {\rm{Circ}}(m; S')$, and $x_iy_{i+d}\in E(G)$ if and only if $d\in T$. 
} 
\end{defn}

\begin{lem}{\label{lem:join}}
Assume $m\geq 3$ is odd, and let $G_1= K_{2m} \big\langle \emptyset, \{0\},  \emptyset \big\rangle  \bowtie \overline{K_{\frac{m-1}{2}}}$ and $G_2= K_{2m} \big\langle \{1\}, \emptyset,  \emptyset \big\rangle  \bowtie \overline{K_{\frac{m-1}{2}}}$. Then $4G_1^{\bullet}$ and $4G_2^{\bullet}$ admit an HOP $(C_m)$-decomposition. 
\end{lem}
\begin{proof}
Let $m=2k+1$ for $k\in\ZZ$, and let $V(\overline{K_k})=\{z_0, z_1,\ldots, z_{k-1}\}$. Let  $$\omega=(x_0\ x_1\ \ldots\ x_{m-1})(y_0\ y_1\ \ldots\ y_{m-1})$$ be a permutation that fixes all vertices in $\overline{K_k}$. Let $\omega_{\bullet}$ be the permutation induced by $\omega$  that preserves the color (and orientation)  of the edges.
Consider the following  $m$-cycles in $G_1$:
$$C_1= x_0\ y_0\ z_0\ y_1\ z_1\ y_2\ z_2\ \ldots\ z_{k-2}\ y_{k-1}\ z_{k-1}\ x_0$$
$$C_2= y_0\ x_0\ z_{k-2}\ x_1\ z_{k-3}\ x_2\ \ldots\ z_1\ x_{k-2}\ z_0\ x_{k-1}\ z_{k-1}\ y_0$$
In \cite[Proposition 5.4]{2Kn-dec}, it is shown that $\{\omega^i(C_1), \omega^i(C_2): i=0,1,\ldots, m-1\}$ is a $(C_m)$-decomposition of $2G_1$. 
Take two copies of each of $C_1$ and $C_2$ and color them as follows:
 \begin{figure}[H]
 \centering
   \includegraphics[width=1\linewidth,height=0.8\textheight,keepaspectratio]{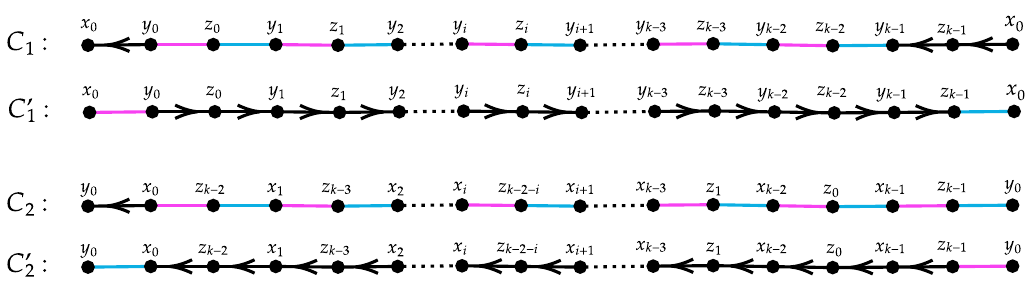}
  \end{figure}
It can be verified $C_1, C'_1, C_2, C'_2$ jointly contain  exactly one edge from each orbit of $\langle \omega_{\bullet}\rangle$. Additionally, they satisfy Condition {\bf (C1)} of Definition {\rm{\ref{def}}}.
 Hence, $$\{\omega_{\bullet}^i(C_1), \omega_{\bullet}^i(C'_1),\omega_{\bullet}^i(C_2), \omega_{\bullet}^i(C'_2): i=0,1,\ldots, m-1\}$$ is an HOP $(C_m)$-decomposition of $4G_1^{\bullet}$.

Now, consider the following  $m$-cycles in $G_2$:
$$C_1 = x_0\ x_1\ z_0\ y_0\ z_1\ y_1\ z_2\ \ldots\ z_{k-2}\ y_{k-2}\ z_{k-1}\ x_0$$
$$C_2 = x_1\ x_0\ z_{1}\ x_2\ z_{2}\ x_3\ \ldots\ z_{k-3}\ x_{k-2}\ z_{k-2}\ x_{k-1}\ z_{k-1}\ y_0\ z_0\ x_1$$
In \cite[Proposition 5.4]{2Kn-dec}, it is  proved that $\{\omega^i(C_1), \omega^i(C_2): i=0,1,\ldots, m-1\}$ is a $(C_m)$-decomposition of $2G_2$. 
Take two copies of each cycle and color them as follows:
 \begin{figure}[H]
 \centering
   \includegraphics[width=1\linewidth,height=0.8\textheight,keepaspectratio]{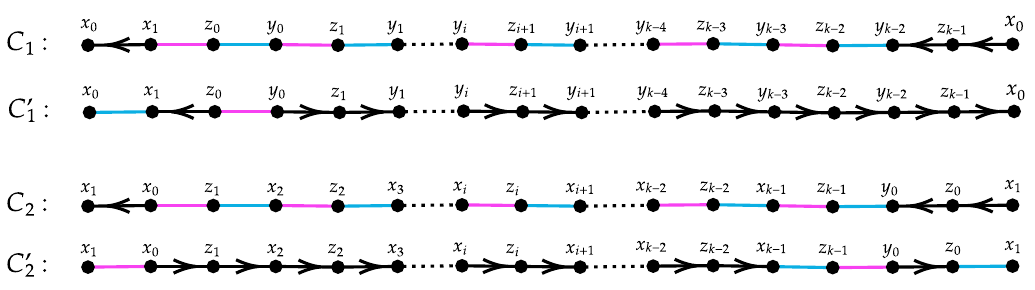}
  \end{figure}
Observe  that $C_1, C'_1, C_2,$ and $C'_2$ satisfy Condition {\bf (C1)} of Definition {\rm{\ref{def}}}, and jointly contain  exactly one edge from each orbit of $\langle \omega_{\bullet}\rangle$.
 Therefore, $$\{\omega_{\bullet}^i(C_1), \omega_{\bullet}^i(C'_1),\omega_{\bullet}^i(C_2), \omega_{\bullet}^i(C'_2): i=0,1,\ldots, m-1\}$$ is an HOP $(C_m)$-decomposition of $4G_2^{\bullet}$. 
\end{proof}

The following proposition is based on \cite[Proposition 5.5]{2Kn-dec}. 
\begin{prop} {\label{propcor:join2}}
Let $G=K_{2m}   \bowtie \overline{K_t}$. For integers $m$ and $t$ satisfying $m = 2k + 1$ and $t = qk + r$ where $0 \le r\leq k-1$ and $1 \le q \le m + 2r - 1$, the multigraph $2G$ is $(C_m)$-decomposable. In particular, $2G$ is $(C_m)$-decomposable whenever $\frac{m-1}{2} \le t \le \frac{(m-1)^2}{2}$.
\end{prop}

The following corollary is a direct consequence of Proposition~\ref{propcor:join2}. 


\begin{cor} {\label{prop:join2}}
Let $m=2k+1$, and let $t=qk+r$ for $0\leq r\leq k-1$ and $1\leq q\leq m+2r-1$. Let $G=K_{2m}   \bowtie \overline{K_t}$. Then $4G^{\bullet}$ admits an HOP $(C_m)$-decomposition. In particular, $4G^{\bullet}$ admits an HOP $(C_m)$-decomposition whenever $\frac{m-1}{2}\leq t\leq \frac{(m-1)^2}{2}$.
\end{cor}

Having established the necessary terminology and tools, we are now ready to state the reduction step. 

\begin{lem}{\label{lem:color-5}}
Let $m\geq 5$ be an odd integer. If $4K_n^{\bullet}$ admits an HOP $(C_m)$-decomposition for all even $n$ such that $m< n<3m$ and $m|n(n-1)$, then $4K_n^{\bullet}$ admits an HOP $(C_m)$-decomposition for all even $n> m$ such that $m|n(n-1)$.
\end{lem}
\begin{proof}
Take any even $n'> m$ such that $m|n'(n'-1)$. We can write $n'=n+2qm$, for integers $q\geq 0$ and $m< n< 3m$. Moreover, since $m|n'(n'-1)$, it follows that $m|n(n-1)$. 
Notice that we can partition the vertex set of $4K_{n'}^{\bullet}$ into $q$ sets of $2m$ vertices and one set of $n$ vertices. Depending on the value of $q$, there are three cases:
\begin{itemize}
\item $q=0$. Then  $4K_{n'}^{\bullet}=4K_n^\bullet$.

\item $q=1$. Then  $4K_{n'}^{\bullet}=4K^{\bullet}_{2m} \bowtie \overline{4K^{\bullet}_n}\oplus4K^{\bullet}_{n}$. 

\item $q\geq 3$. Then  $4K_{n'}^{\bullet}$ is an edge-disjoint union of $q$ copies of $4K^{\bullet}_{2m} \bowtie \overline{4K^{\bullet}_n}$, one copy of $4K^{\bullet}_{q[2m]}$, and one copy of $4K^{\bullet}_{n}$. 

\item $q=2$. Then $4K_{n'}^{\bullet}$ is an edge-disjoint union of one copy of $4K^{\bullet}_{2m} \bowtie \overline{4K^{\bullet}_n}$,  one copy of $4K^{\bullet}_{n}$, and one copy of $4K^{\bullet}_{2m} \bowtie \overline{4K^{\bullet}_{2m+n}}$. 
\end{itemize}
By \cite{H-sys}, the complete multipartite graph with $q\geq3$ parts of cardinality $2m$, that is, $K_{q[2m]}$, has  a $(C_m)$-decomposition; by Lemma~\ref{lem:Gtool4}, there exists an HOP $(C_m)$-decomposition of $4K^{\bullet}_{q[2m]}$. Moreover, by supposition $4K^{\bullet}_{n}$ has an HOP $(C_m)$-decomposition.  It remains to show that $4K^{\bullet}_{2m} \bowtie \overline{4K^{\bullet}_n}$ and $4K^{\bullet}_{2m} \bowtie \overline{4K^{\bullet}_{2m+n}}$ admit   HOP $(C_m)$-decompositions. 
\begin{itemize}
\item For \( m \geq 13 \), we have \( n < 2m + n < 5m \leq \frac{(m-1)^2}{2} \); hence, by Corollary~\ref{prop:join2}, each of \( 4K^{\bullet}_{2m} \bowtie \overline{4K^{\bullet}_n} \) and \( 4K^{\bullet}_{2m} \bowtie \overline{4K^{\bullet}_{2m+n}} \) admits an HOP \((C_m)\)-decomposition. 

\item
For \( m \leq 11 \), since \( m \mid n(n-1) \) and \( m \) is odd, it follows that \( m \mid n \) or \( m \mid (n-1) \). Since \( n \) is even and \( m < n < 3m \), it follows that \( n \) must be either \( m + 1 \) or \( 2m \). 
Consequently, we need an HOP \( (C_m) \)-decomposition for \( 4K^{\bullet}_{2m} \bowtie \overline{4K^{\bullet}_{m+1}} \), \( 4K^{\bullet}_{2m} \bowtie \overline{4K^{\bullet}_{2m}} \), \( 4K^{\bullet}_{2m} \bowtie \overline{4K^{\bullet}_{3m+1}} \), and \( 4K^{\bullet}_{2m} \bowtie \overline{4K^{\bullet}_{4m}} \).  
First, observe that the case \( 4K^{\bullet}_{2m} \bowtie \overline{4K^{\bullet}_{4m}} \) arises from  \( n = 2m \) and \( q = 2 \). Thus, we actually seek an HOP \((C_m)\)-decomposition of \( 4K^{\bullet}_{6m} \). The multigraph \( 4K^{\bullet}_{6m} \) is an edge-disjoint union of three copies of \( 4K^{\bullet}_{2m} \) and one copy of \( 4K^{\bullet}_{3[2m]} \), each of which admits an HOP \((C_m)\)-decomposition. For the remaining cases, if \( m \geq 9 \), then \( m+1<2m<3m+1 \leq \frac{(m-1)^2}{2} \). Hence, each of these multigraphs admits an HOP \( (C_m) \)-decomposition by Corollary~\ref{prop:join2}. Thus, we only need to verify the cases where \( m \leq 7 \).

\item  For $m=7$, we need to find an HOP $(C_m)$-decomposition for $4K^{\bullet}_{14} \bowtie \overline{4K^{\bullet}_{8}}$, $4K^{\bullet}_{14} \bowtie \overline{4K^{\bullet}_{14}}$, and $4K^{\bullet}_{14} \bowtie \overline{4K^{\bullet}_{22}}$.  It can be verified that these exist by Corollary~\ref{prop:join2}.

\item  For $m=5$, we need to find an HOP $(C_5)$-decomposition for $4K^{\bullet}_{10} \bowtie \overline{4K^{\bullet}_{6}}$, $4K^{\bullet}_{10} \bowtie \overline{4K^{\bullet}_{10}}$, and $4K^{\bullet}_{10} \bowtie \overline{4K^{\bullet}_{16}}$. 

For $4K^{\bullet}_{10} \bowtie \overline{4K^{\bullet}_{6}}$, we can use Corollary~\ref{prop:join2}.

For $4K^{\bullet}_{10} \bowtie \overline{4K^{\bullet}_{10}}$, partition the vertices in $\overline{4K^{\bullet}_{10}}$ into five sets of two vertices. Then,  $4K^{\bullet}_{10} \bowtie \overline{4K^{\bullet}_{10}}$ can be viewed as an edge-disjoint union of five isomorphic copies of $4G^{\bullet}$ where $G=K_{10} \big\langle \emptyset, \{0\},  \emptyset \big\rangle  \bowtie \overline{K_2}$ and two copies of $4K_5^{\bullet}$. Clearly, $4K_5^{\bullet}$ admits an HOP $(C_5)$-decomposition, and by Lemma~\ref{lem:join}, there exists an HOP $(C_5)$-decomposition of $4G^{\bullet}$. Hence, $4K^{\bullet}_{10} \bowtie \overline{4K^{\bullet}_{10}}$ admits an HOP $(C_5)$-decomposition. 

For $4K^{\bullet}_{10} \bowtie \overline{4K^{\bullet}_{16}}$, partition the vertices in $\overline{4K^{\bullet}_{16}}$ into eight sets of two vertices. Then,  $4K^{\bullet}_{10} \bowtie \overline{4K^{\bullet}_{16}}$ can be viewed as an edge-disjoint union of five isomorphic copies of $4G_1^{\bullet}$ where $G_1=K_{10} \big\langle \emptyset, \{0\},  \emptyset \big\rangle  \bowtie \overline{K_2}$,  three isomorphic copies of $4G_2^{\bullet}$ where $G_2=K_{10} \big\langle \{1\}, \emptyset,  \emptyset \big\rangle  \bowtie \overline{K_2}$, and  one $4C_5^{\bullet}$.  By Lemma~\ref{lem:join}, there exist  HOP $(C_5)$-decompositions of $4G_1^{\bullet}$ and $4G_2^{\bullet}$. Thus, $4K^{\bullet}_{10} \bowtie \overline{4K^{\bullet}_{16}}$ admits an HOP $(C_5)$-decomposition.
\end{itemize}
Therefore, in all cases, $4K_{n'}^{\bullet}$ admits an HOP $(C_m)$-decomposition. 
\end{proof}

We now address the special cases \( n = m + 1 \) and \( n = 2m \).

\begin{lem}{\label{lem:color-6}}
For any odd integer $m \geq 5$, each of the multigraphs $4K_{2m}^{\bullet}$ and $4K_{m+1}^{\bullet}$ admits an HOP $(C_m)$-decomposition.
\end{lem}
\begin{proof}
By Lemma~\ref{lems-2starterQ}, there exists an HOP $C_m$-factorization of $4K_{2m}^{\bullet}$, and hence an HOP $(C_m)$-decomposition.

Next, let $V(4K_{m+1}^{\bullet}) = \{x_i : i \in \mathbb{Z}_m\} \cup \{x_{\infty}\}$. Let $\rho_{\bullet}$ be the permutation on $E(4K_{m+1}^{\bullet})$ that preserves edge colors and orientations, induced by the permutation $\rho = (x_{\infty})(x_0 \ x_1 \ x_2 \ \ldots \ x_{m-1})$.
In \cite[Lemma 5.7]{2Kn-dec}, it is shown that $\W=\{C_1, \rho(C_1), \ldots, \rho^{m-1}(C_1), C_2\}$  is a $(C_m)$-decomposition of $2K_{m+1}$, where 
$$C_1= x_0\ x_{1}\ x_{-1}\ x_{2}\ \ldots \ x_{\frac{m-1}{4}}\ x_{-\frac{m-1}{4}}\ x_{\frac{m+7}{4}}\ x_{-\frac{m+3}{4}}\ \ldots \ x_{\frac{m+3}{2}} \ x_{\frac{m+1}{2}}\ x_{\infty}\ x_0$$ 
if $m\equiv 1\ ({\rm mod}\ 4)$, and 
$$C_1= x_0\ x_{1}\ x_{-1}\ x_{2}\ \ldots \ x_{\frac{m-3}{4}}\ x_{-\frac{m-3}{4}}\ x_{\frac{m+5}{4}}\ x_{-\frac{m+1}{4}}\ x_{\frac{m+9}{4}}\ \ldots \ x_{\frac{m+3}{2}} \ x_{\frac{m+1}{2}}\ x_{\infty}\ x_0$$ 
if $m\equiv 3\ ({\rm mod}\ 4)$, and $C_2=x_0\ x_{\frac{m-1}{2}}\ x_{m-1}\ x_{\frac{m-3}{2}}\ \ldots\ x_{\frac{m+1}{2}}\ x_0.$
Observe that $C_2$ only covers difference $\frac{m-1}{2}$. Let $C'_2$ be another copy of $C_2$. Color $C_2$ and $C'_2$ black and orient them in opposite directions. By doing this, we cover all black arcs corresponding to the difference $\frac{m-1}{2}$. 
%
Notice that $C_1$ covers each difference in $\{ 1,2\ldots, \frac{m-3}{2}\}$ exactly twice and difference $\frac{m-1}{2}$ exactly once, namely,  in the order
$1, 2, 3,\ldots, \textstyle{\frac{m-3}{2}}, \textstyle{\frac{m-1}{2}}, \textstyle{\frac{m-3}{2}}, \ldots, 3, 2, 1,\infty, \infty.$
Let $C_1'$ be another copy of $C_1$ and color them as follows.
\begin{itemize}
\item If $m\equiv 1\ ({\rm mod}\ 4)$: 
 \begin{figure}[H]
 \centering
   \includegraphics[width=0.9\linewidth,height=0.9\textheight,keepaspectratio]{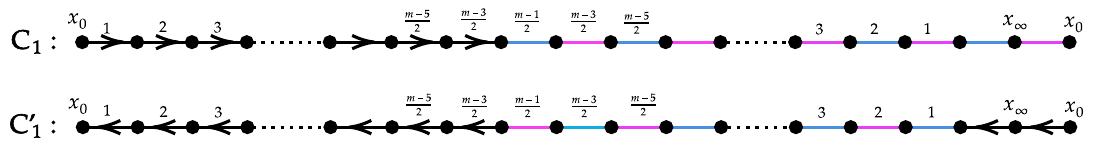}
   \vspace{-15pt}
  \end{figure} 
\item If $m\equiv 3\ ({\rm mod}\ 4)$: 
 \begin{figure}[H]
 \centering
   \includegraphics[width=0.9\linewidth,height=0.9\textheight,keepaspectratio]{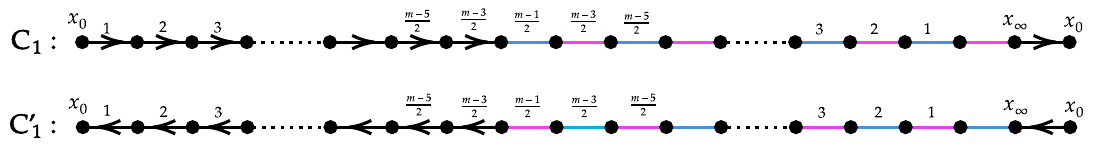}
   \vspace{-15pt}
  \end{figure} 
\end{itemize}
%
It can be verified that 
$\W'=\big\{\rho_{\bullet}^i(C_1), \rho_{\bullet}^i(C'_1): i=0,1,\ldots,m-1\big\}\cup\{C_2, C'_2\}$
 is an HOP $(C_m)$-decomposition  for $4K_{m+1}^{\bullet}$. 
\end{proof}
%
%
\begin{lem}{\label{lem:color-m-odd-n-even}}
Let $m$ be odd, and let $n$ be an even integer with $3\leq m\leq n$. If $m|n(n-1)$, then $4K_n^{\bullet}$ admits an HOP $(C_m)$-decomposition. 
\end{lem}

\begin{proof}
The case $m = 3$ is handled in Lemma~\ref{lem:m=3-00}. For $m \geq 5$, by Lemma~\ref{lem:color-5}, it is enough to prove the result for even $n$ in the range $m < n < 3m$.

The required HOP $(C_m)$-decomposition will be generated by starter central and starter peripheral cycles. For the peripheral cycles, we use the construction provided for \(2K_n\) in \cite{2Kn-dec} and extend it to starter peripheral cycles for \(4K_n^{\bullet}\). However, for the central cycles, we introduce a new construction that differs from the one in \cite{2Kn-dec}.

By Lemma~\ref{lem:color-6}, we know that $4K_{m+1}^{\bullet}$ and $4K_{2m}^{\bullet}$ admit  HOP $(C_m)$-decompositions. Therefore, we may assume $m + 2 < n < 2m$ or $2m < n < 3m$.

Let \( V(2K_n) = \{x_i : i \in \mathbb{Z}_{n-1}\} \cup \{x_{\infty}\} \), and let  $\rho=(x_{\infty})(x_0 \ x_1 \ \ldots \ x_{n-2})$. Observe that the group $\langle \rho \rangle$ has the following orbits on the edge set of $K_n$:
\begin{itemize}
\item for each $s\in \{1,2,\ldots, \textstyle{\frac{n-2}{2}}\}$, we have  an orbit $\{x_ix_{i+s}: i\in \ZZ_{n-1} \}$; and
\item $\{x_ix_{\infty}: i\in \ZZ_{n-1} \}$.
\end{itemize}
For convenience, let \( S = \{ 1,  2, \ldots,  \frac{n-2}{2}, \infty\} \) be the set of all differences. 
Let $n= 2^e a$, where $a$ is odd, and let $m=a'b'$, where $a'$ and $b'$ both are odd with $a'|a$ and $b'|(n-1)$. Note that $b'\geq 3$ since $m\nmid n$. 

To construct the peripheral cycles, we partition the \( n-1 \) vertices in \( \{x_i : i \in \mathbb{Z}_{n-1}\} \) into \( b' \) segments, each containing \( \ell = \frac{n-1}{b'} \) consecutive vertices. Note that each segment will contribute \( a' \) edges toward a peripheral  \( m \)-cycle.
%

If \( 2K_n \) is \((C_m)\)-decomposable, then the number of \( m \)-cycles in the decomposition is
$$\frac{n(n-1)}{m} = \frac{(n-1)(m + n - m)}{m} = (n-1) + \frac{n-1}{b'} \cdot \frac{n - m}{a'} = (n-1) + \ell \cdot \frac{n - m}{a'}=(n-1) + \ell F.$$
where
\(
F = \frac{n - m}{a'} = \frac{2^{e} a - a' b' }{a'}.
\)
Since \( a' \mid a \), and \( a' \), \( b' \), \( a \) are all odd, and $e\geq 1$,  it follows that \( F \) is an odd integer.
The number of \( m \)-cycles suggests constructing one starter central cycle and \( F \) starter peripheral cycles. 

As in \cite{2Kn-dec}, we will address the subcases \( (m, n) \in \{(15, 36), (15, 40)\} \) separately at the end of the proof. In all others cases, the peripheral cycles are constructed as follows. 

The authors of \cite{2Kn-dec} first construct a starter peripheral cycle \( C_0 \), which covers a set of differences \( A = \{s_{1}, s_{2},  \ldots,  s_{a'}\} \subset S \), where \(s_{1}< s_{2}<  \ldots<  s_{a'}\). If \( F = 1 \), the remaining differences can be covered by the central cycles. However, if \( F \geq 3 \), additional families of peripheral cycles are needed. 
For \( F \geq 3 \), let \( c = \frac{F - 1}{2} \). The authors of \cite{2Kn-dec} construct \( c \) starter peripheral cycles, called \( C_1, C_2, \ldots, C_c \) such that 
\(
\{C_i, \rho(C_i), \ldots, \rho^{\ell-1}(C_i) \mid i = 1, 2, \ldots, c\} 
\) 
is a \((C_m)\)-decomposition of \( G_1 = \text{Circ}(n - 1; \pm B) \), for a suitable set $B\subset S\setminus A$ with $|B|=ca'$. By Lemma~\ref{lem:Gtool4} , we know \(4G_1^{\bullet}\) admits an HOP \((C_m)\)-decomposition.

Let \( C'_0 \) be another copy of \( C_0 \). Color both \( C_0 \) and \( C'_0 \) black and orient them in opposite directions. The families of cycles generated by these two cycles cover the black orbits corresponding to the differences in \( A \). The differences in \( A \) will also appear in the starter central cycle \( C \), to be constructed next.

Let  \(\mathscrsfs{L} = S \setminus (A \cup B)= \{a_1, a_2, \ldots, a_t\} \cup \{\infty\}\), where \( a_1 < a_2 < \ldots < a_t \leq \frac{n-2}{2} \) are integers and \( t = \frac{n-2}{2} - \frac{F+1}{2}a' \).  
Consider the following sets of differences: 
\begin{itemize}
\item \( W_1=\{ s_1, s_2, \ldots, s_{a'}, a_1, a_2, \ldots, a_{t-1}, a_t\}\)
\item \( W_2=\{ a_1, a_2,  \ldots, a_{t-1}, a_t\} \)
\end{itemize}
Here,  \( |W_1|= a' + t \), and since the differences in \( W_1 \) are pairwise distinct, we may relabel them as
\(
W_1 = \{ b_1, b_2, \ldots, b_{a'+t} \}  \text{ with }  b_1 < b_2 < \cdots < b_{a'+t}.
\)
By Lemma~\ref{zig-zag}(i), there exists a path $T$ starting at $x_0$ that covers the differences in $W_1$ and $W_2$ in the order $b_1, b_2, \dots, b_{a'+t}, a_t, a_{t-1}, \dots, a_1$. Substituting the value of $t$ and $F$, we see that $T$ has length $2t + a' = m - 2$.

Let \( C = x_{\infty} T x_{\infty} \). Observe that \( C \) covers each difference in \(\mathscrsfs{L} \) exactly twice and each difference in \( A \) exactly once. Let \( C' \) be another copy of \( C \).  
In \( C \), color the first \( t + a' \) edges with differences \( b_1, b_2, \ldots, b_{a'+t}\) alternately pink and blue, starting with \( b_1 \) in blue, and in $C'$, in the opposite way. The remaining coloring depends on whether \( t \) is even or odd:
\begin{itemize}
    \item Assume \( t \) is even. Since \( a' \) is odd, \( t + a' \) is also odd, so the edge with difference \( b_{a'+t} \) is colored blue.  Then, complete the coloring as shown below.
    \begin{figure}[H]
        \centering
        \includegraphics[width=0.9\linewidth,height=0.9\textheight,keepaspectratio]{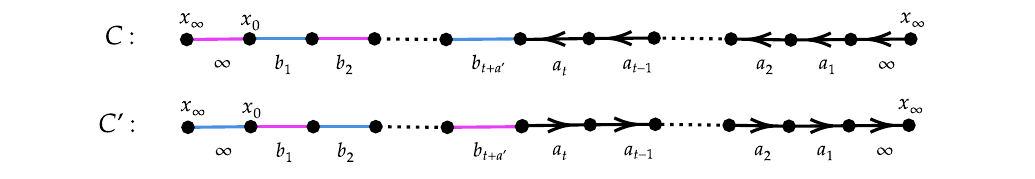}
        \vspace{-15pt}
    \end{figure}
    \item Assume \( t \) is odd. Then, \( t + a' \) is even, so the edge with difference \( b_{a'+t} \) is colored pink.  Then, complete the coloring as shown below. 
    \begin{figure}[H]
        \centering
        \includegraphics[width=0.9\linewidth,height=0.9\textheight,keepaspectratio]{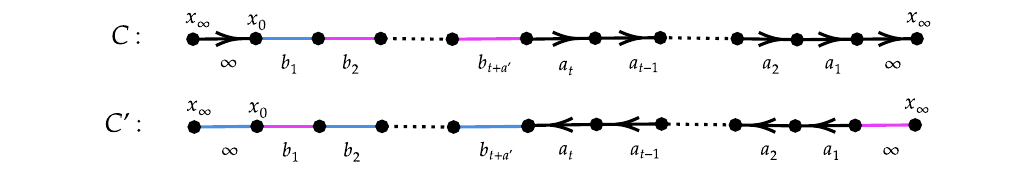}
        \vspace{-15pt}
    \end{figure}

\end{itemize}  

Observe that the opposite black arcs corresponding to each difference in \( A \) are covered by \( C_0 \) and \( C'_0 \), while the pink and the blue edges corresponding to each difference in \( A \) are covered by the starter central cycles \( C \) and \( C' \). 
Moreover, the starter central cycles jointly contain exactly one blue edge, one pink edge, and two opposite black arcs corresponding to each difference in \(\mathscrsfs{L} \).  
Therefore,  
\(
\{\rho_{\bullet}^i(C), \rho_{\bullet}^i(C') \mid i = 0, 1, \ldots, n-2\} \cup  \{\rho_{\bullet}^i(C_0), \rho_{\bullet}^{i}(C_0') \mid i = 0, 1, \ldots, \ell-1\}
\)
is an HOP  \((C_m)\)-decomposition of \( 4G_2^{\bullet} \), where  
\(
G_2 = \text{Circ}(n-1; \pm(A \cup (\mathscrsfs{L} - \{\infty\})) \bowtie K_1.
\)
Since $4K_n^{\bullet}=4G_1^{\bullet}\oplus 4G_2^{\bullet}$,  by Lemma~\ref{Gtool3}, the multigraph $4K_n^{\bullet}$ admits an HOP $(C_m)$-decomposition.


To complete the proof of Lemma \ref{lem:color-m-odd-n-even},  we prove that HOP $(C_{15})$-decompositions of $4K^{\bullet}_{36}$ and $4K^{\bullet}_{40}$ exist.
As before, we construct starter peripheral and starter central cycles. We begin with the starter peripheral cycles for $2K_{36}$ and $2K_{40}$, as provided in \cite{2Kn-dec}, and extend them to obtain the starter peripheral cycles for $4K_{36}^{\bullet}$ and $4K_{40}^{\bullet}$. Finally, we describe the construction of the starter central cycles.

For $n=36$, we have $a=9$, $a'=3$, $b'=5$, and the set of differences is $S = \{1,2,\ldots,17,\infty\}$.
We partition the $n-1 = 35$ vertices into $b' = 5$ segments, where each segment contains $\ell = \frac{n-1}{b'} = 7$ consecutive vertices. Each segment contributes $a' = 3$ edges toward a $15$-cycle. 
The number of families of peripheral cycles is $F = \frac{n-m}{a'} = 7$. Following \cite{2Kn-dec}, we define the paths $P_0 = x_0\ x_5\ x_{-1}\ x_{14}$, $P_1 = x_0\ x_1\ x_{-1}\ x_7$, $P_2 = x_0\ x_3\ x_{-1}\ x_{-14}$, and $P_3 = x_0\ x_9\ x_{-2}\ x_{-14}$.
For each \( i = 0, 1, 2, 3 \), we have
\(
C_i = P_i \cup \rho^\ell(P_i) \cup \rho^{2\ell}(P_i) \cup \rho^{3\ell}(P_i) \cup \rho^{4\ell}(P_i),
\)
which is a cycle of length 15.
The families generated by the peripheral cycles $C_1, C_2, C_3$ jointly cover differences in $B = \{1,2,3,4,8,9,11,12,13\}$, and form a $(C_{15})$-decomposition of $G_1 = \text{Circ}(35; \pm B)$. By Lemma~\ref{lem:Gtool4}, there exists 
 an HOP $(C_{15})$-decomposition of $4G_1^{\bullet}$.
The starter peripheral cycle $C_0$ covers  the differences in the set $A = \{5, 6, 15\}$.
 Let $\mathscrsfs{L}= S \setminus (A \cup B) = \{7,10,14,16,17,\infty\}$ be the set of unused differences.
Consider the sets  \( W_1=\{ 5,6,7,10,14,15,16,17\}\)
and \( W_2=\{ 7, 10,14,16,17\} \). 
Using Lemma~\ref{zig-zag}(i), we construct a path $T$ starting from $x_0$, covering differences in the order 
\(  5,6,7,10,14,15,16,17,17,16,14,10,7. \)
Let \( C = x_{\infty} T x_{\infty} \). Notice that \( C \) is a cycle of length 15. Take another copy of \( C \), call it \( C' \). Color them as follows:
\begin{figure}[H]
    \centering
    \includegraphics[width=0.8\linewidth,height=0.8\textheight,keepaspectratio]{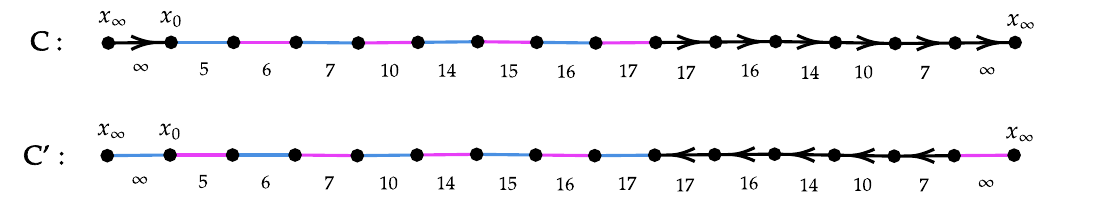}
    \vspace{-10pt}
\end{figure}
Now, let $C'_0$ be another copy of $C_0$, color both $C_0$ and $C'_0$ black, and orient them in the opposite directions. We see that
$$ \{\rho_{\bullet}^i(C), \rho_{\bullet}^{i}(C') \mid i=0,1,\dots,34\} \cup \{\rho_{\bullet}^i(C_0), \rho_{\bullet}^{i}(C_0') \mid i=0,1,\dots,6\} $$
is an HOP $(C_{15})$-decomposition of $4G_2^{\bullet}$, where $G_2=\text{Circ}(35; \pm(A \cup (\mathscrsfs{L} - \{\infty\})) \bowtie K_1$. 
Since $4K_{36}^{\bullet} =4G_1^{\bullet}\oplus 4G_2^{\bullet}$, the multigraph  $4K_{36}^{\bullet}$ admits an HOP $(C_{15})$-decomposition.

 
For \( n = 40 \), we have \( a = 5 \), \( a' = 5 \), \( b' = 3 \), and \( \ell = 13 \). The set of differences is \( S = \{1, 2, \dots, 19, \infty\} \).
Consider the paths \( P_0 = x_0\ x_9\ x_{-1}\ x_{10}\ x_{-2}\ x_{13} \), \( P_1 = x_0\ x_1\ x_{-3}\ x_2\ x_{-5}\ x_{13} \), and \( P_2 = x_0\ x_2\ x_{-1}\ x_5\ x_{-3}\ x_{13} \) from \cite{2Kn-dec}. 
For \( i = 0, 1,2 \), we see that \( C_i = P_i \cup \rho^\ell(P_i) \cup \rho^{2\ell}(P_i) \) is a cycle of length 15. 
The families generated by the peripheral cycles \( C_1 \) and \( C_2 \) jointly cover the differences in \( B = \{1, 2, 3, 4, 5, 6, 7, 8, 16, 18\} \).
As in the case $n=36$, we obtain an HOP \((C_{15})\)-decomposition of \( 4G_1^{\bullet} \), where \( G_1 = \text{Circ}(39; \pm B) \).

Note that the family generated by starter peripheral cycle $C_0$ covers the differences in the set \( A = \{9, 10, 11, 12, 15\} \).
The starter central cycle uses the remaining differences, namely \( \{13, 14, 17, 19, \infty\} \). Let \( C = x_{\infty} T x_{\infty} \), where \( T \) is a path of length 13 constructed using Lemma~\ref{zig-zag}(i).  
The path \( T \) starts at \( x_0 \) and covers the differences in the following order:
\[
9, 10, 11, 12, 13, 14, 15, 17, 19, 19, 17, 14, 13.
\]
Let \( C' \) be another copy of the starter central cycle $C$. We color them as follows:
\begin{figure}[H]
    \centering
    \includegraphics[width=0.8\linewidth,height=0.8\textheight,keepaspectratio]{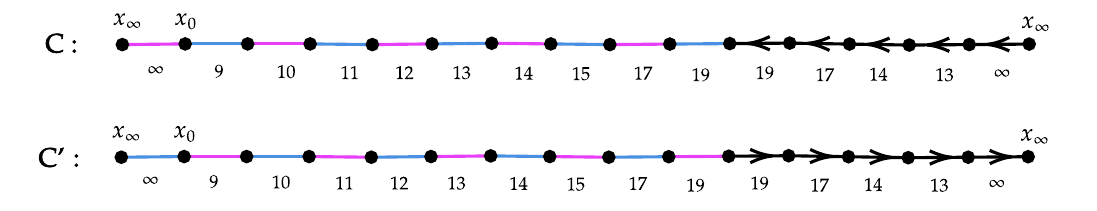}
     \vspace{-10pt}
\end{figure}
The construction is then completed as in the case $n=36$. 
\end{proof}


\section{HOP $(C_m)$-decomposition of $4K_n^{\bullet}$}{\label{sec:7}}

We now address the most challenging case; namely, when $m \mid 2n(n-1)$ but $m \nmid n(n-1)$. This condition implies that $m \equiv 0 \pmod{4}$ and $2n(n-1) \equiv m \pmod{2m}$. The proof is structured into several cases and subcases, where the main idea for each is to construct starter peripheral and central cycles that are then permuted to form the decomposition. 



\subsection{The case when $n$ is odd }


\begin{lem}{\label{lem:n-odd-4K-last}}
Let $m\equiv 0\ ({\rm mod}\ 4)$, and let $n$ be an odd integer such that $4\leq m< n$. If $2n(n-1)\equiv m\ ({\rm mod}\ 2m)$, then $4K_n^{\bullet}$ admits an HOP $(C_m)$-decomposition. 
\end{lem}
\begin{proof}
By Lemma \ref{lem:n-odd-general}, it suffices to prove this result for odd \( n \) in the range \( m < n < 2m \). 
We first outline the parameters. 
 Let \( V(4K_n^{\bullet}) = \{x_i : i \in \mathbb{Z}_{n-1}\} \cup \{x_{\infty}\} \), and let $\rho_{\bullet}$ be the permutation on $E(4K_n^{\bullet})$ that preserves the color (and orientation) of the edges, and is induced by  the permutation  $\rho=(x_{\infty})(x_0 \ x_1\  x_2 \ \ldots \ x_{n-2})$. Observe that the group $\langle \rho_{\bullet} \rangle$ has the following orbits on the edge set of $4K_n^{\bullet}$:
\begin{itemize}
\item for each $s\in \{1,2,\ldots, \textstyle{\frac{n-3}{2}}\}$, we have a pink and a blue orbit $\{x_ix_{i+s}: i\in \ZZ_{n-1} \}$;
\item a pink and a blue orbit $\{x_ix_{i+\frac{n-1}{2}}: i=0,1,\ldots, \textstyle{\frac{n-3}{2}} \}$;
\item for each $s\in \{1,2,\ldots, n-2\}$, we have a black orbit $\{(x_i, x_{i+s}): i\in \ZZ_{n-1} \}$;
\item a pink and a blue orbit  $\{x_ix_{\infty}: i\in \ZZ_{n-1} \}$; and
\item black orbits $\{(x_{\infty}, x_{i}): i\in \ZZ_{n-1} \}$ and $\{(x_i, x_{\infty}): i\in \ZZ_{n-1} \}$.
\end{itemize}
For convenience, let \( S = \{ 1,  2, \ldots,  \frac{n-3}{2},  \frac{n-1}{2}, \infty\} \) be the set of differences. Notice that $\frac{n-1}{2}$ is the diameter difference. 
Let \( n-1 = 2^a p \) and \( m = 2^b q \), where \( p \) and \( q \) are both odd. Observe that \( b \geq 2 \) since \( m \equiv 0 \pmod{4} \) and \(a\geq1\) since \(n-1\) is even. 

Since \(m=2^bq\mid 2n(n-1)=2^{a+1}np\) and \(n\), \(p\), and \(q\) are odd, we have \(b\le a+1\). Since \(m\nmid n(n-1)=2^anp\), we must have \(b>a\). Hence \(b=a+1\), and thus \(m=2^{a+1}q\).

Let \( d = \gcd(m, n - 1) \). Since \( m = 2^{a+1} q\) and \( n - 1 = 2^a p \), it follows that \( 2^a \mid d \).  Since $a\geq 1$,  \( d \) is even. Now, let \( m = m' d \) and \( n - 1 = n' d \). We see that \( \gcd(n', m') = 1 \), \( n' \) is odd, and \( m' \equiv 2 \pmod 4 \).

Let \( r = n - m \). Since \( n \) is odd and \( m \) is even, we have that \( r \) is odd. If \( r = 1 \), then \( n = m + 1 \), which implies that \( m \mid n(n - 1) \), leading to a contradiction. Thus, 
\begin{equation}
r > 1.
\label{eq:greater_than_one}
\end{equation}
Now, partition the \( n-1 \) vertices in \( \{x_i : i \in \mathbb{Z}_{n-1}\} \) into \( d \) segments, each containing \( n' \) consecutive vertices. Each segment will contribute \( m' \) edges toward a peripheral  \( m \)-cycle. 

If \( 4K_n^{\bullet} \) is \((C_m)\)-decomposable, then the number of \( m \)-cycles in the decomposition is
\allowdisplaybreaks
\begin{align*}
\frac{2n(n-1)}{m} &= \frac{2(m+r)(n-1)}{m} = 2(n-1) + \frac{2r(n-1)}{m} \\
&= 2(n-1) + \frac{2rn'd}{m'd} = 4\left(\frac{n-1}{2}\right) + \frac{2r}{m'}n'= 4\left(\frac{n-1}{2}\right) + Fn',
\end{align*}
where $F'=\frac{2r}{m'}$. Since  \( \gcd(n', m') = 1 \), $r$ is odd, and \( m' \equiv 2 \pmod 4 \), it is easy to see that $F$ is an odd integer. 
The number of \( m \)-cycles suggests constructing four starter central cycles and \( F \) starter peripheral cycles.  The four starter central cycles will be rotated through \( \frac{n-1}{2} \) positions, while each of the \( F \) starter peripheral cycles will be rotated through \( n' \) positions to generate all cycles in the decomposition.

Before proceeding further, we establish some inequalities to be used later. 

\noindent {\bf{Bounds on $n'$ and $F$:}}

From (\ref{eq:greater_than_one}), we know $r > 1$, which implies $m < n - 1$. Since $n-1 < 2m$, we have $m < n-1 < 2m$, and thus $m'd < n'd < 2m'd$. Hence:
\begin{equation}
m' < n' < 2m'. \label{eq:greater_than_one-m'}
\end{equation} 
Since $n < 2m$ implies $r = n - m < m$, it follows that $\frac{2r}{m'} < \frac{2m}{m'}$, yielding:
\begin{equation}
F \leq 2d - 1. \label{eq:bound-F}
\end{equation}
Moreover, since $d$ is even, 
\begin{equation}
\text{if } F \equiv 1\pmod{4}, \text{ then } F \leq 2d - 3. \label{eq:newineq}
\end{equation}

\noindent {\bf{Analysis for $F = 2d-1$:}}

Since $F = \frac{2(n-m)}{m'}$ and $m=m'd$, we have $n = \frac{m'(F+2d)}{2}$, which gives $n' = \frac{n-1}{d} = \frac{m'(F+2d)-2}{2d}$. Given that $m' \equiv 2 \pmod 4$ and $d$ is even:
\begin{equation} \label{eq:n-prime}
\text{if } F = 2d-1, \text{ then } n' = 2m' - \frac{m' + 2}{2d}.
\end{equation}
For $m'>2$, the inequality $\frac{m'+2}{m'-2} < d$ holds except when $(m', d) = (6, 2)$. Thus, $\frac{m'+2}{d} < m'-2$, implying $2m' - \frac{m' + 2}{2d} > 2m' - \frac{m' - 2}{2}$, which yields:
\begin{equation}\label{eq:1nprimebiger}
n' > \frac{3m' + 2}{2}, \quad \text{provided } m'>2, \, (m', d) \neq (6, 2), \, \text{and } F = 2d-1.
\end{equation}

\noindent {\bf{Analysis for $F = 2d-3$:}}

\begin{equation} \label{eq:n-primenew}
\text{If } F = 2d-3, \text{ then } n' = 2m' - \frac{3m' + 2}{2d}.
\end{equation} 
For $m'>2$ and $d \geq 4$, the inequality $\frac{3m'+2}{m'-2} < d$ holds except when $(m', d) \in \{(6, 4), (10, 4)\}$. Consequently, $\frac{3m'+2}{d} < m'-2$, thus $2m' - \frac{3m' + 2}{2d} > 2m' - \frac{m' - 2}{2}$, yielding:
\begin{equation} \label{eq:n'>}
n' > \frac{3m' + 2}{2}, \quad \text{provided } m'>2, \, d \geq 4, \, (m', d) \notin \{(6, 4), (10, 4)\}, \, \text{and } F = 2d-3.
\end{equation}

With the necessary inequalities established, we now begin the construction.
We have seen that \( F \) is odd, so we have two cases to consider:   \( F \equiv 1 \pmod{4} \) and   \( F \equiv 3 \pmod{4} \).

\vspace{0.3cm}

\noindent
\textbf{\boldmath Case 1: \( F \equiv 3 \pmod{4} \).} \label{case1-ref}
The approach is to first construct \(\frac{F-3}{4} + 1\) starter peripheral cycles for $K_n$. 
Then, take four copies of each of the \(\frac{F-3}{4}\)  cycles and color them appropriately to obtain $F-3$ starter peripheral  cycles in $4K^{\bullet}_n$. We will then be left with one starter peripheral  cycle in $K_n$; we take three copies of it to generate the remaining three starter peripheral cycles in $4K^{\bullet}_n$. 

We begin by constructing the starter peripheral cycles. Define  paths \( P_i \) as follows.
\begin{itemize}
\item  Let
$P_0=x_0\ x_{1}\ x_{-1}\ x_{2}\ x_{-2} \dots x_{\frac{m'-2}{2}}\ x_{-\frac{m'-2}{2}}\ x_{\frac{m'}{2}}\ x_{-n'}.$
The differences covered by \( P_0 \) are:
\begin{equation}
1,2,3,4,\ldots, m'-2, m'-1, s, \label{eq:1odd1}
\end{equation}
where $s=n'+\frac{m'}{2}$ if $d\geq 4$ and $s=n'-\frac{m'}{2}$ if $d=2$.
\item For \( i = 1, 2, \ldots, \lceil\frac{F-3}{8}\rceil \), define 
\begin{align*}
P_{2i-1} &= x_0\ x_{(2i-1)n' + 1}\ x_{-1}\ x_{(2i-1)n' + 2}\ x_{-2}\ \dots\ \ x_{-\frac{m' - 6}{4}} \ x_{(2i-1)n' + \frac{m' - 2}{4}} \ x_{-\frac{m' - 2}{4}} \\
&\quad x_{(2i-1)n' + \frac{m' + 6}{4}}\ x_{-\frac{m' + 2}{4}} \ \dots\ x_{-\frac{m' - 2}{2}}\ x_{(2i-1)n' + \frac{m' + 2}{2}} \ x_{-n'}.
\end{align*}
The differences covered by \( P_{2i-1}  \) are:
\begin{align}
& (2i-1)n' + 1, (2i-1)n' + 2, (2i-1)n' + 3, \ldots, (2i-1)n' + \tfrac{m' - 4}{2}, (2i-1)n' + \tfrac{m' - 2}{2}, \notag \\
& (2i-1)n' + \tfrac{m' + 2}{2}, (2i-1)n' + \tfrac{m' + 4}{2}, \ldots, (2i-1)n' + m', 2in' + \tfrac{m' + 2}{2}. \label{eq:1odd2}
\end{align}
\item For \( i = 1, 2, \ldots, \lfloor\frac{F-3}{8}\rfloor\), define 
\begin{align*}
P_{2i} &= x_0\ x_{2in' + 1}\ x_{-1}\ x_{2in' + 2}\ x_{-2}\ \dots\ x_{2in' + \frac{m' - 2}{4}}\ x_{-\frac{m' - 2}{4}} \\
&\quad x_{2in' + \frac{m' + 2}{4}}\ x_{-\frac{m' + 6}{4}}\ x_{2in' + \frac{m' + 6}{4}}\ \dots\ x_{-\frac{m'}{2}}\ x_{2in' + \frac{m'}{2}}\ x_{-n'}.
\end{align*}
The differences covered by \( P_{2i}  \) are:
\begin{align}
&\textstyle{2in' + 1, \ 2in' + 2, \ 2in' + 3, \ldots, 2in' + \frac{m'-2}{2}, 2in' + \frac{m'}{2},} \notag \\
&\textstyle{2in' + \frac{m' + 4}{2}, 2in' + \frac{m' + 6}{2} \ldots, 2in' + m', \ (2i + 1)n' + \frac{m'}{2}}. \label{eq:1odd3}
\end{align}
\end{itemize}

Since \( m' < n' \) by (\ref{eq:greater_than_one-m'}), the paths \( \rho^{j}(P_i) \), for \( j = 0, n', 2n', \ldots, (d - 1)n' \), are pairwise vertex-disjoint except at their endpoints. Thus, each path $P_i$ generates an $m$-cycle:
\[
C_i=P_i \cup \rho^{n'}(P_i) \cup \ldots \cup \rho^{(d-1)n'}(P_i).
\]
Note that the differences \((2i + 1)n' + \frac{m'}{2}\) occur at the ends of the paths $P_{2i}$, and these differences are  avoided in the paths $P_{2i-1}$. Similarly, the differences \(2in' + \frac{m' + 2}{2}\) appear at the ends of the paths $P_{2i-1}$ and are  avoided in the paths $P_{2i}$. We now show that no difference occurs in more than one $P_i$.

{\bf{First, assume \(\mathbf{d =2}\).}} By~(\ref{eq:bound-F}), we have \( F \leq 2d - 1 \), which implies that \( F = 3 \). Thus, there is only one path, \( P_0 \), and it covers differences $1, 2, 3, \ldots, m' - 2, m' - 1, n' - \frac{m'}{2}.$
By~(\ref{eq:1nprimebiger}) unless $m'=2$ or $(m', d) = (6, 2)$, we have \( n' > \frac{3m' + 2}{2} \), which implies \( n' - \frac{m'}{2} > m' + 1 \). Therefore,
\begin{equation} \label{ineq:chainx}
m' - 1 < m' < n' - \frac{m'}{2}.
\end{equation}
Note that for \( (m', d) = (6, 2) \), we have \( n' = 10 \) by~\eqref{eq:n-prime}. However, since \( n' \) must be odd, this case does not arise.
For the case \( m' = 2 \), using inequality \eqref{eq:greater_than_one-m'},  we obtain \( n' = 3 \). In this case, \( m = 4 \), \( n = 7 \), and \( F = 3 \). There is only one path, \( P_0 = x_0\, x_1\, x_3 \), and the covered differences are 1 and 2, which are distinct.

{\bf{Now, assume \(\mathbf{d \geq 4}\).}} The paths \( P_i \) jointly cover the following differences:
\begin{itemize}
    \item For \( i = 0: \quad\quad  1, 2, \ldots, m' - 2, m' - 1 \)

    \item For \( i = 1, 2, \ldots, \frac{F - 3}{4}:  \quad\quad    in' + 1,\, in' + 2,\, \ldots,\, in' + m'\)

    \item One additional difference, depending on the parity of \( \frac{F - 3}{4} \), which corresponds to
    \[
    \left( \frac{F - 3}{4} + 1 \right)n' + \frac{m'}{2}
    \quad \text{or} \quad
    \left( \frac{F - 3}{4} + 1 \right)n' + \frac{m' + 2}{2}.
    \]
\end{itemize}

\noindent
To show that the above differences are pairwise distinct, it suffices to prove the following inequalities. 

\begin{enumerate}[\bf(i)]
\item \( \frac{F - 3}{4}n' + m' < \frac{n-1}{2} \). 

By (\ref{eq:bound-F}), we have \( F \leq 2d - 1 \). Then, by (\ref{eq:greater_than_one-m'}), 
\(
\frac{F - 3}{4}n' + m' < \frac{d - 2}{2}n' + n'  = \frac{n - 1}{2}.
\)

\item \(\frac{F + 1}{4}n' + \frac{m' + 2}{2} < \frac{n - 1}{2}\) when \( F < 2d - 1 \).

Since  \( F \leq 2d - 5 \), by (\ref{eq:greater_than_one-m'}), we have
\(
\frac{F + 1}{4}n' + \frac{m' + 2}{2} < \frac{d - 2}{2}n' + n'  = \frac{n - 1}{2}.
\)

\item \(\frac{F - 3}{4}n' + m' < \frac{F + 1}{4}n' - \frac{m' + 2}{2} \) when \( F =2d - 1 \). 

We have $\frac{F + 1}{4}n' = \frac{n - 1}{2}$, and the largest difference covered by $P_{\frac{F-3}{4}}$ is either $\frac{F + 1}{4}n' - \frac{m'}{2}$ or $\frac{F + 1}{4}n' - \frac{m' + 2}{2}$, depending on the parity of $\frac{F - 3}{4}$. By~\eqref{eq:1nprimebiger}, $n' > \frac{3m' + 2}{2}$ as $d \geq 4$, which is equivalent to 
\(
\frac{F - 3}{4}n' + m' < \frac{F + 1}{4}n' - \frac{m' + 2}{2}.
\)

\end{enumerate}
%
%
Thus, the differences covered by the \( m \)-cycles \( C_0, C_1, \ldots, C_{\frac{F-3}{4}} \), listed in (\ref{eq:1odd1}), (\ref{eq:1odd2}), and (\ref{eq:1odd3}), are pairwise distinct.
Let \( B \) be the set of differences listed in (\ref{eq:1odd2}) and (\ref{eq:1odd3}) that are covered by the \( m \)-cycles \( C_1, C_2, \ldots, C_{\frac{F-3}{4}} \). 
Let \(G_1=\text{Circ}(n-1;\pm B)\). Then $\{C_i,\rho(C_i),\ldots,\rho^{n'-1}(C_i)\mid i=1,\ldots,{\textstyle\frac{F-3}{4}}\}$ is a \((C_m)\)-decomposition of \(G_1\). By Lemma~\ref{lem:Gtool4}, there exists an HOP \((C_m)\)-decomposition of \(4G_1^{\bullet}\) generated by \(F-3\) starter peripheral cycles obtained from four copies of the \(\frac{F-3}{4}\) cycles of \(G_1\).


Recall that $A=\{1, 2, 3, \ldots, m'-2, m'-1, s\}$ is the set of differences covered by \( C_0 \). Take three copies of \( C_0 \) and label them as \( C'_0 \), \(  C'_1\) and \( C'_2 \); these are the remaining three starter peripheral cycles for \( 4K_n^{\bullet}\). Color the edges of the cycles \(C'_1\) and \(C'_2\) black, and orient them in the opposite directions.  Color the edges of \(P_0\) in the cycle  \(C'_0\) alternately pink and blue (as shown below), starting with pink. Since \(P_0\) has length \(m'\), and \( m' \equiv 2 \pmod 4 \), it ends with blue.
\begin{figure}[H]
\centering
   \includegraphics[width=0.9\linewidth,height=0.9\textheight,keepaspectratio]{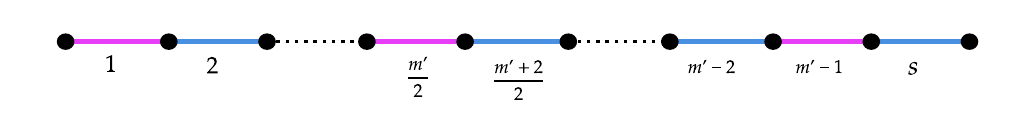}
   \vspace{-15pt}
\end{figure}
Observe that the families generated by the starter cycles \( C'_1 \) and \( C'_2 \) jointly cover the black orbits corresponding to the differences in the set \( A \). The family generated by the starter cycle \( C'_0 \) covers:
\begin{enumerate}[(i)]
    \item The pink orbits corresponding to the differences
    \begin{equation}
    \textstyle{1, 3, 5, \ldots, \frac{m'}{2}, \frac{m'+4}{2}, \ldots, m' - 3, m' - 1}.
    \label{eq:2pink_differences_odd}
    \end{equation}
    \item The blue orbits corresponding to the differences
    \begin{equation}
    \textstyle{2, 4, 6, \ldots, \frac{m'-2}{2}, \frac{m'+2}{2}, \ldots, m' - 2, s}.
    \label{eq:2blue_differences_odd}
    \end{equation}
\end{enumerate}
%
We now need to cover each difference exactly one more time using the central cycles \( C \) and \( C' \), to be constructed below. The differences in the set \( A \) are carefully divided into two parts, with each half covered by one of the central cycles and assigned the proper coloring, such that the blue orbits corresponding to the differences listed in~(\ref{eq:2pink_differences_odd}) and the pink orbits corresponding to the differences listed in~(\ref{eq:2blue_differences_odd}) are covered.

Let \(\mathscrsfs{L} \) be the set of unused differences in \(S\), that is, \(\mathscrsfs{L}  = S \setminus (A \cup B)=\{a_1, a_2, \ldots, a_t\} \cup \{\infty\}\), where \(a_1=m' < a_2 < \ldots < a_t = \frac{n-1}{2}\)  and \( \textstyle{t = \frac{n-1}{2} - \frac{(F+1) m'}{4}}\). 
Observe that  $a_1 < s < a_t$.

\noindent
Consider the following sets of differences:  
\begin{itemize}
\item \( W_1=\{ a_1, a_2, a_3, \ldots, a_{t-2}, a_{t-1}, a_t\}\)
\item \( W_2=\{ 2, 5, 6, \ldots, m'-9, m'-8, m'-5, m'-4, m'-1, a_1, a_2, \ldots, s, \ldots, a_{t-1}\} \)
\end{itemize}
Note that $s\notin W_1$,  \(|W_1| = t\) and \(|W_2| = t + \frac{m'-2}{2}\).  By Lemma \ref{zig-zag}(i), there exists a path  \( T_1 \) of length \( 2t + \frac{m'-2}{2} \) that starts at \( x_0 \) and covers the differences in $W_1$ and $W_2$ in the following order:
\[
T_1: a_1, a_2\ldots, a_{t-1}, a_t, a_{t-1}, \ldots, s, \ldots, a_1, m'-1, m'-4, m'-5, \ldots, 6, 5, 2.
\]
Now, consider the following set of differences: 
\begin{itemize}
\item \( W_2'=\{ 3,4, \ldots, m'-10, m'-7, m'-6, m'-3, m'-2, a_1,\ldots, s, \ldots, a_{t-1}\} \)
\end{itemize}
Here, \(|W_2'| = t + \frac{m'-2}{2}\).  By Lemma \ref{zig-zag}(i), there exists a path of length \( 2t + \frac{m'-2}{2}\) that starts at \( x_0 \) and covers the differences  in $W_1$ and $W_2'$ in the following order:
\[
a_1, \ldots, a_{t-1}, a_t, a_{t-1}, \ldots, s, \ldots, a_1, m'-2, m'-3, m'-6, m'-7,\ldots, 4, 3.
\]
By Lemma \ref{zig-zag}(ii), the difference \( s \) in the above path can be replaced with difference \( 1 \); that is, there exists a path  \( T_2 \) that covers the differences in the following order:
\[
T_2: a_1, \ldots, a_{t-1}, a_t, a_{t-1}, \ldots, 1, \ldots, a_1, m'-2, m'-3, m'-6, m'-7,\ldots, 4, 3.
\]
Moreover, by Lemma~\ref{zig-zag}(ii), the paths \( T_1 \) and \( T_2 \) can be constructed to be identical up to the occurrence of the difference \( s \) in \( T_1 \) and the difference \( 1 \) in \( T_2 \).
Note that since $t = \frac{n-1}{2} - \frac{(F+1)m'}{4}$ and $F = \frac{2(n-m)}{m'}$, the length of each of $T_1$ and $T_2$ is $2t + \frac{m'-2}{2} = m - 2$.

To simplify the explanation of the coloring process, we express $T_1$ and $T_2$ as the concatenations $T_1 = PQ_1R_1$ and $T_2 = PQ_2R_2$. Here, \( P \), \( Q_1 \), \( Q_2 \), \( R_1 \), and \( R_2 \) represent subpaths that cover the following sequences of differences: \label{case1x}
\begin{itemize}
\item \( P: a_1, a_2, a_3, \ldots, a_{t-2}, a_{t-1}, a_t\)
\item \( Q_1: a_{t-1}, a_{t-2}, \ldots, s, \ldots, a_3, a_2, a_1 \)
\item \( Q_2: a_{t-1}, a_{t-2}, \ldots, 1, \ldots, a_3, a_2, a_1 \)
\item \( R_1: m'-1, m'-4, m'-5, m'-8, m'-9, \ldots, 6, 5, 2 \)
\item \( R_2: m'-2, m'-3, m'-6, m'-7, m'-10, m'-11, \ldots, 4, 3 \)
\end{itemize}
The difference \( a_t = \frac{n-1}{2} \) occurs only in \( P \), and importantly, \( P \) is identical in both \( T_1 \) and \( T_2 \).
Also, notice that  \( P \), \( Q_1 \), and \( Q_2 \) each have length \( t  \). 
 The paths \( R_1 \) and \( R_2 \) each have length \( \textstyle{\frac{m'-2}{2}} \). 
We now describe the coloring.

Let \( C=x_{\infty} PQ_1R_1 x_{\infty} \) and \( C'=x_{\infty} PQ_2R_2 x_{\infty}\). 
First, color the edges in the path \( R_1 \) alternately pink and blue, starting with blue. Since its length is \( \textstyle{\frac{m'-2}{2}} \), which is even, it ends with pink. Similarly, color the edges in the path \( R_2 \) alternately pink and blue, starting with pink and ending with blue.

Next, in the path \( Q_1 \), color the edge with difference \( s \) pink, and in the path \( Q_2 \), color the edge with difference \( 1 \) blue, as shown below.
 With this coloring, we ensure that the blue orbits corresponding to the differences listed in (\ref{eq:2pink_differences_odd}) and the pink orbits corresponding to the differences listed in (\ref{eq:2blue_differences_odd}) are covered. 
\vspace{-13pt}
\begin{figure}[H]
\centering
   \includegraphics[width=0.9\linewidth,height=0.9\textheight,keepaspectratio]{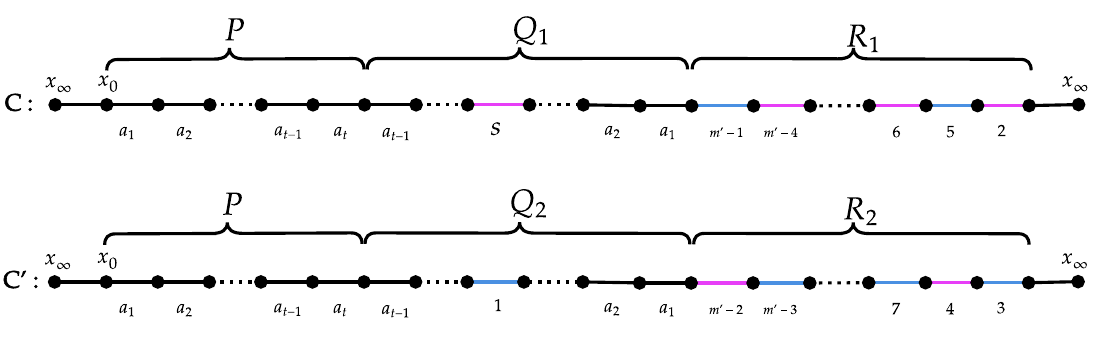}
   \vspace{-13pt}
\end{figure} 

Let \( \omega \) be the set of differences in \( Q_1 \) that fall within the interval \( \left[m', s\right) \), and recall that \( a_1 = m' \),  
 \( s = n' - \frac{m'}{2} \) if \( d = 2 \), and \( s = n' + \frac{m'}{2} \) otherwise. 
 In either case, it can be verified that \(|\omega|\) is even.
We start by coloring the edge with difference \(a_1\) pink in \(Q_1\) and blue in \(Q_2\), then continue coloring the remaining edges with differences in $\omega$ alternately pink and blue, as below.
\begin{figure}[H]
\centering
   \includegraphics[width=0.9\linewidth,height=0.9\textheight,keepaspectratio]{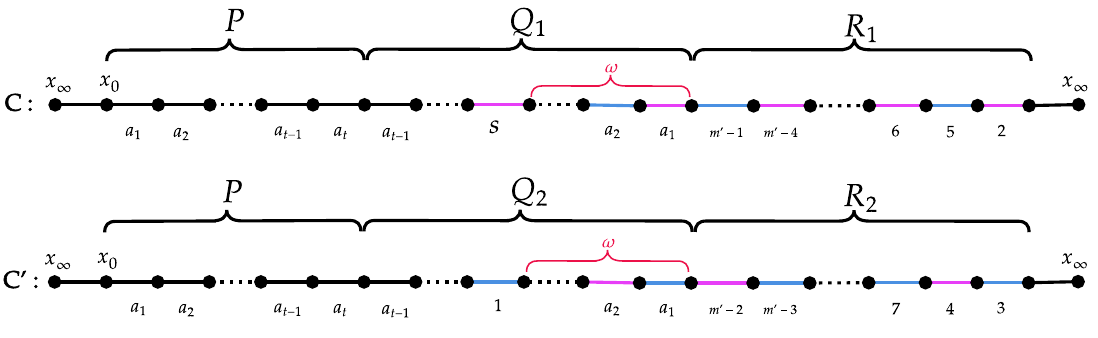}
    \vspace{-13pt}
\end{figure} 

Let \( C_c \) be a copy of \( C \), and let \( C'_c \) be a copy of \( C' \) with the same coloring as described above. Note that the number of uncolored edges in each of \( Q_1 \) and \( Q_2 \) is \( t - (|\omega| + 1) \). Depending on the parity of \( t \), we proceed with the rest of the coloring as follows.

\begin{enumerate}[\bf(i)] \label{OD:i}
\item $t$ is odd. Then  \( t - (|\omega| + 1) \) is even. 

In \( C \) and \( C_c \), start by coloring the uncolored edges in \( Q_1 \) alternately pink and blue, beginning with the edge of difference \( a_{t-1} \) in pink. Additionally, in \( C_c \), color the edge with difference \( a_t \)  blue. Then, proceed to color the remaining edges as shown below.

\begin{figure}[H]
\centering
  \includegraphics[width=0.9\linewidth,height=0.9\textheight,keepaspectratio]{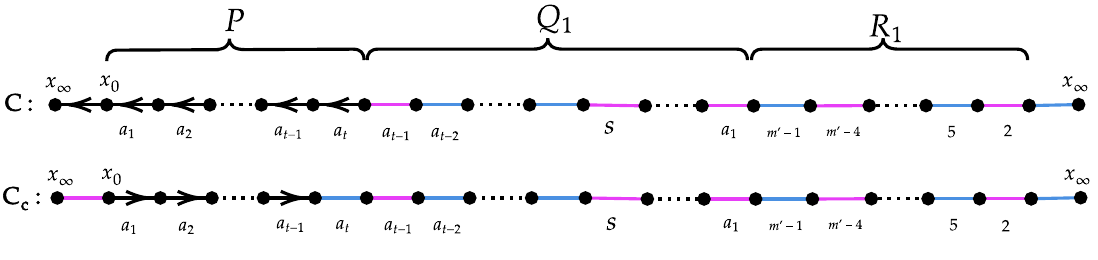}
   \vspace{-15pt}
\end{figure}

In \( C' \) and \( C'_c \), start by coloring the uncolored edges in \( Q_2 \) alternately pink and blue,  beginning with the edge of difference \( a_{t-1} \) in blue. Additionally, in \( C'_c \), color the edge with difference \( a_t \) pink. Then, color the remaining edges as shown below.

\begin{figure}[H]
\centering
 \includegraphics[width=0.9\linewidth,height=0.9\textheight,keepaspectratio]{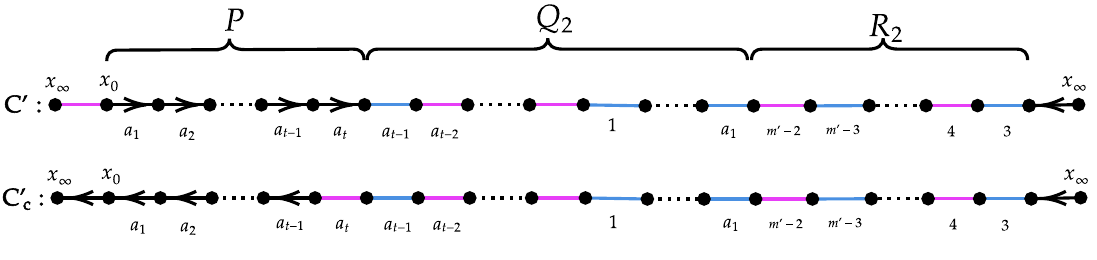}
  \vspace{-15pt}
\end{figure}

\item $t$ is even. Then  \( t - (|\omega| + 1) \) is odd. \label{EV:ii}

In \( C \) and \( C_c \), start by coloring the uncolored edges in \( Q_1 \) alternately pink and blue, beginning with the edge of difference \( a_{t-1} \) in blue. Additionally, in \( C_c \), color the edge with difference \( a_t \)  pink. Then, proceed to color the remaining edges as shown below.

\begin{figure}[H]
\centering
  \includegraphics[width=0.9\linewidth,height=0.9\textheight,keepaspectratio]{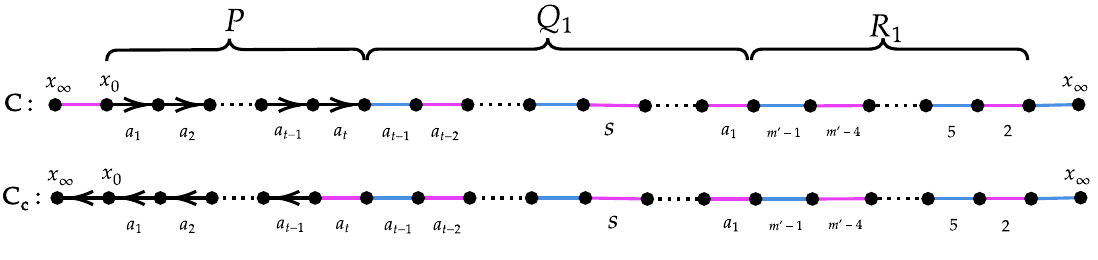}
   \vspace{-15pt}
\end{figure}

In \( C' \) and \( C'_c \), start by coloring the uncolored edges in \( Q_2 \) alternately pink and blue,  beginning with the edge of difference \( a_{t-1} \) in pink. Additionally, in \( C'_c \), color the edge with difference \( a_t \) blue. Then, color the remaining edges as shown below.

\begin{figure}[H]
\centering
 \includegraphics[width=0.9\linewidth,height=0.9\textheight,keepaspectratio]{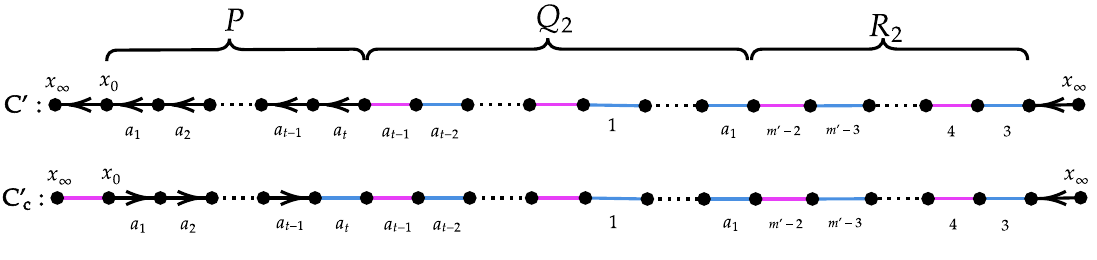}
  \vspace{-15pt}
\end{figure} 

\end{enumerate}

It can be verified  that
\begin{align*}
&\big\{\rho_{\bullet}^i(C), \rho_{\bullet}^{\frac{n-1}{2}+i}(C_c), \rho_{\bullet}^i(C'), \rho_{\bullet}^{\frac{n-1}{2}+i}(C'_c) \mid i=0,1,\ldots,\textstyle{\frac{n-3}{2}}\big\}  \cup \{C'_i, \rho_{\bullet}(C'_i), \ldots, \rho_{\bullet}^{n'-1}(C'_i) \mid i=0,1,2\}
\end{align*}
is an HOP $(C_m)$-decomposition  for $4G_2^{\bullet}$, where  \( G_2 = \text{Circ}(n-1; \pm(A\cup(\mathscrsfs{L}-\{\infty\}))) \bowtie K_1.\)

We have $4K_n^{\bullet}=4G_1^{\bullet}\oplus 4G_2^{\bullet}$, and since each of $4G_1^{\bullet}$ and $4G_2^{\bullet}$ admits an HOP $(C_m)$-decomposition, by Lemma \ref{Gtool3}, the multigraph $4K_n^{\bullet}$ admits an HOP $(C_m)$-decomposition. 
 
\vspace{0.2cm}

\noindent
\textbf{\boldmath Case 2: \( F \equiv 1 \pmod{4} \).}
Here, we divide the proof into three subcases:  \( d = 2 \);  \( d \geq 4 \) with \( (m, d) \neq (8, 4) \); and  \( (m, d) = (8, 4) \).

\textbf{Subcase 2.1:} \( d = 2 \). From (\ref{eq:newineq}), we have \( F \leq 2d - 3 \). Thus, \( F = 1 \), and we need to construct one family of peripheral cycles. Note that since \( d = 2 \), we have \( n' = \frac{n - 1}{2} \).
Define 
\begin{align*}
P_0&= x_0\ x_{1}\ x_{-1}\ x_{2}\ x_{-2} \dots x_{\frac{n'+1}{2}-\frac{m'+6}{4}}\ x_{-(\frac{n'+1}{2}-\frac{m'+6}{4})} \ x_{\frac{n'+1}{2}-\frac{m'+2}{4}}\\
 & x_{-(\frac{n'-1}{2}-\frac{m'-2}{4})} \ x_{\frac{n'+1}{2}-\frac{m'-2}{4}}\ \ldots\ x_{-\frac{m'-2}{2}}\ x_{\frac{m'-2}{2}}\ x_{-\frac{m'}{2}}\ x_{\frac{m'}{2}}\ x_{n'}.
\end{align*}
The differences covered by \( P_{0}  \) are:
\[
1,2,3,4,\ldots, n'-\frac{m'}{2}-2, n'-\frac{m'}{2}-1, n'-\frac{m'}{2}+1, n'-\frac{m'}{2}+2, \ldots, m'-2, m'-1, m', n'-\frac{m'}{2}.
\]
Since $n' - \frac{m'}{2}$ appears at the end of the path and $n' - \frac{m'}{2} < m'$ (as shown below), its occurrence was  avoided earlier in the path.

Since \( F = 1 \), we have \( 2(n - m) = m' \), which implies
\begin{equation}
n = m + \frac{m'}{2}.
\label{eqn:n+}
\end{equation}
Moreover, since \( d = 2 \), \( m = 2m' \),  we have \( n = \frac{5m'}{2} \), and we can write \( n' = \frac{n - 1}{2} = \frac{5m' - 2}{4} \).
 Therefore, we have
\begin{equation}
n' - \frac{m'}{2} = \frac{3m'}{4} - \frac{1}{2}  < m'.
\label{eq:n'<m'}
\end{equation}

Since \( m' < n' \) by (\ref{eq:greater_than_one-m'}), the paths \(P_0\) and \(\rho^{n'}(P_0)\) are pairwise vertex-disjoint except at the endpoints. Thus, \(C_0=P_0 \cup \rho^{n'}(P_0)\) is an \(m\)-cycle.
Moreover, having \( m' < n' \) implies \( \frac{m'}{2} < n'-\frac{m'}{2} \), and using the inequality in  (\ref{eq:n'<m'}) we have \( \frac{m'}{2} < n'-\frac{m'}{2} <m'.\)

Let \( A = \{1, 2, 3, \ldots,\frac{m'}{2}, \ldots, n'-\frac{m'}{2}, \ldots, m'-1, m'\} \) be the set of differences covered by \( C_0 \).  Color the edges of \( P_0 \) in the cycle \( C_0 \) alternately pink and blue (as shown below), starting with pink. Since the length of \( P_0 \) is \( m' \) and \( m' \equiv 2 \pmod{4} \), it ends with blue. Note that since \( n' \) is odd, \( n' - \frac{m'}{2} - 1 \) is also odd, and the edge with this difference gets pink color.

\begin{figure}[H]
\centering
   \includegraphics[width=0.9\linewidth,height=0.9\textheight,keepaspectratio]{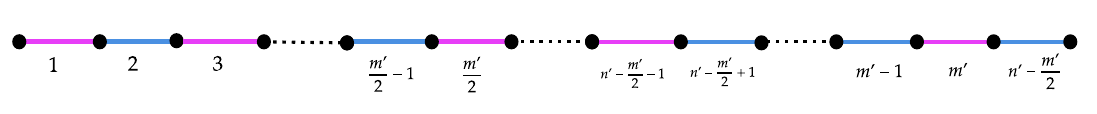}
   \vspace{-15pt}
\end{figure}

Observe that the family generated by the starter cycle \( C_0 \) covers:
\begin{enumerate}[(i)]
    \item The pink orbits corresponding to the differences
    \begin{equation}
    \textstyle{1, 3, 5, \ldots, \frac{m'}{2}, \ldots, n' - \frac{m'}{2}-1, n' - \frac{m'}{2}+2, \ldots, m' - 2, m'}.
    \label{eq:pink_differences_F=1}
    \end{equation}
    \item The blue orbits corresponding to the differences
    \begin{equation}
    \textstyle{2, 4, 6, \ldots, \frac{m'}{2}-1, \ldots,n' - \frac{m'}{2}-2, n' - \frac{m'}{2}+1, \ldots, m' - 1, n' - \frac{m'}{2}}.
    \label{eq:blue_differences_F=1}
    \end{equation}
\end{enumerate}
The differences in the set \( A \) will also appear in the central cycles \( C \) and \( C' \), to be constructed next.

Let \( \mathscrsfs{L} \) be the set of unused differences in \( S \); that is, \( \mathscrsfs{L} =  \{m'+1, m'+2, \ldots, n'\} \cup \{\infty\} \).
Consider the following sets of differences:  
\begin{itemize}
\item \( W_1=\{ 1,2,3,\ldots, n'-2, n'-1, n'\}\)
\item \( W_2=\{ 2, 3, 4, \ldots, \frac{m'}{2}-2, \frac{m'}{2}-1, \frac{m'}{2}, \textstyle{n'-\frac{m'}{2}}, m'+1, m'+2, \ldots, n'-1\} \)
\end{itemize}
The set \( W_1 \) contains all  differences from the set \( A \) and all differences from \( \mathscrsfs{L} \), except the difference \( \infty \), and has size \( |W_1| = n' \). The set \( W_2 \) contains half of the differences from \( A \) and all differences from \( \mathscrsfs{L} \), excluding differences \( \infty \) and \( n' \), and has size \( |W_2| = n' - \frac{m' + 2}{2} \).
 By Lemma \ref{zig-zag}(i), there exists a path  \( T_1 \) of length \( 2n' - \frac{m'+2}{2} \) that starts at \( x_0 \) and covers the differences in the following order:
\begin{align*}
T_1:  \ 1,2,3,\ldots, m'+1 ,\ldots, n'-1, n',  n' - 1, \ldots,  m' + 1, \textstyle{n' - \frac{m'}{2}, \frac{m'}{2}, \frac{m'}{2} - 1},  \ldots, 4, 3, 2.
\end{align*}
Now, consider the following set of differences:
\begin{itemize}
\item \( W_2'=\{ 1,  \frac{m'}{2}+1,  \frac{m'}{2}+2, \ldots, n' - \frac{m'}{2}-1, n' - \frac{m'}{2}+1,  \ldots, m'-2, m'-1, m', m'+1, \allowbreak m'+2, \ldots,n'-2,  n'-1\} \)
\end{itemize}
Note that $n' - \frac{m'}{2}\notin W_2'$, and \(|W_2'| = n' - \frac{m'+2}{2}\). 
By Lemma \ref{zig-zag}(i), there exists a path  \( T_2 \) of length \( 2n' - \frac{m'+2}{2} \) that starts at \( x_0 \) and covers the differences in $W_1$ and $W_2'$ in the following order:
\begin{align*}
T_2: \ & 1, 2, 3, \ldots, n' - 1, n', n' - 1, \ldots, m' + 2, m' + 1, m', m' - 1, \ldots, \\
& n' - \frac{m'}{2} + 1, n' - \frac{m'}{2} - 1, \ldots, \frac{m'}{2} + 2, \frac{m'}{2} + 1, 1.
\end{align*}
Moreover, using the proof of  Lemma~\ref{zig-zag}, the paths \( T_1 \) and \( T_2 \) can be constructed to be identical up to the occurrence of the difference \( m' + 1 \) in both \( T_1 \) and \( T_2 \). Recall $n' = \frac{n-1}{2}$ and applying (\ref{eqn:n+}), the lengths of $T_1$ and $T_2$ are $2n' - \frac{m'+2}{2} =   (m + \frac{m'}{2}) - 2 - \frac{m'}{2} = m - 2.$

As before, we express \( T_1 = PQR_1 \) and \( T_2 = PQR_2 \), where
\begin{itemize}
\item \( P: 1, 2, 3 , \ldots, n'-2, n'-1, n'\)
\item \( Q: n' - 1, n' - 2, \ldots, m' + 2, m' + 1 \)
\item \( R_1: n' - \frac{m'}{2}, \frac{m'}{2}, \frac{m'}{2} - 1, \frac{m'}{2} - 2, \ldots, 4, 3, 2 \)
\item \( R_2: m',m'-1, m'-2, \ldots, \frac{m'}{2}+2, \frac{m'}{2}+1, 1 \)
\end{itemize}
The diameter difference \( n' = \frac{n - 1}{2} \) occurs only in \( P \). Also, notice that \( P \) has length \( n' \), and \( Q \) has length \( n' - m' - 1 \). The paths \( R_1 \) and \( R_2 \) each have length \( \frac{m'}{2} \), and $n'-\frac{m'}{2}\notin R_2$. 

Color the edges in the path \( R_1 \) alternately pink and blue, starting with pink. Since its length is \( \frac{m'}{2} \), which is odd, it ends with pink. Similarly, color the edges in the path \( R_2 \) alternately pink and blue, starting with blue and ending with blue. Note that the difference \( n' - \frac{m'}{2} \) is covered by \( R_1 \) but not \( R_2 \).
%
With this coloring, we ensure that the blue orbits corresponding to the differences listed in (\ref{eq:pink_differences_F=1}) and the pink orbits corresponding to the differences listed in (\ref{eq:blue_differences_F=1}) are covered. 

Let \( C = x_{\infty} PQR_1 x_{\infty} \) and \( C' = x_{\infty} PQR_2 x_{\infty} \). 
We know that the length of \( Q \) is \( n' - m' - 1 \), which is even. Color the edges of \( Q \) in \( C \) alternately pink and blue, starting with pink and ending with blue, and oppositely in \( C' \), as shown below.
\begin{figure}[H]
\centering
   \includegraphics[width=0.9\linewidth,height=0.9\textheight,keepaspectratio]{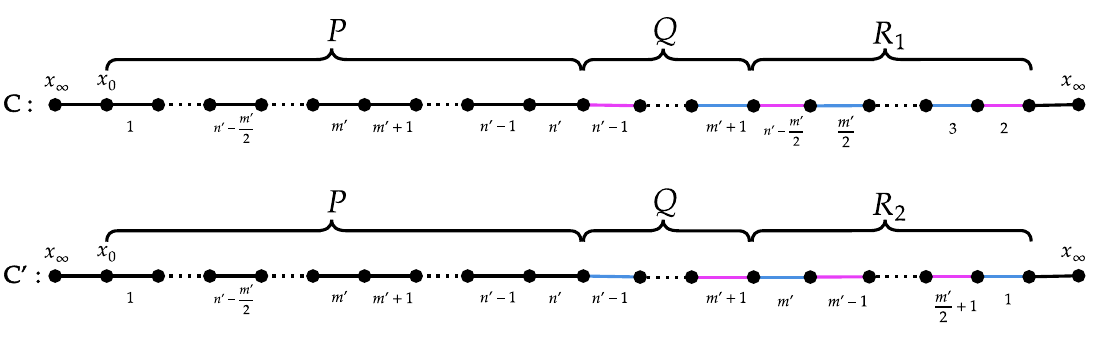}
     \vspace{-15pt}
\end{figure}

Let \( C_c \) be a copy of \( C \), and let \( C'_c \) be a copy of \( C' \) with the above coloring. In \( C_c \), color the edge with difference \( n' \) blue, and in \( C'_c \), color the edge with difference \( n' \) pink. Then, color the remaining edges as shown below. 
\begin{figure}[H]
\centering
   \includegraphics[width=0.9\linewidth,height=0.9\textheight,keepaspectratio]{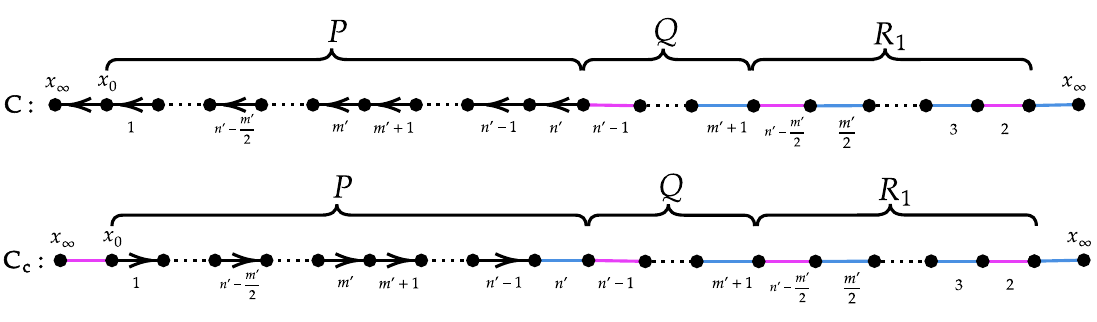}
     \vspace{-19pt}
\end{figure} 

\begin{figure}[H]
\centering
   \includegraphics[width=0.9\linewidth,height=0.9\textheight,keepaspectratio]{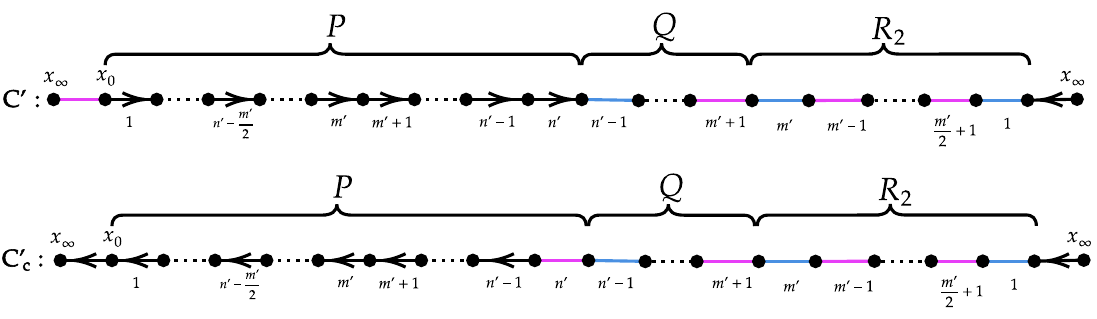}
     \vspace{-15pt}
\end{figure} 
Again, it can be verified that 
\begin{align*}
&\big\{\rho_{\bullet}^i(C), \rho_{\bullet}^{\frac{n-1}{2}+i}(C_c), \rho_{\bullet}^i(C'), \rho_{\bullet}^{\frac{n-1}{2}+i}(C'_c) \mid i=0,1,\ldots,\textstyle{\frac{n-3}{2}}\big\}   \cup \{C_0, \rho_{\bullet}(C_0), \ldots, \rho_{\bullet}^{n'-1}(C_0) \}
\end{align*}
is an HOP $(C_m)$-decomposition  for $4K_n^{\bullet}$. 


\textbf{Subcase 2.2:} $d\geq 4 \text{ and } (m, d)\neq(8, 4)$. 

Define paths $P_0$; $P_{2i-1} $ for $i = 1, 2, \ldots, \lceil\frac{F-1}{8}\rceil$; and $P_{2i-1}$ for $i = 1, 2, \ldots, \lfloor\frac{F-1}{8}\rfloor$ exactly as in Case~1 (page \pageref{case1-ref}).  
We see that $C_i=P_i \cup \rho^{n'}(P_i) \cup \ldots \cup \rho^{(d-1)n'}(P_i)$ is an \(m\)-cycle.
The paths \( P_i \)  jointly cover the following differences:
\begin{itemize}
    \item For \( i = 0 : \quad
    1, 2, \ldots, m' - 2, m' - 1
    \)

    \item For \( i = 1, 2, \ldots, \frac{F - 1}{4}: \quad
    in' + 1,\, in' + 2,\, \ldots,\, in' + m'
    \)

    \item One additional difference, depending on the parity of \( \frac{F - 1}{4} \), which corresponds to
    \[
    \left( \frac{F - 1}{4} + 1 \right)n' + \frac{m'}{2}
    \quad \text{or} \quad
    \left( \frac{F - 1}{4} + 1 \right)n' + \frac{m' + 2}{2}.
    \]
\end{itemize}
To show that the above differences are distinct, as in Case~1,  it suffices to prove the following inequalities. 
\begin{enumerate}[\bf(i)]
\item \( \frac{F - 1}{4}n' + m' < \frac{n-1}{2} \). 
This follows from (\ref{eq:newineq}) and (\ref{eq:greater_than_one-m'}).

\item \( \frac{F + 3}{4}n' + \frac{m' + 2}{2} < \frac{n - 1}{2} \) when $F < 2d - 3$. 
Since $F \equiv 1 \pmod{4}$, we have $F \leq 2d - 7$. It follows that $\frac{F + 3}{4} \leq \frac{d - 2}{2}$; combining this with $m' < n'$ from \eqref{eq:greater_than_one-m'} yields the inequality.

\item $\frac{F - 1}{4}n' + m' < \frac{F + 3}{4}n' - \frac{m' + 2}{2}$ when \( F = 2d - 3 \). 
Observe that
\(
\frac{F + 3}{4}n' + \frac{m' + 2}{2} = \frac{n - 1}{2} + \frac{m' + 2}{2}.
\)
By \eqref{eq:n'>},  for \( F = 2d - 3 \), all \( m' > 2 \), \( m' \equiv 2 \pmod{4} \), and even \( d \geq 4 \), except for \( (m', d) \in \{ (6, 4), (10, 4) \} \), we showed that \( n' > \frac{3m' + 2}{2} \). Hence, $\frac{F - 1}{4}n' + m'=\frac{F +3}{4}n'-n' + m' <\frac{F + 3}{4}n' - \frac{m' + 2}{2}$. 

It can be verified that cases with $m'=2$ or \( (m', d) \in \{ (6, 4), (10, 4) \} \) lead to a contradiction.
\end{enumerate}
%
%
Thus, we conclude that the differences covered by \( C_0, C_1, \ldots, C_{\frac{F-1}{4}} \),  listed in (\ref{eq:1odd1}), (\ref{eq:1odd2}), and (\ref{eq:1odd3}), are pairwise distinct.
%
%
Let \( B \) be the set of differences covered by the \( m \)-cycles \( C_1, C_2 \ldots, C_{\frac{F-1}{4}} \) listed in (\ref{eq:1odd2}), and (\ref{eq:1odd3}). 
Let \( G_1 = \text{Circ}(n-1; \pm B) \).  
Then,  \(\{C_i, \rho(C_i), \ldots, \rho^{n'-1}(C_i) \mid i=1, \ldots, \frac{F - 1}{4}\}\) is a \((C_m)\)-decomposition of \( G_1 \), and hence by Lemma~\ref{lem:Gtool4}, there exists an HOP \((C_m)\)-decomposition of \( 4G_1^{\bullet} \).

Recall that \( A = \{1, 2, 3, \ldots, m'-2, m'-1, n' + \frac{m'}{2}\} \) is the set of differences covered by \( C_0 \). Color the edges of \( P_0 \) in the cycle \( C_0 \) alternately pink and blue (as shown below), starting with pink and necessarily ending with blue. 
\begin{figure}[H]
\centering
   \includegraphics[width=0.9\linewidth,height=0.9\textheight,keepaspectratio]{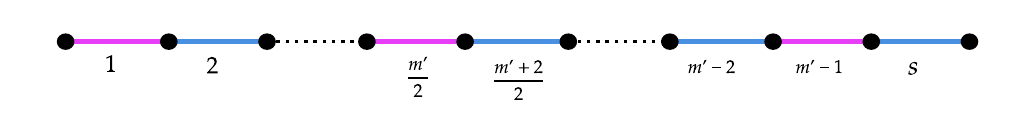}
   \vspace{-15pt}
\end{figure}

Observe that the family generated by the starter cycle \( C_0 \) covers:
\begin{enumerate}[(i)]
    \item The pink orbits corresponding to the differences
    \begin{equation}
    \textstyle{1, 3, 5, \ldots, \frac{m'}{2}, \frac{m'+4}{2}, \ldots, m' - 3, m' - 1}.
    \label{eq:5pink_differences_odd}
    \end{equation}
    \item The blue orbits corresponding to the differences
    \begin{equation}
    \textstyle{2, 4, 6, \ldots, \frac{m'-2}{2}, \frac{m'+2}{2}, \ldots, m' - 2, n' + \frac{m'}{2}}.
    \label{eq:5blue_differences_odd}
    \end{equation}
\end{enumerate}
Each difference in the set \( A \) will be covered  three more times using the starter central cycles \( C \) and \( C' \), to be constructed below. All differences in \( A \) will appear in both \( C \) and \( C' \), colored black and properly oriented, accounting for two additional occurrences and ensuring coverage of the black orbits. To cover the blue orbits corresponding to the differences listed in \eqref{eq:5pink_differences_odd}, and the pink orbits corresponding to the differences listed in \eqref{eq:5blue_differences_odd}, the set \( A \) is divided into two parts, each part covered by one of the starter central cycles.

Let \( \mathscrsfs{L} = S \setminus (A \cup B) = \{a_1, a_2, \ldots, a_t\} \cup \{\infty\} \), where \( m'=a_1 < a_2 < \ldots < a_t = \frac{n-1}{2} \) and \( t = \frac{n-1}{2} - \frac{(F+3)m'}{4} \).  
Since $d\geq 4$, we have $ m'-1<m' = a_1 < n' + \frac{m'}{2} < a_t. $

%
Recall that \( m' \equiv 2 \pmod{4} \). Now, consider the following sets of differences:  
\begin{itemize}
\item \( W_1=\{ 1,2,\ldots, m'-2, m'-1, a_1, a_2,\ldots, \textstyle{n'+\frac{m'}{2}}, \ldots, a_{t-2}, a_{t-1}, a_t\}\)
\item \( W_2=\{ 2, 5, 6, \ldots,  m'-8, m'-5, m'-4, m'-1, a_1, a_2, \ldots, \textstyle{n'+\frac{m'}{2}}, \ldots, a_{t-2}, a_{t-1}\} \)
\end{itemize}
%
The set \( W_1 \) contains all the differences from the set \( A \) and all the differences from \( \mathscrsfs{L} \), except the difference \( \infty \), and has size \( |W_1| = t+m'\). The set \( W_2 \) contains half of the differences from \( A \) and all differences from \( \mathscrsfs{L} \), excluding differences \( \infty \) and \( a_t \), and has size \( |W_2| = t -1+ \frac{m'}{2} \).  By Lemma \ref{zig-zag}(i), there exists a path  \( T_1 \) of length \( 2t + \frac{3m'-2}{2} \) that starts at \( x_0 \) and covers the differences in the following order:
\begin{align*}
T_1: & \ 1, 2, \ldots, m'-2, m'-1, a_1, a_2, \ldots, n' + \frac{m'}{2}, \ldots, a_{t-2}, a_{t-1}, a_t, \\
     &  a_{t-1}, a_{t-2}, \ldots, n' + \frac{m'}{2}, \ldots, a_2,  a_1, m'-1, m'-4, m'-5, m'-8, \ldots, 6, 5, 2.
\end{align*}
Now, consider the following set of differences: 
\begin{itemize}
\item \( W_2'=\{ 3,4, \ldots,  m'-7, m'-6, m'-3, m'-2, a_1, a_2, \ldots, \textstyle{n'+\frac{m'}{2}}, \ldots, a_{t-2}, a_{t-1}\} \)
\end{itemize}
Here,  \(|W_2'| = t + \frac{m'-2}{2}\).  By Lemma \ref{zig-zag}(i), there exists a path of length \( 2t + \frac{3m'-2}{2}\) that starts at \( x_0 \) and covers the differences in $W_1$ and $W_2'$ in the following order:
\begin{align*}
& 1, 2, \ldots, m'-2, m'-1, a_1, a_2, \ldots, n' + \frac{m'}{2}, \ldots, a_{t-2}, a_{t-1}, a_t, \\
& a_{t-1}, a_{t-2}, \ldots, n' + \frac{m'}{2}, \ldots, a_2, a_1, m'-2, m'-3, m'-6, m'-7, \ldots, 4, 3.
\end{align*}
By Lemma \ref{zig-zag}(ii), the second occurrence of difference \( n' + \frac{m'}{2} \) in the above path can be replaced with  difference \( 1 \); that is, there exists a path \( T_2 \) that covers the sequence of differences:
\begin{align*}
T_2: &\ 1, 2, \ldots, m'-2, m'-1, a_1, a_2, \ldots, n' + \frac{m'}{2}, \ldots, a_{t-2}, a_{t-1}, a_t, \\
& a_{t-1}, a_{t-2}, \ldots, 1, \ldots, a_2, a_1, m'-2, m'-3, m'-6, m'-7, \ldots, 4, 3.
\end{align*}
Notice that the lengths of $T_1$ and $T_2$ are $2t + \frac{3m'-2}{2} = m - 2$.

As in Case 1,  we express \( T_1 = PQ_1R_1 \) and \( T_2 = PQ_2R_2 \), where
\begin{itemize}
\item \( P: 1,2,\ldots, m'-2, m'-1, a_1, a_2,\ldots, \textstyle{n'+\frac{m'}{2}}, \ldots, a_{t-2}, a_{t-1}, a_t\)
\item \( Q_1: a_{t-1}, a_{t-2}, \ldots, \textstyle{n'+\frac{m'}{2}}, \ldots, a_3, a_2, a_1 \)
\item \( Q_2: a_{t-1}, a_{t-2}, \ldots, 1, \ldots, a_3, a_2, a_1 \)
\item \( R_1: m'-1, m'-4, m'-5, m'-8, m'-9, \ldots, 6, 5, 2 \)
\item \( R_2: m'-2, m'-3, m'-6, m'-7, m'-10, m'-11, \ldots, 4, 3 \)
\end{itemize}
The difference \( a_t = \frac{n-1}{2} \) occurs only in \( P \), and this path has length \( t + m' \). Importantly, \( P \) is identical in both \( T_1 \) and \( T_2 \).
The paths \( Q_1 \) and \( Q_2 \) each have length \( t \), and the paths \( R_1 \) and \( R_2 \) each have length \( \frac{m'-2}{2} \). 
Observe that, except for the subpath \( P \), all other subpaths are identical to those in Case~1 (see page~\pageref{case1x}) in terms of the sequence of differences they cover. The only distinction is that \( P \) is longer and includes differences from the set \( A \). The coloring process remains exactly the same as in Case~1. See figures on pages \pageref{OD:i} and \pageref{EV:ii}.

Analogous to Case 1, it can be verified that
\begin{align*}
&\big\{\rho_{\bullet}^i(C), \rho_{\bullet}^{\frac{n-1}{2}+i}(C_c), \rho_{\bullet}^i(C'), \rho_{\bullet}^{\frac{n-1}{2}+i}(C'_c) \mid i=0,1,\ldots,\textstyle{\frac{n-3}{2}}\big\} \cup \{C_0, \rho_{\bullet}(C_0), \ldots, \rho_{\bullet}^{n'-1}(C_0)\}
\end{align*}
is an HOP  $(C_m)$-decomposition  for $4G_2^{\bullet}$, where  \( G_2 = \text{Circ}(n-1; \pm(A\cup(\mathscrsfs{L}-\{\infty\}))) \bowtie K_1.\)

We have $4K_n^{\bullet}=4G_1^{\bullet}\oplus 4G_2^{\bullet}$, and since each of $4G_1^{\bullet}$ and $4G_2^{\bullet}$ admits an HOP $(C_m)$-decomposition, by Lemma \ref{Gtool3}, the multigraph $4K_n^{\bullet}$ admits an HOP $(C_m)$-decomposition.


\textbf{Subcase 2.3:} $(m, d)=(8, 4)$. 
In this case, \( n = 13 \), \( F = 5 \), \( n' = 3 \), \( m' = 2 \), and the set of differences is \( S = \{ 1,  2,  3,  4,  5, 6,  \infty\} \). 

Let \( P_0 = x_0\ x_1\ x_3 \), and \( P_1 = x_0\ x_4\ x_{-3} \). The path \( P_0 \) covers differences \( 1 \) and \( 2 \), and the path \( P_1 \) covers differences \( 4 \) and \( 5 \). 
Observe that, for \( i = 0, 1 \),
\(
C_i = P_i \cup \rho^{n'}(P_i) \cup \rho^{2n'}(P_i) \cup \rho^{3n'}(P_i),
\)
is a cycle of length 8.
Take four copies of \( C_1 \), and use proof of Lemma \ref{lem:Gtool4} to color them so that they satisfy Condition \textbf{(C1)} of Definition \ref{def}. By doing so, we generate four starter peripheral cycles of \( 4K_{13}^{\bullet} \) with the appropriate HOP coloring, which we label as \( C'_1, C'_2, C'_3, C'_4 \).
Note that $\{1, 2\}$ is the set of differences covered by $C_0$. In $P_0$, we color the edge with difference $1$ pink and the edge with difference $2$ blue. The black copies of the edges with differences $1$ and $2$, the blue copy of difference $1$, the pink copy of difference $2$, and the set of uncovered differences $\{3, 6, \infty\}$ are all covered by the four central cycles $C, C_c, C'$, and $C'_c$, as shown in Figure~\ref{fig:EXm8}.
\begin{figure}[h!]
\centering
   \includegraphics[width=0.7\linewidth,height=0.7\textheight,keepaspectratio]{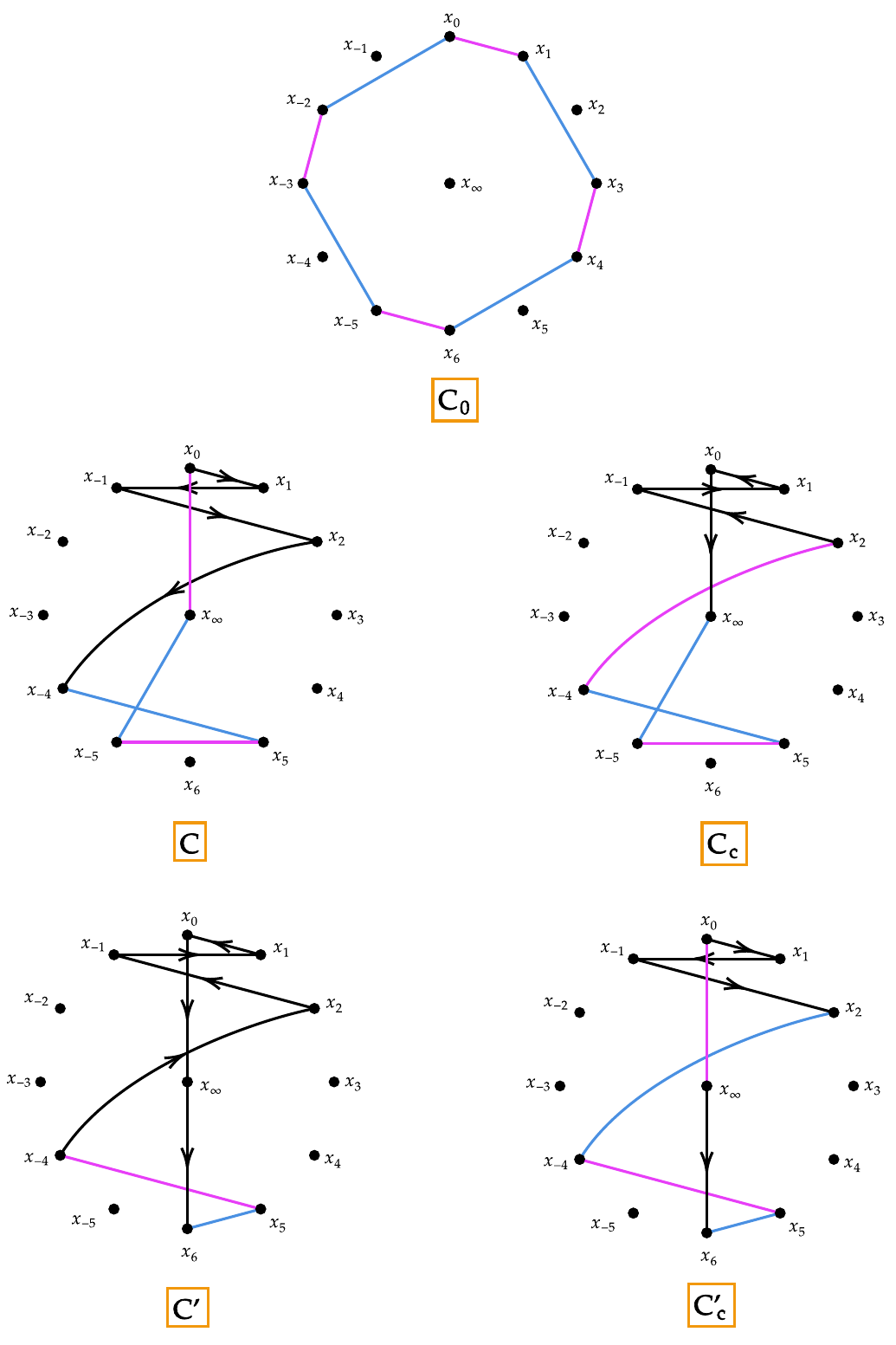}
   \caption{Lemma~\ref{lem:n-odd-4K-last}, Subcase 2.3 -- The starter peripheral cycle $C_0$ and the four starter central cycles for \( m = 8 \) and \( n = 13 \).}
   \label{fig:EXm8}
   \vspace{-10pt}
\end{figure}
We see that
\begin{align*}
& \{\rho_{\bullet}^i(C), \rho_{\bullet}^{6+i}(C_c), \rho_{\bullet}^i(C'), \rho_{\bullet}^{6+i}(C'_c) \mid i = 0, 1, \ldots, 5\} \cup \{C_0, \rho_{\bullet}(C_0), \rho_{\bullet}^2(C_0)\} \\
&\qquad \qquad \qquad \cup \{C'_i, \rho_{\bullet}(C'_i), \rho_{\bullet}^2(C'_i) \mid i = 1, 2, 3, 4\}
\end{align*}
 is an HOP  $(C_8)$-decomposition  of $4K_{13}^{\bullet}$.
\end{proof}

\subsection{The case when $n$ is even }


\begin{lem}{\label{lem:n-even-4K-last}}
Let $m\equiv 0\ ({\rm mod}\ 4)$, and let $n$ be an even positive integer such that $4\leq m\leq n$. If $2n(n-1)\equiv m\ ({\rm mod}\ 2m)$, then $4K_n^{\bullet}$ admits an HOP $(C_m)$-decomposition. 
\end{lem}
\begin{proof}
By Lemma~\ref{lem:n-even-general}, it suffices to prove this result for even \( n \) in the range \( m \leq n < 2m \). Moreover, since \( 2n(n - 1) \equiv m \pmod{2m} \), we have \( m \neq n \), so we restrict our attention to the range \( m < n < 2m \).
As before, the approach is to construct starter central and starter peripheral cycles.
We begin by outlining the parameters and then explain how many starter central and starter peripheral cycles are needed.

Let \( V(4K_n^{\bullet}) = \{x_i : i \in \mathbb{Z}_{n-1}\} \cup \{x_{\infty}\} \), and let $\rho_{\bullet}$ be the permutation on $E(4K_n^{\bullet})$ that preserves the color (and orientation) of the edges, and is induced by  the permutation  $\rho=(x_{\infty})(x_0 \ x_1\  x_2 \ \ldots \ x_{n-2})$. Observe that the group $\langle \rho_{\bullet} \rangle$ has the following orbits on the edge set of $4K_n^{\bullet}$:
\begin{itemize}
\item for each $s\in \{1,2,\ldots, \textstyle{\frac{n-2}{2}}\}$, we have a pink and a blue orbit $\{x_ix_{i+s}: i\in \ZZ_{n-1} \}$;
\item for each $s\in \{1,2,\ldots, n-2\}$, we have a black orbit $\{(x_i, x_{i+s}): i\in \ZZ_{n-1} \}$;
\item a pink and a blue orbit  $\{x_ix_{\infty}: i\in \ZZ_{n-1} \}$; and
\item black orbits $\{(x_{\infty}, x_{i}): i\in \ZZ_{n-1} \}$ and $\{(x_i, x_{\infty}): i\in \ZZ_{n-1} \}$.
\end{itemize}
For convenience, let \( S = \{ 1,  2, \ldots,  \frac{n-2}{2}, \infty\} \) be the set of all differences. 
Let \( n = 2^e a \), where \( a \) is odd, and let \( m = 2^d a' b' \), where \( a' \) and \( b' \) are both odd, with \( a' \mid a \) and \( b' \mid (n-1) \). Since \( m \equiv 0 \pmod{4} \), we have that \( d \geq 2 \). Since \( m \mid 2n(n-1) \) and \( m \nmid n(n-1) \), we have \( d = e + 1 \), and we can write \( n = 2^{d-1} a \).

To construct peripheral cycles, we partition the \( n-1 \) vertices in \( \{x_i : i \in \mathbb{Z}_{n-1}\} \) into \( b' \) segments, each containing \( \ell = \frac{n-1}{b'} \) consecutive vertices and contributing \( 2^d a' \) edges toward 
a peripheral \( m \)-cycle.

If \( 4K_n^{\bullet} \) is \((C_m)\)-decomposable, then the number of \( m \)-cycles in the decomposition is
\allowdisplaybreaks
\begin{align*}
\frac{2n(n-1)}{m} 
&= \frac{2(n-1)(m+n-m)}{m} = 2(n-1) + \frac{2(n-1)(n-m)}{2^da'b'}\\[1ex] 
&=  2(n-1) + \ell \cdot \frac{2(n-m)}{2^da'}=2(n-1)+\ell F,
\end{align*}
where $F= \frac{2(n - m)}{2^d a'} =  \frac{a - 2 a'b'}{a'}$ is an odd integer. This suggests constructing two starter central cycles and \( F \) starter peripheral cycles.

We have two cases to consider:  \( F \equiv 1 \pmod{4} \) and  \( F \equiv 3 \pmod{4} \). 
We first prove some inequalities that will be useful in both cases.

Since \( n < 2m \), we have  \( 2^{d-1}a < 2^{d+1} a' b' \), hence \(\frac{a}{a'} < 4b'\). Therefore, \(\frac{a - 2a' b'}{a'} < 2b'\), which means 
\begin{equation}
    F \leq 2b' - 1.
    \label{eq:F<2b'-1}
\end{equation}
Since \( F \leq 2b' - 1 \), it follows that \( b' \geq 3 \) whenever \( F \equiv 3 \pmod{4} \), and \( b' \geq 1 \) whenever \( F \equiv 1 \pmod{4} \). In the latter case, when \( b' = 1 \), we handle it separately.

Additionally, we have \( m < n - 1 < 2m \); that is, \( 2^d a'b' < b'\ell < 2^{d+1} a'b' \). Thus
\begin{equation}
    2^d a'  < \ell < 2^{d+1} a'.
    \label{eq:ell}
\end{equation}

Suppose $\ell - 2^{d-1}a' \leq 2^d a' - 1$. 
Multiplying this inequality by $b'$ and substituting $b'\ell = n-1$ gives $n - 1 \leq 3 \cdot 2^{d-1}a'b' - b'$. Thus $2n \leq 3 \cdot 2^d a'b' - 2b' + 2$, and  $2n \leq 3 \cdot 2^d a'b'$. Substituting $n = 2^{d-1}a$ and simplifying, we obtain $\frac{a}{a'} \leq 3b'$. It follows that $F=\frac{a - 2a'b'}{a'} \leq b'$,  thus
\begin{equation}
\text{if } \ell - 2^{d-1}a' \leq 2^d a' - 1, \text{ then } F \leq b'.
\label{eq:333}
\end{equation}
Now that we have established the necessary inequalities, we proceed with the constructions.
\vspace{0.2cm}

\noindent
\textbf{\boldmath Case 1: \( F \equiv 3 \pmod{4} \).}
The approach is similar to Case~1 of the proof of Lemma~\ref{lem:n-odd-4K-last}. 
We start by constructing the starter peripheral cycles.

\label{page:cas1L}

For \( i = 0, 1, \ldots, \frac{F-3}{4} \), define 
$$P_i=x_0\ x_{i\ell+1}\ x_{-1}\ x_{i\ell+2}\ x_{-2} \dots x_{i\ell+2^{d-1}a'-1}\ x_{-(2^{d-1}a'-1)}\ x_{i\ell+2^{d-1}a'}\ x_{\frac{b'+1}{2}\ell}.$$
The differences covered by \( P_0 \) correspond to:
\begin{equation}
1, \; 2, \;  3, \; \ldots, \;  2^d a' - 2, \;  2^d a' - 1, \; \textstyle{\frac{(b' - 1)\ell}{2} +2^{d-1}a'}, \label{eq:1-P0}
\end{equation}
and the differences covered by \( P_i \), for \( i = 1, \ldots, \frac{F - 3}{4} \), are:
\begin{equation}
i\ell + 1, \; i\ell + 2, \; i\ell + 3, \; \ldots, \; i\ell + 2^d a' - 2, \; i\ell + 2^d a' - 1, \; \textstyle{\frac{(b' - 2i + 1)\ell}{2} - 2^{d-1}a'}. \label{eq:1}
\end{equation}
Since \(\ell> 2^d a' \) by (\ref{eq:ell}), the paths \(\rho^{j}(P_i)\), for \(j = 0, \ell, 2\ell, \ldots, (b' - 1)\ell\), are pairwise vertex-disjoint except at the endpoints. Also, since \(b'\) is odd, we have \(\gcd\left(b', \frac{b' + 1}{2}\right) = 1\). Thus,
\(
C_i=P_i \cup \rho^{\ell}(P_i) \cup \ldots \cup \rho^{(b'-1)\ell}(P_i)
\)
is an \(m\)-cycle.

Next, we show that the differences covered by the \( m \)-cycles \( C_0, C_1, \ldots, C_{\frac{F-3}{4}} \), listed in (\ref{eq:1-P0}) and (\ref{eq:1}), are pairwise distinct.
To show \( i\ell + 1, i\ell + 2, \ldots, i\ell + 2^d a' - 1 \) are pairwise distinct for \( i = 0, 1, \ldots, \frac{F - 3}{4} \), it suffices to show that the largest element on this list does not exceed \(\frac{n-2}{2}\). Note that $F \leq 2b' - 1$ from \eqref{eq:F<2b'-1} implies $\frac{F - 3}{4} \ell \leq \frac{b' - 2}{2} \ell$. Using $2^d a' < \ell$ from \eqref{eq:ell}, we obtain:
\begin{align*}
\frac{F - 3}{4} \ell + 2^d a' - 1 
&\leq \frac{b' - 2}{2} \ell + 2^d a' - 1 < \frac{b' - 1}{2} \ell + 2^{d-1} a' < \frac{n - 1}{2}.
\end{align*}
The above inequality also ensures that for all $j = 0, 1, \ldots, \frac{F - 3}{4}$, 
\[
\frac{b' - 1}{2} \ell + 2^{d-1} a' \notin \{j\ell + 1, j\ell + 2, \ldots, j\ell + 2^d a' - 1\}.
\]
Next, we show that for all \( i= 1, \ldots, \frac{F-3}{4} \) and \( j = 0, 1, \ldots, \frac{F-3}{4} \), the difference 
\(
\frac{(b' - 2i + 1)\ell}{2} - 2^{d-1}a'
\)
does not appear among the differences 
\(
j\ell + 1, j\ell + 2, \ldots, j\ell + 2^d a' - 1.
\)
Since 
\[
\textstyle{
0<\frac{(b' - 2\cdot \frac{F - 3}{4} + 1)\ell}{2} - 2^{d-1}a'
< \frac{(b' - 2\cdot \frac{F - 7}{4} + 1)\ell}{2} - 2^{d-1}a'
< \ldots 
< \frac{(b' - 2\cdot 1 + 1)\ell}{2} - 2^{d-1}a'
< \frac{(b' -1)\ell}{2} +2^{d-1}a'.
}\]
It suffices to show that for all \( j = 0, 1, \ldots, \frac{F-3}{4} \) and \(1 \leq \alpha \leq 2^d a' - 1\)
\[
\frac{(b' - 2\cdot\frac{F-3}{4} + 1)\ell}{2} - 2^{d-1}a' \neq j\ell + \alpha. 
\]
We prove this by contradiction. Suppose to the contrary that $\frac{2b' -F+1}{4}\ell + \ell- 2^{d-1}a' = j\ell + \alpha.$
Since \( F \equiv 3 \pmod{4} \) and \( b' \) is odd, the factor \( \frac{2b' - F + 1}{4} \) is an integer, and  since \(\ell> 2^d a' \) by (\ref{eq:ell}), we have
\(
\frac{2b' -F+1}{4}=j  \text { and } \ell - 2^{d-1}a' = \alpha .
\)
Given that \(j \leq \frac{F - 3}{4}\), it follows that \(F \geq b' + 2\).
On the other hand, since \(1 \leq \alpha \leq 2^d a' - 1\) and \(\ell - 2^{d-1}a' = \alpha\), it follows that \(\ell - 2^{d-1}a' \leq 2^d a' - 1\), and by (\ref{eq:333}), we have \(F \leq b'\).
 Therefore, we have \(b' + 2 \leq F \leq b'\), which is a contradiction. 
 
We conclude that the differences listed in~(\ref{eq:1-P0}) and in~(\ref{eq:1}) are pairwise distinct.
 Let \( B \) be the set of differences covered by the \( m \)-cycles \( C_1, C_2 \ldots, C_{\frac{F-3}{4}} \).
Then,  \(\{C_i, \rho(C_i), \ldots, \rho^{\ell-1}(C_i) \mid i=1, \ldots, \frac{F - 3}{4}\}\) is a \((C_m)\)-decomposition of \( G_1 = \text{Circ}(n-1; \pm B) \), and hence by Lemma~\ref{lem:Gtool4}, there exists an HOP \((C_m)\)-decomposition of \( 4G_1^{\bullet} \).

Notice that $A=\{1, 2, \ldots,  2^d a' - 2, 2^d a' - 1, \textstyle{\frac{(b' - 1)\ell}{2} + 2^{d-1}a'}\}$ is the set of differences covered by \( C_0 \). Take three copies of \( C_0 \) and denote them \(  C'_0, C'_1\) and \( C'_2 \). These are the remaining three peripheral cycles. Color the edges of the cycles \(C'_1\) and \(C'_2\) black, and orient them in the opposite directions.  Color the edges of \(P_0\) in the cycle  \(C'_0\) alternately pink and blue (as shown below), starting with pink. Since \(P_0\) has length \(2^d a'\), which is even, it ends with blue.
\begin{figure}[H]
 \centering
   \includegraphics[width=0.9\linewidth,height=0.9\textheight,keepaspectratio]{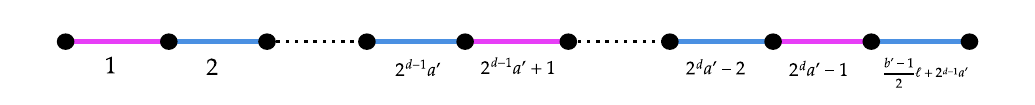}
   \vspace{-15pt}
  \end{figure}
The families of cycles generated by \( C'_0, C'_1 \) and \( C'_2 \)  jointly cover the black orbits corresponding to the differences in the set \( A \), and cover:
\begin{enumerate}[(i)]
    \item the pink orbits corresponding to the differences
    \begin{equation}
    1, 3, 5, \ldots, 2^{d-1}a' + 1, 2^{d-1}a' + 3, \ldots, 2^d a' - 3, 2^d a' - 1, \text{ and}
    \label{eq:pink_differences}
    \end{equation}
    
    \item the blue orbits corresponding to the differences
    \begin{equation}
    2, 4, 6, \ldots, 2^{d-1}a', 2^{d-1}a' + 2, \ldots, 2^d a' - 2, \frac{(b' - 1)\ell}{2} + 2^{d-1}a'.
    \label{eq:blue_differences}
    \end{equation}
\end{enumerate}
Now, we focus on constructing the starter central cycles \(C\) and \(C'\). Let \(\mathscrsfs{L}  = S \setminus (A \cup B)=\{a_1, a_2, \ldots, a_t\} \cup \{\infty\}\), where \(a_1 < a_2 < \ldots < a_t \leq \frac{n-2}{2}\)  and \(t = \frac{n-2}{2} - \frac{(F+1) 2^d a'}{4}\). 

Since \( F \leq 2b' - 1 \) by \eqref{eq:F<2b'-1}, and \( F \equiv 3 \pmod{4} \), we have \( b' \geq 3 \). Moreover, since difference \( \ell \) is not covered by any peripheral cycles, we observe that
\(
a_1\leq \ell < \frac{(b' - 1)\ell}{2} + 2^{d-1}a' \leq \frac{n-2}{2}.
\)
Consider the following sets of differences: 
\begin{itemize}
\item \( W_1=\{ a_1, a_2, a_3, \ldots, a_{t-2}, a_{t-1}, a_t\}\)
\item \( W_2=\{   2^{d-1}a'+1, 2^{d-1}a'+2,  \ldots, 2^d a'-1, a_1, a_2, \ldots, \textstyle{\frac{(b' - 1)\ell}{2} + 2^{d-1}a'}, \ldots, a_{t-1}, a_t\} \)
\end{itemize}
Note that \(\textstyle{\frac{(b' - 1)\ell}{2} + 2^{d-1}a'}\notin W_1\), and in the sequence \( a_1, a_2, \ldots, \frac{(b' - 1)\ell}{2} + 2^{d-1}a', \ldots,  a_t \), we assume
\[
a_1 < a_2 < \ldots < \frac{(b' - 1)\ell}{2} + 2^{d-1}a' < \ldots < a_{t-1} < a_t,
\]
with the understanding that
\(
a_1 < a_2 < \ldots < a_{t-1} < a_t < \frac{(b' - 1)\ell}{2} + 2^{d-1}a'
\)
is also possible.

Here,  \( |W_1| = t \), and \( W_2 \) contains half of the differences from the set \( A \), along with all the differences from \( \mathscrsfs{L} \) except \( \infty \), thus \( |W_2| = t + 2^{d-1}a' \).
   By Lemma \ref{zig-zag}(i), there exists a path  \( T_1 \) of length \( 2t + 2^{d-1} a' \) that starts at \( x_0 \) and covers the differences in $W_1$ and $W_2$ in the following order:
\[
T_1: a_1, \ldots, a_{t-1}, a_t, a_t, a_{t-1}, \ldots, \textstyle{\frac{(b' - 1)\ell}{2} + 2^{d-1}a'}, \ldots, a_1, 2^d a' - 1, \ldots, 2^{d-1}a' + 1.
\]
By Lemma \ref{zig-zag}(ii), the difference \( \frac{(b' -1)\ell}{2} + 2^{d-1}a' \) can be replaced with  difference \( 1 \); that is, there exists a path  \( T_2 \) that covers the differences in the following order:
\[
T_2: a_1, a_2, \ldots, a_{t-1}, a_t, a_t, a_{t-1}, \ldots, 1, \ldots, a_2, a_1, 2^d a' - 1, \ldots, 2^{d-1}a' + 1.
\]
Now, consider the following set of differences: 
\begin{itemize}
\item \( W_2'=\{   2,3,4, \ldots, 2^{d-1}a'-1, 2^{d-1}a', a_1, a_2, \ldots, \textstyle{\frac{(b' - 1)\ell}{2} + 2^{d-1}a'}, \ldots, a_{t-1}, a_t\} \)
\end{itemize}
Here,  \(|W_2'| = t + 2^{d-1}a'\).  By Lemma \ref{zig-zag}(i), there exists a path  \( T_3 \) of length \( 2t + 2^{d-1} a' \) that starts at \( x_0 \) and covers the differences in $W_1$ and $W_2'$ in the following order:
\[
T_3: a_1, \ldots, a_{t-1}, a_t, a_t, a_{t-1}, \ldots, \textstyle{\frac{(b' - 1)\ell}{2} + 2^{d-1}a'}, \ldots, a_1, 2^{d-1}a', 2^{d-1}a'-1, \ldots, 3, 2.
\]
By Lemma \ref{zig-zag}(ii), the difference \( \frac{(b' -1)\ell}{2} + 2^{d-1}a' \) can be replaced with  difference \( 1 \); that is, there exists a path  \( T_4 \) that covers the differences in the following order:
\[
T_4: a_1, \ldots, a_{t-1}, a_t, a_t, a_{t-1}, \ldots, 1, \ldots, a_1, 2^{d-1}a', 2^{d-1}a'-1, \ldots, 3, 2.
\]
The lengths of $T_1, T_2, T_3,$ and $T_4$ are all equal to 
\begin{align*}
2t + 2^{d-1} a' &= (n - 2 - (F+1) 2^{d-1} a') + 2^{d-1} a'  = m - 2.
\end{align*}
We express \( T_1 = PQ_1R_1 \), \( T_2 = PQ_2R'_1 \), \( T_3 = PQ_1R_2 \), and \( T_4 = PQ_2R'_2 \), where 
\begin{itemize}
\item \( P: a_1, a_2, a_3, \ldots, a_{t-2}, a_{t-1}, a_t \)
\item \( Q_1: a_t, a_{t-1}, a_{t-2}, \ldots, \textstyle{\frac{(b' - 1)\ell}{2} +2^{d-1}a'}, \ldots, a_3, a_2, a_1 \)
\item \( Q_2: a_t, a_{t-1}, a_{t-2}, \ldots, 1, \ldots, a_3, a_2, a_1 \)
\item \( R_1, R_1': 2^d a'-1, 2^d a'-2,  \ldots,  2^{d-1}a'+2, 2^{d-1}a'+1 \)
\item \( R_2, R_2': 2^{d-1}a', 2^{d-1}a'-1, 2^{d-1}a'-2, \ldots, 4, 3, 2 \)
\end{itemize}
Note that \( P \) has length \( t \), while \( Q_1 \) and \( Q_2 \) each have length \( t + 1 \). The paths \( R_1 \) and \( R_1' \) cover the same sequence  of differences, as do \( R_2 \) and \( R_2' \), with each path having length \( 2^{d-1}a' - 1 \).

\label{page:cas1xx}

Now, color the edges in the paths \( R_1 \) and \( R_1' \) alternately pink and blue, starting with blue. Since their length is \( 2^{d-1}a' - 1 \), which is odd, they end with blue. Similarly, color the edges in the paths \( R_2 \) and \( R_2' \) alternately pink and blue, starting with pink, and necessarily ending with pink.

%
%

Next, color the edges in the path \( Q_1 \) alternately pink and blue, ensuring that the edge with difference \(\frac{(b' - 1)\ell}{2} + 2^{d-1}a'\) is pink. Color the edges in the path \( Q_2 \) alternately pink and blue, but in the opposite way of \( Q_1 \);  this ensures that the edge with difference \(1\) is blue.
 With this coloring approach, we ensure that the blue copies of the differences listed in (\ref{eq:pink_differences}) and the pink copies of the differences listed in (\ref{eq:blue_differences}) are covered. We  color the remaining paths as described below and then construct the central cycles based on the coloring.

First, assume that \( Q_1 \) starts with pink, meaning that the edge with difference \( a_t \) is assigned the pink color. Then \( Q_2 \) must start with blue. Depending on the parity of \( t \), there are two cases:
\begin{enumerate}[\bf(i)]
    \item \label{POD:i} If \( t \) is even, since \( Q_1 \) and \( Q_2 \) have length \( t + 1 \), we know \( Q_1 \) ends with pink and \( Q_2 \)  ends with blue. In this case, let \( C = x_{\infty} \, P \, Q_1 \, R_1 \, x_{\infty} \) and \( C' = x_{\infty} \, P \, Q_2 \, R_2' \, x_{\infty} \), and complete the coloring as follows. 
\vspace{-5pt}
\begin{figure}[H]
 \centering
   \includegraphics[width=0.9\linewidth,height=0.9\textheight,keepaspectratio]{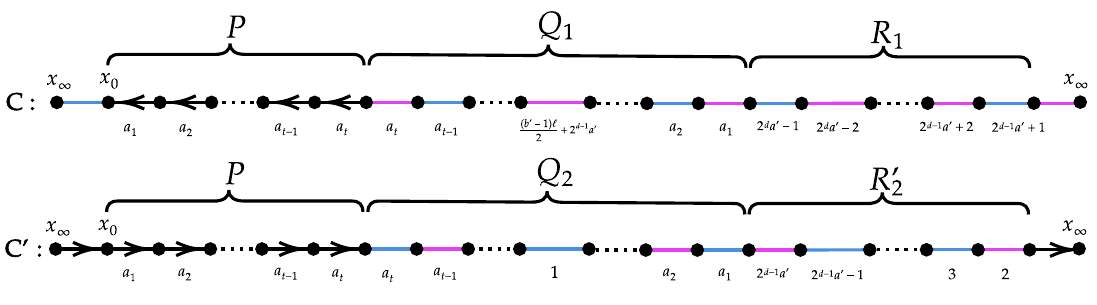}
  \vspace{-15pt}
  \end{figure} 
    \item \label{POD:ii} If \( t \) is odd,  we know \( Q_1 \)  ends with blue and \( Q_2 \)  ends with pink. In this case, let \( C = x_{\infty} \, P \, Q_1 \, R_2 \, x_{\infty} \) and \( C' = x_{\infty} \, P \, Q_2 \, R'_1 \, x_{\infty} \), and complete the coloring as follows.  
\end{enumerate}
\vspace{-15pt}
\begin{figure}[H]
 \centering
   \includegraphics[width=0.9\linewidth,height=0.9\textheight,keepaspectratio]{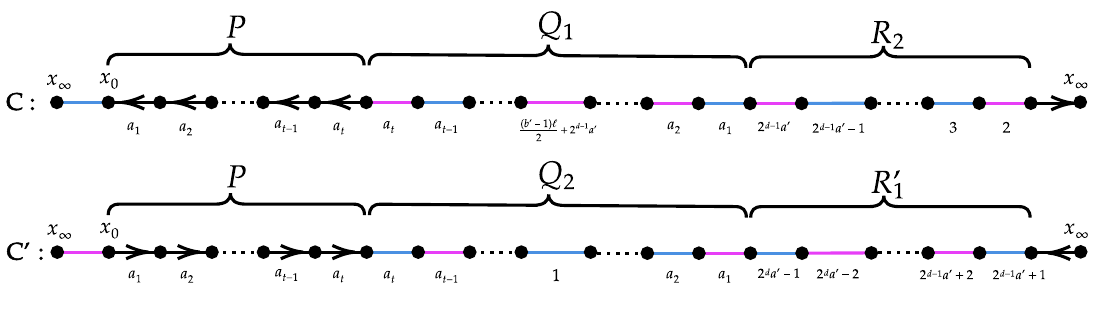}
    \vspace{-15pt}
  \end{figure} 
  
Next, assume that \( Q_1 \) starts with blue. Then \( Q_2 \) must start with pink. Depending on the parity of \( t \), there are two cases:
\begin{enumerate}[\bf(i)]
    \item \label{BOD:i} If \( t \) is even, then \( Q_1 \) ends with blue and \( Q_2 \)  ends with pink. In this case, let \( C = x_{\infty} \, P \, Q_1 \, R_2 \, x_{\infty} \) and \( C' = x_{\infty} \, P \, Q_2 \, R_1' \, x_{\infty} \), and complete the coloring as follows.  
\begin{figure}[H]
 \centering
   \includegraphics[width=0.9\linewidth,height=0.9\textheight,keepaspectratio]{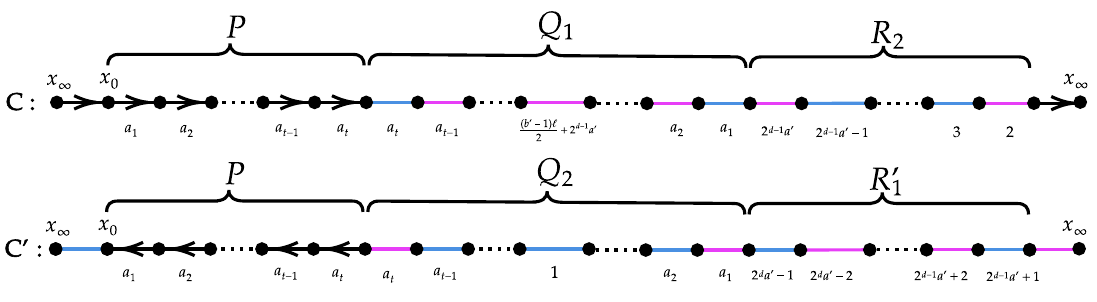}
    \vspace{-15pt}
  \end{figure}   
    \item \label{BOD:ii} If \( t \) is odd, then \( Q_1 \)  ends with pink and \( Q_2 \)  ends with blue. In this case, let \( C = x_{\infty} \, P \, Q_1 \, R_1 \, x_{\infty} \) and \( C' = x_{\infty} \, P \, Q_2 \, R'_2 \, x_{\infty} \), and complete the coloring as follows.  
\end{enumerate}
\begin{figure}[H]
 \centering
   \includegraphics[width=0.9\linewidth,height=0.9\textheight,keepaspectratio]{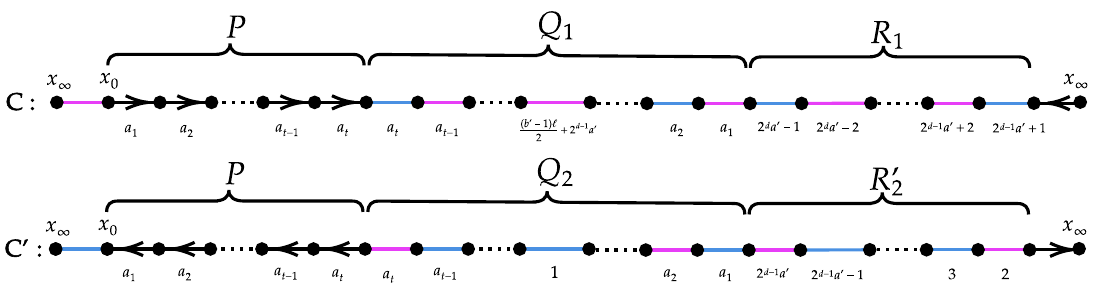}
    \vspace{-15pt}
  \end{figure} 
    
In each case, we can easily verify that both \( C \) and \( C' \) are cycles of length \( m \), and that they satisfy Condition \textbf{(C1)} of Definition \ref{def}. Moreover, they jointly contain exactly one blue edge, one pink edge, and two black arcs from each orbit of \( \langle \rho_{\bullet} \rangle \) corresponding to the differences in \( \mathscrsfs{L} \). 

The black orbits corresponding to the differences in $A$ are covered by $C'_1$ and $C'_2$, while the pink and the blue orbits of the differences in $A$ are  covered by $C, C'$, and  $C'_0$. 
%
%
Thus, 
$$\big\{\rho_{\bullet}^i(C), \rho_{\bullet}^i(C')\mid i=0,1,\ldots,n-2\big\}\cup\{C'_i, \rho_{\bullet}(C'_i), \ldots, \rho_{\bullet}^{\ell-1}(C'_i) \mid i=0,1,2\}$$
 is an HOP $(C_m)$-decomposition  for $4G_2^{\bullet}$, where \( G_2 = \text{Circ}(n-1; \pm(A\cup(\mathscrsfs{L}-\{\infty\}))) \bowtie K_1.\) 

Hence,  $4K_n^{\bullet}=4G_1^{\bullet}\oplus 4G_2^{\bullet}$ admits an HOP $(C_m)$-decomposition. 

\vspace{0.25cm}

\noindent
\textbf{\boldmath Case 2: \( F \equiv 1 \pmod{4} \).}
%
%
Here, we have two subcases:   \( b' = 1 \) and   \( b' \geq 3 \).

\textbf{Subcase 2.1:} $b'=1$. It follows from (\ref{eq:F<2b'-1}) that \( F = 1 \), and thus we only need to construct one family of peripheral cycles.
Since \( F = 1 \), and \( b' = 1 \), it follows that \( a = 3a' \). Thus, \( n = \frac{3}{2} m \), and the set of differences is \( S = \{ 1,  2, \ldots,  \frac{3m - 4}{4},  \infty\} \). 

We define the starter peripheral cycle \( C_0 \) as follows (see Figure~\ref{fig:F=1}):
\begin{align*}
C_0 = \, &x_{-1}\ x_{1}\ x_{2}\ x_{-2}\ x_{3}\ x_{-3}\ x_{4}\ x_{-4} \ldots x_{-(\frac{m}{4}-2)}\ x_{\frac{m}{4}-1}\ x_{-(\frac{m}{4}-1)}\ x_{\frac{m}{4}}\ x_{-\frac{m}{4}}\ x_{\frac{m}{2}}\  \\
& x_{-(\frac{m}{2}-1)}\ x_{\frac{m}{2}+1}\ x_{-\frac{m}{2}} \dots\dots x_{\frac{3m}{4}-3}\ x_{-(\frac{3m}{4}-4)}\ x_{\frac{3m}{4}-2}\ x_{-(\frac{3m}{4}-3)}\ x_{-(\frac{3m}{4}-2)}\ x_{\frac{3m}{4}-1}\ x_{-1}.
\end{align*}
The cycle \( C_0 \) covers the differences in the following order:
$$\textstyle{
 2,1,4,5,6,\ldots, \frac{m}{2}-3, \frac{m}{2}-2, \frac{m}{2}-1, \frac{m}{2}, \frac{3m-4}{4}, \frac{m}{2}, \frac{m}{2}-1, \frac{m}{2}-2, \ldots, 6,5,4,1, 2,\frac{3m-4}{4}.
}$$
Observe that \( C_0 \) covers each difference in \( A = \{ 1, 2, 4, 5, \ldots, \frac{m}{2} - 1, \frac{m}{2}, \frac{3m - 4}{4} \} \) exactly twice. Color the edges of \( C_0 \) black and orient them forwards as shown in Figure \ref{fig:F=1}. 
\begin{figure}[h]
    \centering
    \includegraphics[width=0.7\linewidth,height=0.7\textheight,keepaspectratio]{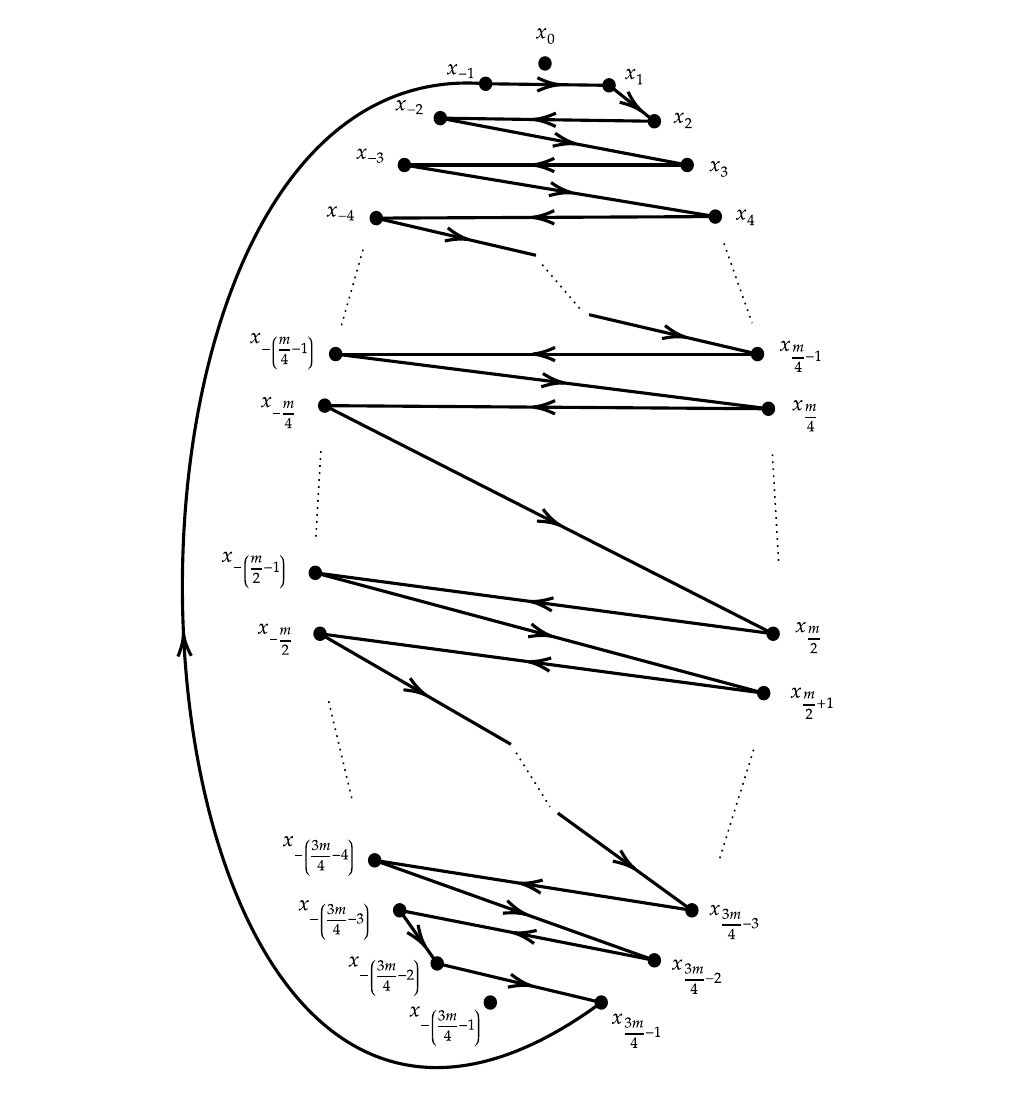}
    \caption{Orientation and coloring of the edges of \( C_0 \).}
    \label{fig:F=1}
\end{figure}
Let \(\mathscrsfs{L} \) be the set of unused differences in \(S\), that is, \(\mathscrsfs{L}  = S \setminus A=\{a_1, a_2, \ldots, a_t\} \cup \{\infty\}\), where \(a_1=3 < a_2 < \ldots < a_t < \frac{n-2}{2}=\frac{3m-4}{4}\)   and \(t = \frac{m-4}{4}\). Also, observe that \( \frac{m}{2} < a_2 \).
Consider the following sets of differences: 
\begin{itemize}
\item \( W_1=\{ 1, 2, a_1, 4, 5\ldots, \frac{m}{2}-1, \frac{m}{2}, a_2, a_3, \ldots, a_{t-1}, a_{t}, \frac{3m-4}{4}\}\)
\item \( W_2=\{  a_1, a_2, a_3, \ldots, a_{t-2}, a_{t-1}, a_t \} \)
\end{itemize}
Observe that 
\[
1<2<a_1< 4< 5<\ldots < \frac{m}{2} < a_2 < a_3 < \ldots < a_{t-1} < a_t <\frac{3m-4}{4}.
\]
Here, \(|W_1| = \frac{m}{2}+t\) and \(|W_2| = t \).  By Lemma \ref{zig-zag}(i), there exists a path  \( T \) of length \( 2t + \frac{m}{2}=m-2\) that starts at \( x_0 \) and covers the differences in the following order:
\[
T: \textstyle{1, 2, a_1, 4, 5,  \ldots, \frac{m}{2}-1, \frac{m}{2}, a_2, a_3, \ldots, a_{t-1}, a_t,  \frac{3m-4}{4}, a_t, a_{t-1},  \ldots, a_2, a_1}.
\]
Let \( P \), \( Q \),  and \( R \) be the subpaths of $T$ that cover the following  sequences of differences.
\begin{itemize}
\item \( P: 1, 2, a_1, 4,5, \ldots, \frac{m}{2}-1, \frac{m}{2} \)
\item \( Q:  a_2, a_3, \ldots, a_{t-2}, a_{t-1}, a_t, \frac{3m-4}{4}\)
\item \( R: a_t, a_{t-1}, a_{t-2}, \ldots, a_3, a_2, a_1 \)
\end{itemize}
Note that $T=PQR$ and  \( P \) has length \( \frac{m}{2} \), while \( Q \) and \( R \) each have length \( t \). Let \( C = x_{\infty} PQR x_{\infty} \); observe that \( C \) is a cycle of length \( m \).
Let \( C' \) be another copy of \( C \). In \( C \), color the edges of the path \( P \) alternatingly pink and blue, starting with pink. Since the length of the path $P$ is even, it ends with blue. In \( C' \), color the edges of the path \( P \) alternatingly pink and blue, starting with blue;  it ends with pink. Color the rest of the edges as follows:

\begin{itemize}
    \item If \( t \) is even, \( Q \) ends with blue in \( C \) and with pink in \( C' \).
    \begin{figure}[H]
 \centering
   \includegraphics[width=0.9\linewidth,height=0.9\textheight,keepaspectratio]{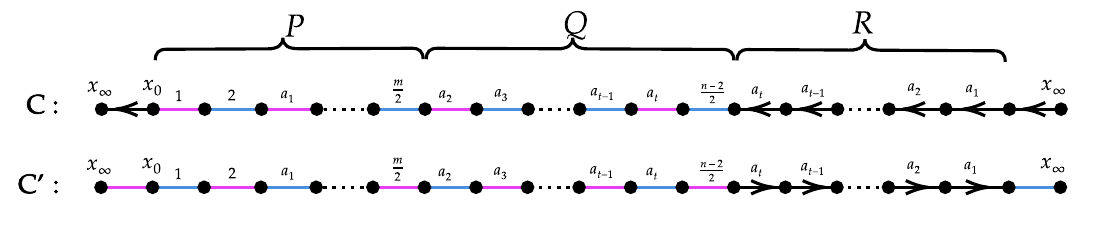}
   \vspace{-15pt}
  \end{figure} 
    \item If \( t \) is odd, \( Q \) ends with pink in \( C \) and with blue in \( C' \).
\begin{figure}[H]
 \centering
   \includegraphics[width=0.9\linewidth,height=0.9\textheight,keepaspectratio]{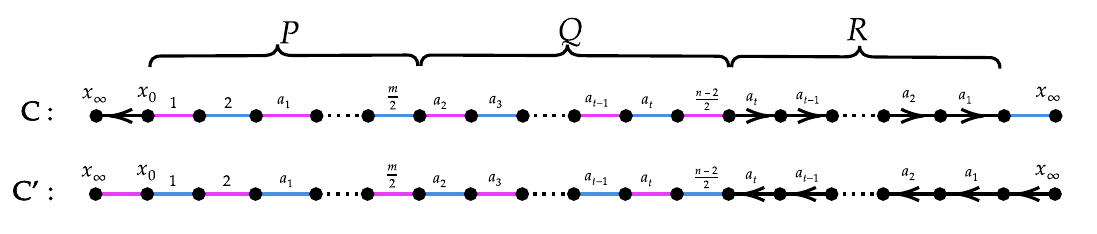}
   \vspace{-15pt}
  \end{figure} 
\end{itemize}
The cycle \( C_0 \) covers the black copies of the differences in the set \( A \), while the pink and blue copies of these differences are covered by the paths \( P \) and \( Q \) in the central cycles. Additionally, the central cycles \( C \) and \( C' \) together cover exactly one blue edge, one pink edge, and two opposite black arcs corresponding to the differences in the set \(\mathscrsfs{L} \).
Furthermore, the cycles \( C_0 \), \( C \), and \( C' \) satisfy Condition {\bf (C1)} in Definition {\rm \ref{def}}.
Therefore, 
\(
\{\rho^i_{\bullet}(C_0), \rho^i_{\bullet}(C), \rho^i_{\bullet}(C') \mid i=0,1,\ldots,n-2\}
\)
is an HOP $(C_m)$-decomposition of $4K_{n}^{\bullet}$. 


\textbf{Subcase 2.2:} $b'\geq 3$. As before, the approach is to first construct $\frac{F-1}{4} + 1$ starter peripheral cycles for $K_n$. 
We construct paths \( P_i \) for \( i = 0, 1, 2, \ldots, \frac{F - 1}{4} \). From (\ref{eq:F<2b'-1}), it follows that
\(
\frac{F - 1}{4} \leq \frac{b' - 1}{2}.
\)
Note that the construction of $P_i$  depends on whether \( \frac{F - 1}{4} < \frac{b' - 1}{2} \) or \( \frac{F - 1}{4} = \frac{b' - 1}{2} \). 

\textbf{ Subcase 2.2.1:} {\boldmath {\( \frac{F - 1}{4} < \frac{b' - 1}{2}. \)}}
%
%
%
For \( i = 0, 1, \ldots, \frac{F-1}{4}\), define \( P_i \) as in Case~1 (page \pageref{page:cas1L}). The differences covered by $P_i$ are given in \ref{eq:1-P0} and \ref{eq:1}.

As in Case~1,
\(
C_i=P_i \cup \rho^{\ell}(P_i) \cup \ldots \cup \rho^{(b'-1)\ell}(P_i)
\)
is an \(m\)-cycle.
Since \(\frac{F-1}{4} < \frac{b'-1}{2}\) and  \( 2^d a'  < \ell \) by \ref{eq:ell}, we have 
\begin{align*}
\frac{F - 1}{4} \ell + 2^d a' - 1 
\leq \frac{b' - 3}{2} \ell + 2^d a' - 1 
< \frac{b' - 1}{2} \ell + 2^{d-1} a' < \frac{n - 1}{2}.
\end{align*}
The above inequality guarantees 
the differences \( i\ell + 1, i\ell + 2, \ldots, i\ell + 2^d a' - 1 \) are pairwise distinct, and for all \( j = 0, 1, \ldots, \frac{F - 1}{4} \), we have 
\(
\frac{b' - 1}{2} \ell + 2^{d-1} a' \notin \{j\ell + 1, j\ell + 2, \ldots, j\ell + 2^d a' - 1\}.
\)

Next, we show that for all \( i= 1, \ldots, \frac{F-1}{4} \) and \( j = 0, 1, \ldots, \frac{F-1}{4} \)
\[
\frac{(b' - 2i + 1)\ell}{2} - 2^{d-1}a'
\notin 
\{j\ell + 1, j\ell + 2, \ldots, j\ell + 2^d a' - 1\}.
\]
Since 
\[
\textstyle{
0<\frac{(b' - 2\cdot \frac{F - 1}{4} + 1)\ell}{2} - 2^{d-1}a'
< \frac{(b' - 2\cdot \frac{F - 5}{4} + 1)\ell}{2} - 2^{d-1}a'
< \ldots 
< \frac{(b' - 2\cdot 1 + 1)\ell}{2} - 2^{d-1}a'
< \frac{(b' -1)\ell}{2} +2^{d-1}a',
}\]
it suffices to show that for all \( j = 0, 1, \ldots, \frac{F-1}{4} \) and \(1 \leq \alpha \leq 2^d a' - 1\)
\[
\frac{(b' - 2( \frac{F-1}{4}) + 1)\ell}{2} - 2^{d-1}a' \neq j\ell + \alpha.
\]
We prove this by contradiction. Suppose to the contrary that $\frac{2b' -F-1}{4}\ell + \ell- 2^{d-1}a' = j\ell + \alpha$. Since $F \equiv 1 \pmod{4}$ and $b'$ is odd, $\frac{2b' - F - 1}{4}$ is an integer; combining this with $2^d a' < \ell$ from \eqref{eq:ell}, we have $j = \frac{2b' - F - 1}{4}$ and $\alpha = \ell - 2^{d-1}a'$.
Given \(j \leq \frac{F - 1}{4}\), it follows that \(F \geq b'\). 
On the other hand, since  \(1 \leq \alpha \leq 2^d a' - 1\) and \(\ell - 2^{d-1}a' = \alpha\), it follows that \(\ell - 2^{d-1}a' \leq 2^d a' - 1\), and by (\ref{eq:333}), we have  \(F \leq b'\).
 Therefore, we have \( b' \leq F \leq b' \), which implies \( F =\frac{a - 2a'b'}{a'} = b'\), and  \( a = 3a'b' \). Consequently, \( n = 2^{d-1} \cdot 3a'b' \). Since \( b' \mid (n-1) \), it must hold that \( b' = 1 \), a contradiction.  
Thus,  the differences listed in (\ref{eq:1-P0}) and (\ref{eq:1}) are pairwise distinct.



\textbf{Subcase 2.2.2:} {\boldmath {\( \frac{F - 1}{4} = \frac{b' - 1}{2}. \)}}
Let \(f=\frac{F-1}{4} = \frac{b'-1}{2}\).
Note that  path $P_f$ covers the following differences:
\begin{align}
    &f\ell + 1,\; f\ell + 2,\; \ldots,\; f\ell + 2^{d-1}a' - 1,\; \textcolor{red}{f\ell + 2^{d-1}a'},\; f\ell + 2^{d-1}a' + 1, \nonumber \\
    &\ldots,\; f\ell + \tfrac{\ell - 3}{2},\; \textcolor{red}{f\ell + \tfrac{\ell - 1}{2}},\; f\ell + \tfrac{\ell - 1}{2},\; f\ell + \tfrac{\ell - 3}{2},\; \ldots,\; \ell - 2^{d-1}a'. \label{eq:sequence21}
\end{align}
Notice that \( f\ell + 2^{d-1}a' = \frac{b' - 1}{2} \ell + 2^{d-1}a' \). However, this difference has already occurred in \( P_0 \). Continuing further, we encounter $f\ell + \frac{\ell-1}{2} = \frac{n-2}{2}$, after which more repeated differences occur.
To address this issue, we modify the construction of $P_f$ to avoid repeated differences. Specifically:
\begin{itemize}
\item  We  avoid the difference $f\ell + 2^{d-1}a' = \frac{b'-1}{2}\ell + 2^{d-1}a'$.
\item  We replace differences occurring after $f\ell + \frac{\ell-1}{2} = \frac{n-2}{2}$ with differences from the interval $[2^d a', \ell]$.
\end{itemize}
Given that the length of \(P_f\) must be \(2^d a'\), we replace \(2^d a' - \frac{\ell - 3}{2}\) of the differences listed in (\ref{eq:sequence21}).

Next, we explain why we chose differences from the interval $[2^d a', \ell]$.
\begin{itemize}
\item The second-to-last difference used in $P_0$ is $2^d a' - 1$, while the first difference used in $P_1$ is $\ell + 1$.
\item  The differences $\frac{(b' - 2i + 1)\ell}{2} - 2^{d-1}a'$ decrease as $i$ increases.
The last path before $P_f$ is $P_{f-1}$. For $i = f-1 = \frac{b'-3}{2}$, the corresponding difference is $2\ell - 2^{d-1}a'$. 
By (\ref{eq:ell}), we have \(2^{d-1} a' <2^d a' < \ell\), and hence $\ell < 2\ell - 2^{d-1}a'.$
\end{itemize}
Therefore, replacing the (potentially) repeated differences with differences in the range $[2^d a', \ell]$ ensures that the differences covered by $P_f$ are distinct from those covered by $P_i$ for $i = 0, 1, \ldots, f-1$.

Before discussing the construction of $P_f$, we first prove some inequalities. 
Recall that $d \geq 2$, with $a' \geq 1$ and $b' \geq 3$ both odd. Furthermore, $\frac{F-1}{4} = \frac{b' - 1}{2}$, which implies $F = 2b' - 1$. Therefore, we have $a = 4a'b' - a'$ and $n = 2^{d-1}(4a'b' - a')$. Thus,
\begin{equation}
\ell = \frac{n - 1}{b'} = 2^{d+1}a' - \frac{2^{d-1}a' + 1}{b'}. \label{eq:ell_formula}
\end{equation}
By \eqref{eq:ell_formula}, the parameters $(a',d,b') \in \{(1,2,5),(1,3,3)\}$ lead to contradictions. Apart from these parameters, 
we see
\begin{equation}
2\left(\frac{2^{d-1}a' + 1}{2^{d-1}a' - 1}\right) < b', \quad \text{for } (a', d, b') \notin \{(1, 2, 3)\}. \label{eq:ine--y}
\end{equation}
Applying the inequality from \eqref{eq:ine--y} to \eqref{eq:ell_formula}, we obtain:
\begin{align*}
\ell &> 2^{d+1}a' - \frac{2^{d-1}a' - 1}{2} = 2^{d+1}a' - 2^{d-2}a' + \frac{1}{2} = 7 \cdot 2^{d-2}a' + \frac{1}{2}.
\end{align*}
Therefore, we have the following:
\begin{equation}
\ell > 7 \cdot 2^{d-2}a' + \frac{1}{2} \quad \text{for } (a', d, b') \notin \{(1, 2, 3)\}.
\label{eq:ell_lower_bound}
\end{equation}
Using \eqref{eq:ell_lower_bound}, we establish the following inequalities:
\begin{itemize}

\item We have $\frac{\ell-1}{4} > \frac{7}{4} 2^{d-2} a' - \frac{1}{8} > 2^{d-2} a' + 1$. Rearranging terms yields:
\begin{equation} 
-(2^{d-2}a' + 1) \geq -\frac{\ell - 1}{4}. \label{eq:Simple1}
\end{equation}

\item We have $\frac{3\ell - 1}{2} > 21 \cdot 2^{d-3}a' + \frac{1}{4} > 2^{d+1}a'$. Thus,
\begin{equation}
\frac{3\ell - 1}{2} > 2^{d+1}a'. \label{eq:Simple2}
\end{equation}
\end{itemize}
 Finally, from \eqref{eq:ell} we know $\frac{\ell}{2} < 2^d a'$, and hence
\begin{equation}
\frac{\ell - 3}{4} < 2^d a' - \frac{\ell + 1}{4}. \label{eq:Simple3}
\end{equation}
%
%
%
Note that since $b'\ell = n - 1 < 2m = 2^{d+1} a'b'$, it follows that $\ell < 2^{d+1} a'$. This implies $\frac{\ell - 3}{2} < 2^d a'$, confirming that the number of replaced differences, $2^d a' - \frac{\ell - 3}{2}$, is positive.

Next, we explain how we construct the path \( P_f \). There are two cases to consider:   \(\ell \equiv 3 \pmod{4}\)  and   \(\ell \equiv 1 \pmod{4}\). The construction will be verified using the inequalities proven above. Note that the inequality (\ref{eq:ell_lower_bound}) and its consequent inequalities are proved for  \( (a', d, b') \notin \{(1, 2, 3)\} \). First note that if \( (a', d, b') = (1, 2, 3) \), we have \( \ell = 7 \) by \eqref{eq:ell_formula}, \( m = 12 \), \( n = 22 \), \(f=1\), and \( F = 5 \). In this case, \( \ell \equiv 3 \pmod{4} \), and we verify the construction of the path \( P_f \) without using inequality (\ref{eq:ell_lower_bound}).


\begin{itemize}

\item \textbf{Subcase 2.2.2.1: }\(\ell \equiv 3 \pmod{4}\). If \( (a', d, b') = (1, 2, 3) \), we have \( P_f = x_0\ x_8\ x_{-2}\ x_2\ x_7 \), and the differences covered are \( 8, 10, 4, 5 \). It is clear that \( P_f \) is a path and the differences covered are pairwise distinct. 
In all other cases, define a walk  $P_f$ as follows (see Figure~\ref{Ofig:2b'-1}):
\begin{align*}
P_f = &x_0\ x_{f\ell+1}\ x_{-1}\ x_{f\ell+2}\ x_{-2} \dots \\
& x_{-(2^{d-2}a'-1)}\ x_{f\ell+2^{d-2}a'}\ x_{-(2^{d-2}a'+1)}\ x_{f\ell+2^{d-2}a'+1}\ x_{-(2^{d-2}a'+2)}\ x_{f\ell+2^{d-2}a'+2} \dots \\
& \ x_{f\ell+ \frac{\ell-3}{4} } \ x_{- \frac{\ell+1}{4} }\ x_{2^da'- \frac{\ell+1}{4} }\ x_{- \frac{\ell+5}{4} }\ x_{2^da'- \frac{\ell-3}{4} }  \dots \\
 &\ x_{-(2^{d-1}a'-1)}\ x_{3\cdot2^{d-1}a'-\frac{\ell+3}{2}}\ x_{-(2^{d-1}a')}\ x_{3\cdot2^{d-1}a'-\frac{\ell+1}{2}} \ x_{\ell}
\end{align*}
\end{itemize}
To show that \( P_f \) is a path, we verify the following:
\begin{itemize}
\item By (\ref{eq:Simple1}), we have $
-(2^{d-2}a' +1)  \geq - \frac{\ell+1}{4}.$

\item Since  $\frac{\ell}{4} < 2^{d-1}a'$ by \eqref{eq:ell}, and $\ell \equiv 3 \pmod{4}$, we have $\frac{\ell + 1}{4} \leq 2^{d-1}a'$. This confirms $3 \cdot 2^{d-1}a' - \frac{\ell+1}{2} \geq 2^d a' - \frac{\ell+1}{4}$.

\item By (\ref{eq:ell_lower_bound}), we see $21 \cdot 2^{d-3}a' + \frac{3}{4}<\frac{3\ell}{2}$,  thus
\(
3 \cdot 2^{d-1}a'  < 21 \cdot 2^{d-3}a' + \frac{3}{4}< \frac{3\ell+1}{2}.
\)
Therefore, we have $3\cdot2^{d-1}a'-\frac{\ell+1}{2} < \ell$.
\end{itemize}
Thus, we conclude that $P_f$ is a path.  The differences covered by \( P_f \) are:
\begin{align*}
&f\ell+1, f\ell+2, f\ell+3, \ldots, f\ell+2^{d-1}a'-1, f\ell+2^{d-1}a'+1, f\ell+2^{d-1}a'+2, \ldots \\
& \ldots, f\ell+{\textstyle{\frac{\ell-1}{2}}},\  2^da', 2^da'+1,  \ldots, 2^{d+1}a'-\frac{\ell+3}{2}, 2^{d+1}a'-\frac{\ell+1}{2}, \frac{3\ell+1}{2}-3 \cdot 2^{d-1}a'.
\end{align*}
Next, we show that the following differences from the above list are pairwise distinct and lie in the interval $[2^d a', \ell]$.
\[
2^da', 2^da'+1,  \ldots, 2^{d+1}a'-\frac{\ell+3}{2}, 2^{d+1}a'-\frac{\ell+1}{2}, \frac{3\ell+1}{2}-3 \cdot 2^{d-1}a'
\] 
To show they are distinct, it suffices to prove that $\frac{3\ell+1}{2}-3 \cdot 2^{d-1}a' > 2^{d+1}a'-\frac{\ell+1}{2}$, which is equivalent to $2\ell+1 > 7 \cdot 2^{d-1}a'$. This follows by  \eqref{eq:ell_lower_bound}.

Furthermore, since $\ell< 2^{d+1}a'$ by \eqref{eq:ell}, we have $\frac{3\ell + 1}{2} - 3 \cdot 2^{d-1} a' < \ell$, so the differences are all less than $\ell$. 

Next, we show that the internal vertices of the path \( P_f \) are pairwise distinct modulo \( \ell \)  (see Figure~ \ref{Ofig:2b'-1}). We do this by proving the following two inequalities.
\begin{figure}[h!]
 \centering
   \includegraphics[width=0.8\linewidth,height=0.8\textheight,keepaspectratio]{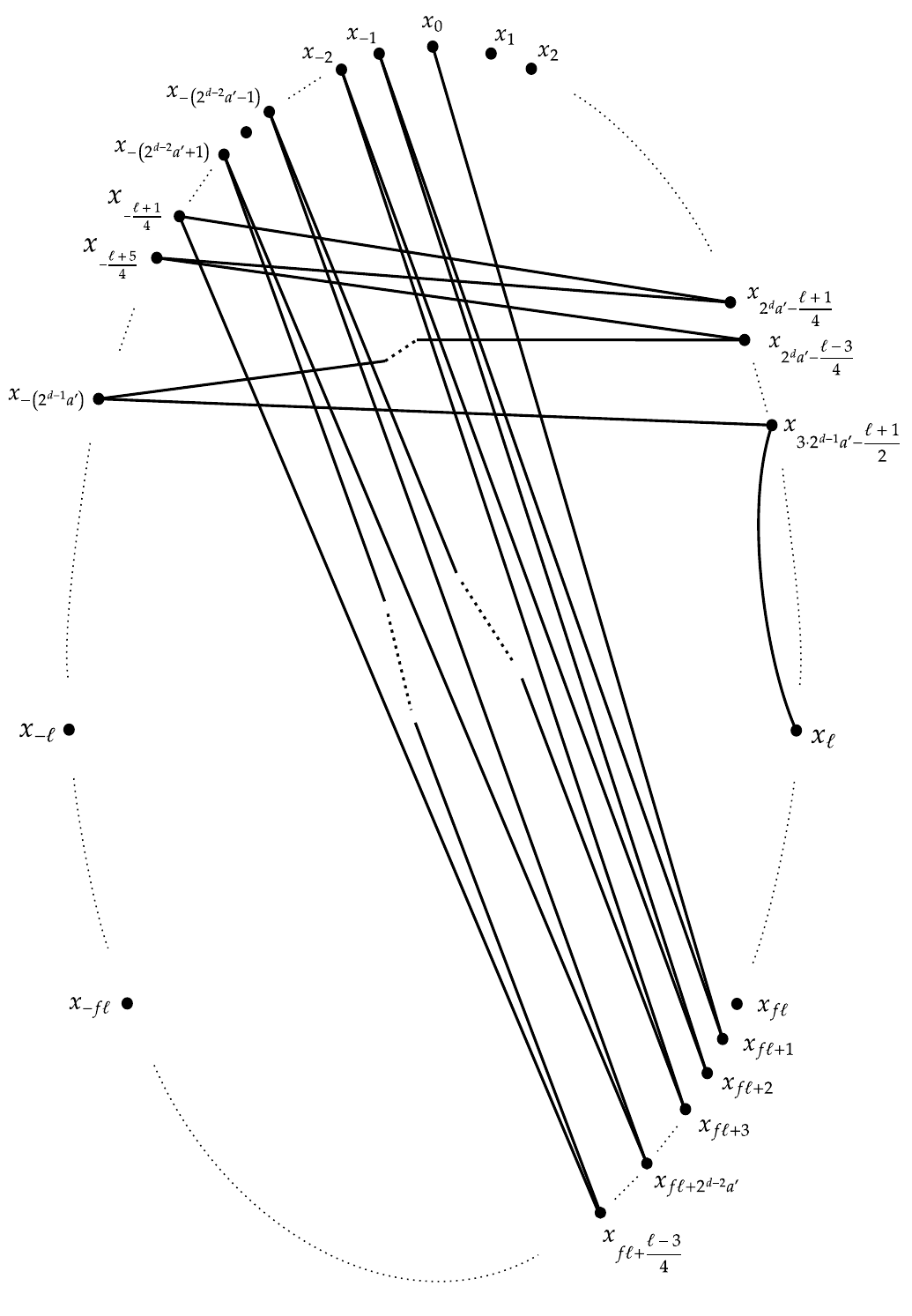}
   \caption{Illustration of the path \( P_f \) for \(\ell \equiv 3 \pmod{4}\).}
   \label{Ofig:2b'-1}
\end{figure}
%
\begin{enumerate}[\bf{(i)}]
\item 
\( \frac{\ell - 3}{4} < 2^d a' -  \frac{\ell +1}{4}  \).
This is proved in (\ref{eq:Simple3}).
\end{enumerate}

\begin{enumerate}[\bf{(ii)}]
\item 
\( \ell-2^{d-1} a' > 3\cdot2^{d-1}a'-\frac{\ell+1}{2},   {\text{ which is equivalent to }}  \frac{3\ell + 1}{2} > 2^{d+1}a' .\) The latter follows from (\ref{eq:Simple2}).
\end{enumerate}
Therefore, the internal vertices of the path \( P_f \) are pairwise distinct modulo \( \ell \).


\begin{itemize}
\item  \textbf{Subcase 2.2.2.2: } \(\ell \equiv 1 \pmod{4}\).
Define a walk $P_f$ as follows:
\begin{align*}
P_f = &x_0\ x_{f\ell+1}\ x_{-1}\ x_{f\ell+2}\ x_{-2} \dots \\
& x_{-(2^{d-2}a'-1)}\ x_{f\ell+2^{d-2}a'}\ x_{-(2^{d-2}a'+1)}\ x_{f\ell+2^{d-2}a'+1}\ x_{-(2^{d-2}a'+2)}\ x_{f\ell+2^{d-2}a'+2} \dots \\
& \ x_{f\ell+ \frac{\ell-5}{4} } \ x_{- \frac{\ell-1}{4} }\ x_{2^da'- \frac{\ell-1}{4} }\ x_{- \frac{\ell+3}{4} }\ x_{2^da'- \frac{\ell-5}{4} }  \dots \\
 &\ x_{-(2^{d-1}a'-1)}\ x_{3\cdot2^{d-1}a'-\frac{\ell+1}{2}}\ x_{-(2^{d-1}a')}\ x_{3\cdot2^{d-1}a'-\frac{\ell-1}{2}} \ x_{\ell}
\end{align*}
\end{itemize}
Analogous to Subcase 2.2.2.1, it can be shown that \(P_f\) is a path that covers pairwise distinct differences in the correct interval and that  the internal vertices of  \( P_f \) are pairwise distinct modulo \( \ell \).

In both Subcases 2.2.2.1 and  2.2.2.2, for \( j = 0, \ell, 2\ell, \ldots, (b' - 1)\ell \), the paths \( \rho^j(P_f) \) are pairwise vertex-disjoint, except at the endpoints.  Hence, 
\(
C_f = P_f \cup \rho^{\ell}(P_f) \cup \cdots \cup \rho^{(b' - 1)\ell}(P_f)
\)
 is an \( m \)-cycle.

Furthermore, the path \( P_f \) is constructed so that the differences it covers are distinct from those covered by the paths \( P_i \) for \( i = 0, 1, \ldots, f-1 \). Hence, the differences covered by the paths \( P_i \) for \( i = 0, 1, \ldots, \frac{F-1}{4} \) are pairwise distinct.

Let \( B \) be the set of differences covered by the \( m \)-cycles \( C_1, C_2 \ldots, C_{\frac{F-1}{4}} \). 
Let \( G_1 = \text{Circ}(n-1; \pm B) \). Then,  \(\{C_i, \rho(C_i), \ldots, \rho^{\ell-1}(C_i) \mid i=1, \ldots, \frac{F - 1}{4}\}\) is a \((C_m)\)-decomposition of \( G_1 \), and hence by Lemma~\ref{lem:Gtool4}, there exists  an HOP \((C_m)\)-decomposition of \( 4G_1^{\bullet} \).

Notice that $A=\{1, 2, \ldots,  2^d a' - 2, 2^d a' - 1, \textstyle{\frac{(b' - 1)\ell}{2} + 2^{d-1}a'}\}$ is the set of differences covered by \( C_0 \).  As in Case 1, color the edges of \(P_0\) alternately pink and blue (as shown below), starting with pink. Since \(P_0\) has length \(2^d a'\), which is even, it ends with blue. 
\begin{figure}[H]
 \centering
   \includegraphics[width=0.9\linewidth,height=0.9\textheight,keepaspectratio]{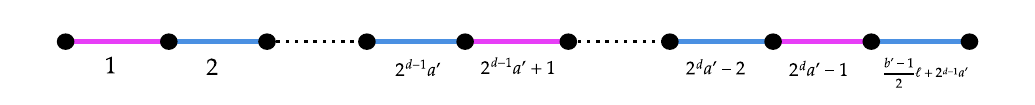}
   \vspace{-15pt}
  \end{figure}
  
The family of starter cycles generated by \( C_0 \); namely, \(\{C_0, \rho_{\bullet}(C_0), \ldots, \rho_{\bullet}^{\ell-1}(C_0) \}\), covers:
\begin{enumerate}[(i)]
    \item the pink orbits corresponding to the differences
    \begin{equation}
    1, 3, 5, \ldots, 2^{d-1}a' + 1, 2^{d-1}a' + 3, \ldots, 2^d a' - 3, 2^d a' - 1, \text{ and}
    \label{eq:pink_differences2}
    \end{equation}
    
    \item the blue orbits corresponding to the differences
    \begin{equation}
    2, 4, 6, \ldots, 2^{d-1}a', 2^{d-1}a' + 2, \ldots, 2^d a' - 2, \frac{(b' - 1)\ell}{2} + 2^{d-1}a'.
    \label{eq:blue_differences2}
    \end{equation}
 \end{enumerate}   
The differences in the set \( A \) will also appear in the central cycles \( C \) and \( C' \), to be constructed below. 

Let \(\mathscrsfs{L} \) be the set of unused differences in \(S\), that is, \(\mathscrsfs{L}  = S \setminus (A \cup B)= \{a_1, a_2, \ldots, a_t\} \cup \{\infty\}\), where \(a_1 < a_2 < \ldots < a_t \leq \frac{n-2}{2}\)  and \(t = \frac{n-2}{2} - \frac{(F+3) 2^d a'}{4}\). 

Notice that since \( b' \geq 3 \), we have \( \frac{(b' - 1)\ell}{2} + 2^{d-1}a' > \ell \geq 2^da' \). Observe that \[ 2^d a' - 1 < a_1 \leq \frac{(b' - 1)\ell}{2} + 2^{d-1}a' \leq \frac{n-2}{2}. \]

Consider the following sets of differences: 
\begin{itemize}
\item \( W_1=\{ 1, 2, \ldots,  2^d a' - 1, a_1, a_2, a_3, \ldots, \textstyle{\frac{(b' - 1)\ell}{2} + 2^{d-1}a'}, \ldots, a_{t-2}, a_{t-1}, a_t\}\)
\item \( W_2=\{   2^{d-1}a'+1, 2^{d-1}a'+2,  \ldots, 2^d a'-1, a_1, a_2, \ldots, \textstyle{\frac{(b' - 1)\ell}{2} + 2^{d-1}a'}, \ldots, a_{t-1}, a_t\} \)
\end{itemize}
In the sequence \( a_1, a_2, \ldots, \frac{(b' - 1)\ell}{2} + 2^{d-1}a', \ldots, a_{t-1}, a_t \), we assume
\[
a_1 < a_2 < \ldots < \frac{(b' - 1)\ell}{2} + 2^{d-1}a' < \ldots < a_{t-1} < a_t,
\]
with the understanding that
\(
a_1 < a_2 < \ldots < a_{t-1} < a_t < \frac{(b' - 1)\ell}{2} + 2^{d-1}a'
\)
is also possible.
Here,  \( W_1 \) contains all the differences from \( A \) and all the differences from \( \mathscrsfs{L} \), except the difference \( \infty \), and  \( |W_1| = t +2^d a' \). The set \( W_2 \) contains half of the differences from \( A \) and all differences from \( \mathscrsfs{L} \), except the difference \( \infty \), and  \( |W_2| = t + 2^{d-1}a'\).
 By Lemma \ref{zig-zag}(i), there exists a path  \( T_1 \) of length \( 2t + 3\cdot2^{d-1} a' \) that starts at \( x_0 \) and covers the differences in $W_1$ and $W_2$ in the following order:
\[
\begin{aligned}
T_1: & \, 1, 2, \ldots, 2^d a' - 1, a_1, a_2, \ldots, \textstyle{\frac{(b' - 1)\ell}{2} + 2^{d-1}a'}, \ldots, a_{t-2},  a_{t-1}, a_t, \\
& a_t, a_{t-1},  \ldots,  \textstyle{\frac{(b' - 1)\ell}{2} + 2^{d-1}a'}, \ldots, a_2, a_1, 2^d a' - 1, \ldots, 2^{d-1}a'+2,  2^{d-1}a' + 1.
\end{aligned}
\]
Similar to Case 1, we use Lemma \ref{zig-zag}(ii) to construct a path \( T_2 \) by replacing the second occurrence of the difference \(\textstyle \frac{(b' - 1)\ell}{2} + 2^{d-1}a'\) with  difference \(1\) in the path \( T_1 \).
\[
\begin{aligned}
T_2: & \, 1, 2, \ldots, 2^d a' - 1, a_1, a_2, \ldots, \textstyle{\frac{(b' - 1)\ell}{2} + 2^{d-1}a'}, \ldots, a_{t-2},  a_{t-1}, a_t,  \\
& a_t, a_{t-1}, \ldots,1, \ldots, a_2, a_1, 2^d a' - 1, \ldots, 2^{d-1}a'+2,  2^{d-1}a' + 1.
\end{aligned}
\]
Now, consider the following set of differences: 
\begin{itemize}
\item \( W_2'=\{   2,3,4, \ldots, 2^{d-1}a'-1, 2^{d-1}a', a_1, a_2, \ldots, \textstyle{\frac{(b' - 1)\ell}{2} + 2^{d-1}a'}, \ldots, a_{t-1}, a_t\} \)
\end{itemize}
Here,  \(|W_2'| = t + 2^{d-1}a'\).  By Lemma \ref{zig-zag}(i), there exists a path  \( T_3 \) of length \( 2t + 3\cdot2^{d-1} a' \) that starts at \( x_0 \) and covers the differences in $W_1$ and $W_2'$ in the following order:
\[
\begin{aligned}
T_3: & \, 1, 2, \ldots, 2^d a' - 1, a_1, a_2, \ldots, \textstyle{\frac{(b' - 1)\ell}{2} + 2^{d-1}a'}, \ldots, a_{t-2}, a_{t-1}, a_t,  \\
& a_t, a_{t-1}, \ldots,\textstyle{\frac{(b' - 1)\ell}{2} + 2^{d-1}a'}, \ldots, a_2, a_1, 2^{d-1}a', 2^{d-1}a'-1, \ldots, 3, 2.
\end{aligned}
\]
By Lemma \ref{zig-zag}(ii), the second occurrence of the difference \( \frac{(b' -1)\ell}{2} + 2^{d-1}a' \) can be replaced with  difference \( 1 \); that is, there exists a path  \( T_4 \) that covers the differences in the following order:
\[
\begin{aligned}
T_4: & \, 1, 2, \ldots, 2^d a' - 1, a_1, a_2, \ldots, \textstyle{\frac{(b' - 1)\ell}{2} + 2^{d-1}a'}, \ldots, a_{t-2},  a_{t-1}, a_t, \\
& a_t, a_{t-1}, \ldots, 1, \ldots, a_2, a_1, 2^{d-1}a', 2^{d-1}a'-1, \ldots, 3, 2.
\end{aligned}
\]
The paths $T_1, T_2, T_3, T_4$ each have length 
\begin{align*}
2t + 3 \cdot 2^{d-1} a' = \left(n - 2 - (F+3) 2^{d-1} a'\right) + 3 \cdot 2^{d-1} a' = n - 2 - F \cdot 2^{d-1} a' =  m - 2.
\end{align*}

As in {\bf{Case~1}},  we express \( T_1 = PQ_1R_1 \), \( T_2 = PQ_2R'_1 \), \( T_3 = PQ_1R_2 \), and \( T_4 = PQ_2R'_2 \), where 
\begin{itemize}
\item \( P: 1, 2, \ldots,  2^d a' - 1, a_1, a_2, a_3, \ldots, \textstyle{\frac{(b' -1)\ell}{2} +2^{d-1}a'}, \ldots, a_{t-2}, a_{t-1}, a_t\)
\item \( Q_1: a_t, a_{t-1}, a_{t-2}, \ldots, \textstyle{\frac{(b' - 1)\ell}{2} + 2^{d-1}a'}, \ldots, a_3, a_2, a_1 \)
\item \( Q_2: a_t, a_{t-1}, a_{t-2}, \ldots, 1, \ldots, a_3, a_2, a_1 \)
\item \( R_1, R_1': 2^d a'-1, 2^d a'-2, 2^d a'-3, \ldots, 2^{d-1}a'+3, 2^{d-1}a'+2, 2^{d-1}a'+1 \)
\item \( R_2, R_2': 2^{d-1}a', 2^{d-1}a'-1, 2^{d-1}a'-2, \ldots, 4, 3, 2 \)
\end{itemize}
Note that \( P \) has length \( t +2^d a'\), while \( Q_1 \) and \( Q_2 \) each have length \( t + 1 \). The paths \( R_1 \) and \( R_1' \) cover the same sequences of differences, as do \( R_2 \) and \( R_2' \), with each path having length \( 2^{d-1}a' - 1 \).

The construction of \(C\) and \(C'\), along with their coloring, is the same as in Case~1 (see page~\pageref{page:cas1xx}). 
The only difference from Case~1 is that the path \(P\) is longer and covers the black copies of the differences in the set \(A\). 
As in Case~1, we can verify that
$$\big\{\rho_{\bullet}^i(C), \rho_{\bullet}^i(C')\mid i=0,1,\ldots,n-2\big\}\cup\{C_0, \rho_{\bullet}(C_0), \ldots, \rho_{\bullet}^{\ell-1}(C_0) \}$$
 is an HOP $(C_m)$-decomposition  for $4G_2^{\bullet}$, where \( G_2 = \text{Circ}(n-1; \pm(A\cup(\mathscrsfs{L}-\{\infty\}))) \bowtie K_1.\) 

We have $4K_n^{\bullet}=4G_1^{\bullet}\oplus 4G_2^{\bullet}$, and since each of $4G_1^{\bullet}$ and $4G_2^{\bullet}$ admits an HOP $(C_m)$-decomposition,  so does $4K_n^{\bullet}$ by Lemma \ref{Gtool3}. 
\end{proof}
%
%
\section{Proof of the main result}{\label{sec:8}}

For the reader's convenience, we restate our main result here and summerize its proof.

\mainresultOne*
\begin{proof}
By Theorem~\ref{theo:Gtool1}, this is equivalent to proving that $4K_n^{\bullet}$ admits an $\mathrm{HOP}(C_m)$-decomposition if and only if $2n(n-1)\equiv 0\ ({\rm mod}\ m)$. 
It is clear that if $4K_n^{\bullet}$ has an $\mathrm{HOP}(C_m)$-decomposition, then  
$2n(n-1) \equiv 0 \pmod{m}$. Conversely, assume $2n(n-1) \equiv 0 \pmod{m}$.  We  show $4K_n^{\bullet}$ admits an $\mathrm{HOP}(C_m)$-decomposition. 

First, let $m = 2$. Then it follows from Lemma~\ref{G4C} that there exists an  
$\mathrm{HOP}(C_m)$-decomposition of $4K_n^{\bullet}$.  
From now on, let $m \geq 3$. Then, there are three cases to consider.

\begin{description}

\item[Case 1:]  $m \mid \frac{n(n-1)}{2}$. 
First, assume \( n \) is odd. By Theorem~\ref{theo:Cm-dec}, there exists a \((C_m)\)-decomposition of \( K_n \), and hence by Lemma~\ref{lem:Gtool4}, there exists an HOP \((C_m)\)-decomposition of \( 4K_n^{\bullet} \).
Second, assume \( n \) is even. The following subcases arise:
\begin{itemize}
    \item \( m \) is even.
 By Theorem~\ref{theo:(Cm)-dec-2Kn}, \( 2K_n \) admits a \((C_m)\)-decomposition. Consequently, \( 2K_n^\circ \)  admits a \((C_m)\)-decomposition. If \( m \equiv 0 \pmod{4} \), then by Corollary~\ref{cor:color-decGraph}, we can recolor each \( m \)-cycle in this decomposition so that it contains an even number of pink edges, and hence by  Lemma~\ref{lem:Gtool2}, there exists an HOP \((C_m)\)-decomposition of \( 4K_n^{\bullet} \).
If \( m \equiv 2 \pmod{4} \), then since \( m \mid \frac{n(n-1)}{2} \), it follows from Lemma~\ref{lem:color-m-even-n-even}  that \( 4K_n^{\bullet} \) admits an HOP \((C_m)\)-decomposition.
   
    \item \( m \) is odd. By Lemma~\ref{lem:color-m-odd-n-even}, there exists an HOP \((C_m)\)-decomposition of \( 4K_n^{\bullet} \).
\end{itemize}

\item[Case 2:]  $m \nmid \frac{n(n-1)}{2}$ but $m \mid n(n-1)$. 
This implies that \(m\) is even and \(n(n-1) \equiv m \pmod{2m}\).  
If \(m \equiv 0 \pmod{4}\), then, using Corollary~\ref{cor:color-decGraph} and Lemma~\ref{lem:Gtool2}, there exists an HOP \((C_m)\)-decomposition of \( 4K_n^{\bullet} \).
If \(m \equiv 2 \pmod{4}\), then for odd $n$, the results follows from Lemma~\ref{lem:color-m-even-n-odd}, and for even $n$ from Lemma~\ref{lem:color-m-even-n-even}.

\item[Case 3:]  $m \nmid n(n-1)$ but $m \mid 2n(n-1)$. 
This implies that $m \equiv 0 \pmod{4}$ and  $2n(n-1)\equiv m\ ({\rm mod}\ 2m)$. If $n$ is odd, the results follows by Lemma~\ref{lem:n-odd-4K-last}, and if $n$ is even by Lemma~\ref{lem:n-even-4K-last}. 

\end{description}
\end{proof}

\section{Acknowledgments}

The author would like to thank her PhD supervisor, Dr. Mateja \v{S}ajna, for her invaluable guidance and support during this research.


\section{Data availability statement}
The author has nothing to report.


\end{document}